\documentclass[12pt,a4paper]{article}

\usepackage{latexsym}
\usepackage{amsmath}
\usepackage{amssymb}
\usepackage{amsthm}
\usepackage{amscd}
\usepackage{mathrsfs}
\usepackage[all]{xy}
\usepackage{graphicx}
\usepackage{bm}
\usepackage{comment}

\newtheorem{thm}{Theorem}[section]
\newtheorem{prop}[thm]{Proposition}
\newtheorem{defn}[thm]{Definition}
\newtheorem{lem}[thm]{Lemma}
\newtheorem{cor}[thm]{Corollary}

\newtheorem{rem}[thm]{Remark}

\newtheorem*{thmn}{Theorem}

\theoremstyle{remark}

\makeatletter

\@addtoreset{equation}{section}
\makeatother

\DeclareMathOperator{\Spec}{Spec}
\DeclareMathOperator{\Fr}{Fr}
\DeclareMathOperator{\Gal}{Gal}

\DeclareMathOperator{\tr}{tr}
\DeclareMathOperator{\ch}{ch}

\DeclareMathOperator{\Hom}{Hom}

\DeclareMathOperator{\Ker}{Ker}

\DeclareMathOperator{\Aut}{Aut} 
 
\DeclareMathOperator{\Ind}{Ind} 
\DeclareMathOperator{\Tr}{Tr}
\DeclareMathOperator{\Nr}{Nr}

\DeclareMathOperator{\Art}{Art}

\newcommand{\bom}[1]{\mbox{\boldmath $#1$}}

\newcommand{\bA}{\mathbb{A}}
\newcommand{\bC}{\mathbb{C}}

\newcommand{\bF}{\mathbb{F}}

\newcommand{\bP}{\mathbb{P}}
\newcommand{\bQ}{\mathbb{Q}}
\newcommand{\bR}{\mathbb{R}}

\newcommand{\bZ}{\mathbb{Z}}

\newcommand{\bfy}{\mathbf{y}}

\newcommand{\cF}{\mathcal{F}}

\newcommand{\cL}{\mathcal{L}}

\newcommand{\cO}{\mathcal{O}}

\newcommand{\fp}{\mathfrak{p}}

\newcommand{\iGL}{\mathit{GL}}

\newcommand{\rmab}{\mathrm{ab}}

\newcommand{\rmac}{\mathrm{ac}}

\newcommand{\ol}{\overline}

\newcommand{\cf}{\textit{cf.\ }}

\begin{document}
\title{Local Galois representations of Swan conductor one}
\author{Naoki Imai and Takahiro Tsushima}
\date{}
\maketitle


\begin{abstract} 
We construct the local Galois representations 
over the complex field 
whose Swan conductors are one by 
using etale cohomology of Artin--Schreier sheaves on 
affine lines over finite fields. 
Then, we study the Galois representations, and 
give an explicit description of the local Langlands correspondences for 
simple supercuspidal representations. 
We discuss also a more natural realization of the Galois representations 
in the etale cohomology of Artin--Schreier varieties. 
\end{abstract}

\section*{Introduction}
Let $K$ be a non-archimedean local field. 
Let $n$ be a positive integer. 
The existence of the local Langlands correspondence 
for $\iGL_n (K)$, proved in \cite{LRSellL} and \cite{HTsimSh}, 
is one of the fundamental results in the Langlands program. 
However, even in this fundamental case, 
an explicit construction of the local Langlands correspondence 
has not yet been obtained. 
One of the most striking results in this direction is 
Bushnell--Henniart's result for essentially tame representations 
in \cite{BHestI}, \cite{BHestII} and \cite{BHestIII}. 
On the other hand, 
we don't know much about the explicit construction 
outside essentially tame representations. 

We discuss this problem for 
representations of Swan conductor $1$. 
The irreducible supercuspidal representations of 
$\iGL_n (K)$ of Swan conductor $1$ 
are equivalent to 
the simple supercuspidal representations in 
the sense of \cite{ALssrGL} (\cf \cite{GRAinv}, \cite{RYEinv}). 
Such representations are called ``epipelagic'' in 
\cite{BHLepi}.

Let $p$ be the characteristic of the residue field $k$ of $K$. 
If $n$ is prime to $p$, 
the simple supercuspidal representations of 
$\iGL_n (K)$ 
are essentially tame. 
Hence, this case is covered by Bushnell--Henniart's work. 
See also \cite{ALssrGL}. 
It is discussed in \cite{KalepiL} 
to generalize the construction of the local Langlands correspondence 
for essentially tame epipelagic representations 
to other reductive groups. 

In this paper, we consider the case where 
$p$ divides $n$. 
In this case, 
the simple supercuspidal representations of 
$\iGL_n (K)$ 
are not essentially tame. 
Moreover, 
if $n$ is a power of $p$, 
the irreducible representations of the Weil group $W_K$ 
of Swan conductor $1$, 
which correspond to the simple supercuspidal representations 
via the local Langlands correspondence, 
cannot be induced from any proper subgroup. 
Such representations are called primitive (cf.~\cite{KocCprim}). 
For simple supercuspidal representations, 
we have a straightforward characterization of the local Langlands correspondence 
given in \cite{BHLepi}. 
Further, 
Bushnell--Henniart study 
the restriction to the wild inertia subgroup of 
the Langlands parameters 
for the simple supercuspidal representations explicitly. 
Actually, the restriction to the wild inertia subgroup 
already determines the original Langlands parameters 
up to character twists, 
but we need additional data, which appear 
in Bushnell--Henniart's characterization, to pin down 
the correct Langlands parameters. 
On the other hand, the construction of 
the irreducible representations of $W_K$ of Swan conductor $1$ 
is a non-trivial problem. 
What we will do in this paper is 
\begin{itemize}
\item 
to construct the irreducible representations of $W_K$ of Swan conductor $1$ 
without appealing to the existence of the local Langlands correspondence, and 
\item 
to give a description of the Langlands parameters themselves 
for the simple supercuspidal representations. 
\end{itemize}
Let $\ell$ be a prime number different from $p$. 
For the construction of 
the irreducible representations of $W_K$ of Swan conductor $1$, 
we use etale cohomology of 
an Artin--Schreier $\ell$-adic sheaf on $\bA_{k^{\rmac}}^1$, 
where $k^{\rmac}$ is an algebraic closure of $k$. 
It will be possible to avoid usage of geometry in the construction of 
the irreducible representations of $W_K$ of Swan conductor $1$. 
However, we prefer this approach, because 
\begin{itemize}
\item 
we can use geometric tools such as 
the Lefschetz trace formula and 
the product formula of Deligne--Laumon 
to study the constructed representations, and 
\item 
the construction works also for $\ell$-adic integral coefficients and mod $\ell$ coefficients. 
\end{itemize}
A description of the local Langlands correspondence 
for the simple supercuspidal representations is discussed in 
\cite{ITreal3} in the special case where $n=p=2$. 
Even in the special case, 
our method in this paper 
is totally different from that in \cite{ITreal3}. 

We explain the main result. 
We write $n=p^e n'$, where $n'$ is prime to $p$. 
We fix a uniformizer $\varpi$ of $K$ 
and an isomorphism 
$\iota \colon \ol{\bQ}_{\ell} \simeq \bC$. 

Let $\cL_{\psi}$ be the 
Artin--Schreier $\ol{\bQ}_{\ell}$-sheaf on 
$\bA_{k^{\rmac}}^1$ associated to 
a non-trivial character $\psi$ of $\bF_p$. 
Let $\pi \colon \bA_{k^{\rmac}}^1 \to \bA_{k^{\rmac}}^1$ 
be the morphism defined by 
$\pi (y)=y^{p^e +1}$. 
Let $\zeta \in \mu_{q-1}(K)$, where $q=\lvert k \rvert$. 
We put 
$E_{\zeta}=K[X]/(X^{n'} -\zeta \varpi)$. 
Then we can define a natural action of $W_{E_{\zeta}}$ on 
$H_{\mathrm{c}}^1 (\bA_{k^{\rmac}}^1 , \pi^*\mathcal{L}_{\psi})$. 
Using this action, 
we can associate a primitive representation 
$\tau_{n,\zeta,\chi,c}$ of $W_{E_{\zeta}}$ to 
$\zeta \in \mu_{q-1}(K)$, a character $\chi$ of $k^{\times}$ 
and $c \in \bC^{\times}$. 
We construct an irreducible representation 
$\tau_{\zeta,\chi,c}$ of Swan conductor $1$ as the induction of 
$\tau_{n,\zeta,\chi,c}$ to $W_K$. 

We can associate a simple supercuspidal representation 
$\pi_{\zeta,\chi,c}$ of 
$\iGL_n (K)$ to the same triple $(\zeta,\chi,c)$ 
by type theory. 
Any simple supercuspidal representation 
can be written in this form uniquely (\cf \cite[Proposition 1.3]{ITsimpJL}).

\begin{thmn}
The representations 
$\tau_{\zeta,\chi,c}$ and 
$\pi_{\zeta,\chi,c}$ correspond via the local Langlands correspondence. 
\end{thmn}

In Section \ref{Hgrp}, 
we recall a general fact on 
representations of 
a semi-direct product of a Heisenberg group with a cyclic group. 
In Section \ref{Galrep}, 
we give a construction of the irreducible representations 
of $W_K$ of Swan conductor $1$. 
To construct a representation of $W_K$ which 
naturally fits a description of the local Langlands correspondence, 
we need a subtle character twist. 
Such a twist appears also in the essentially tame case in \cite{BHestIII}, 
where it is called a rectifier. 
Our twist can be considered as an analogue of the rectifier. 
We construct the representations of $W_K$ using geometry, 
but we give also a 
representation theoretic characterization of 
the constructed representations. 
In Section \ref{GLrep}, 
we give a construction of the simple supercuspidal representations 
of $\iGL_n (K)$ using the type theory. 

In Section \ref{LLCstate}, 
we state the main theorem and recall a characterization of 
the local Langlands correspondence 
for simple supercuspidal representations given in \cite{BHLepi}. 
The characterization consists of the three equalities of 
(i) the determinant and the central character, 
(ii) the refined Swan conductors, and 
(iii) the epsilon factors.

In Section \ref{sec:gene}, 
we recall some general facts on epsilon factors. 
In Section \ref{SWdisc}, 
we recall facts on 
Stiefel--Whitney classes, 
multiplicative discriminants and 
additive discriminants. 
We use these facts to 
calculate Langlands constants of 
wildly ramified extensions. 
In Section \ref{DLprod}, 
we recall the product formula of Deligne--Laumon. 
In Section \ref{detsect}, 
we show the equality of 
the determinant and the central character 
using the product formula of Deligne--Laumon. 

In Section \ref{Indchar}, 
we construct a field extension 
$T_{\zeta}^{\mathrm{u}}$ of $E_{\zeta}$ 
such that the restriction of 
$\tau_{n,\zeta,\chi,c}$ to $W_{T_{\zeta}^{\mathrm{u}}}$ 
is an induction of a character and 
$p \nmid [T_{\zeta}^{\mathrm{u}} : E_{\zeta}]$, 
which we call an imprimitive field. 
In Section \ref{RefSwan}, 
we show the equality of the refined Swan conductors. 
We see also that 
the constructed representations of $W_K$ are 
actually of Swan conductor $1$. 

In Section \ref{Epsfac}, 
we show the equality of the epsilon factors. 
It is difficult to calculate the epsilon factors
of irreducible representations of $W_K$ of Swan conductor $1$ 
directly, because primitive representations are involved. 
However, 
we know the equality of the epsilon factors 
up to $p^e$-th roots of unity if $n=p^e$, 
since we have already checked the conditions (i) and (ii) 
in the characterization. 
Using this fact and 
$p \nmid [T_{\zeta}^{\mathrm{u}} : E_{\zeta}]$, 
the problem is reduced to study an epsilon factor of a character. 
Next we reduce the problem to the case where 
the characteristic of $K$ is $p$ and $k=\bF_p$. 
At this stage, it is possible to calculate the epsilon factor 
if $p \neq 2$. 
However, it is still difficult if $p=2$, 
because the direct calculation of the epsilon factor 
involves an explicit study of 
the Artin reciprocity map for a wildly ramified extension with 
a non-trivial ramification filtration. 
This is a special phenomenon 
in the case where $p=2$. 
We will avoid this difficulty by reducing the problem to the 
case where $e=1$. 
In this case, we have already known the equality up to sign. 
Hence, it suffices to show the equality of non-zero real parts. 
This is easy, because the 
difficult study of the Artin reciprocity map involves only 
the imaginary part of the equality. 

In Appendix \ref{RealAS}, we discuss a construction of 
irreducible representations of $W_K$ of Swan conductor $1$ 
in the cohomology of Artin--Schreier varieties. 
This geometric construction incorporates 
a twist by a ``rectifier''. 
We see that 
the ``rectifier'' parts come from the cohomology of 
Artin--Schreier varieties associated to 
quadratic forms. 
The Artin--Schreier varieties which we use 
have origins in studies of Lubin--Tate spaces in 
\cite{ITStab3} and \cite{ITsimpwild}. 

\subsection*{Acknowledgements}
The authors would like to thank the handling editor Don Blasius and referees for helpful comments and suggestions. 
This work was supported by JSPS KAKENHI Grant Numbers 
26707003, 20K03529, 21H00973.

\subsection*{Notation}
For a finite abelian group $A$, 
let $A^{\vee}$ denote the character group 
$\mathrm{Hom}_{\mathbb{Z}}(A,\bC^{\times})$. 
For a non-archimedean local field $K$, 
let 
\begin{itemize}
\item $\cO_K$ denote the ring of integers of $K$, 
\item $\fp_K$ denote the maximal ideal of $\cO_K$, 
\item $v_K$ denote the normalized valuation of $K$ 
which sends a uniformizer of $K$ to $1$, 
\item $\ch K$ denote the characteristic of $K$, 
\item $G_K$ denote the absolute Galois group of $K$, 
\item $W_K$ denote the Weil group of $K$, 
\item $I_K$ denote the inertia subgroup of $W_K$, 
\item $P_K$ denote the wild inertia subgroup of $W_K$, 
\end{itemize}
and we put $U_K^m =1 +\fp_K^m$ for any positive integer $m$.

\section{Representations of finite groups}\label{Hgrp}
First, we recall a fact on representations of Heisenberg groups. 
Let $G$ be a finite group with center $Z$. 
We assume the following:
\begin{enumerate}
 \item 
 The group $G/Z$ is an elementary abelian $p$-group. 
 \item 
 For any $g \in G \setminus Z$, 
 the map $c_g \colon G \to Z;\ g' \mapsto [g ,g']$ 
 is surjective. 
\end{enumerate} 
\begin{rem}
The map $c_g$ in (ii) is a group homomorphism. 
Hence, $Z$ is automatically an elementary abelian $p$-group. 
\end{rem}

Let $\psi \in Z^{\vee}$ be a non-trivial character. 
\begin{prop}\label{prop:rhopsi}
There is a unique irreducible representation 
$\rho_{\psi}$ of $G$ such that 
$\rho_{\psi} |_Z$ contains $\psi$. 
Moreover, we have 
$(\dim \rho_{\psi})^2 = [G:Z]$ 
and 
we can construct $\rho_{\psi}$ as follow: 
Take an abelian subgroup $G_1$ of $G$ such that 
$Z \subset G_1$ and 
$2 \dim_{\bF_p} (G_1/Z) = \dim_{\bF_p} (G/Z)$. 
Extend $\psi$ to a character $\psi_1$ of $G_1$. 
Then $\rho_{\psi}=\Ind_{G_1}^G \psi_1$. 
\end{prop}
\begin{proof}
The claims other than the construction of 
$\rho_{\psi}$ is \cite[(8.3.3) Proposition]{BFGdiv}. 
Note that if 
an abelian subgroup $G_1$ of $G$ satisfies 
the conditions in the claim, 
then $G_1/Z$ is a maximal totally isotropic subspace of 
$G/Z$ under the pairing 
\[
 (G/Z) \times (G/Z) \longrightarrow \bC^{\times} ;\ 
 (gZ,g'Z) \mapsto \psi ([g,g']) . 
\]
Hence the construction follows from 
the proof of \cite[(8.3.3) Proposition]{BFGdiv}. 
\end{proof}

Next, we consider 
representations of 
a semi-direct product of a Heisenberg group with a cyclic group. 
Let $A \subset \Aut (G)$ be a cyclic subgroup of 
order $p^e +1$ where $e =(\log_p [G:Z])/2$. 
We assume the following:
\begin{enumerate}
 \setcounter{enumi}{2}
 \item
 The group $A$ acts on $Z$ trivially. 
 \item 
 For any non-trivial element $a \in A$, 
 the action of $a$ on $G/Z$ fixes only the unit element. 
\end{enumerate} 
We consider the semi-direct product $A \ltimes G$ 
by the action of $A$ on $G$. 

\begin{lem}\label{chara}
There is a unique irreducible representation 
$\rho_{\psi}'$ of $A \ltimes G$ such that 
$\rho_{\psi}'|_G \simeq \rho_{\psi}$ and 
$\tr \rho_{\psi}' (a) =-1$ for every non-trivial element $a \in A$. 
\end{lem}
\begin{proof}
The claim is proved in the proof of \cite[22.2 Lemma]{BHLLCGL2} 
if $Z$ is cyclic and $\psi$ is a faithful character. 
In fact, the same proof works also in our case. 
\end{proof}

\begin{cor}\label{cor:exirr}
There exists a unique representation 
$\rho_{\psi}'$ of $A \ltimes G$ such that 
$\rho_{\psi}'|_Z \simeq \psi^{\oplus p^e}$ and 
$\tr \rho_{\psi}' (a) =-1$ for every non-trivial element $a \in A$. 
Further, the representation $\rho_{\psi}'|_G$ 
is irreducible. 
\end{cor}
\begin{proof}
First we show the existence. 
We take the representation $\rho_{\psi}'$ in 
Lemma \ref{chara}. 
Then $\rho_{\psi}'$ has a central character and 
the central character is equal to $\psi$ by Proposition 
\ref{prop:rhopsi}. 
This shows the existence. 

We show the uniqueness and the irreducibility of $\rho_{\psi}'|_G$. 
Assume that $\rho_{\psi}'$ satisfies the condition 
in the claim. 
Take an irreducible subrepresentation 
$\rho_{\psi}$ of $\rho_{\psi}'|_G$. 
Then $\rho_{\psi}$ satisfies the condition of 
Proposition \ref{prop:rhopsi}. 
Hence, $\dim \rho_{\psi}=p^e$. 
Then we have $\rho_{\psi} = \rho_{\psi}'|_G$ 
and $\rho_{\psi}'|_G$ is irreducible. 
Such $\rho_{\psi}$ is unique by 
Lemma \ref{chara}. 
\end{proof}

\section{Galois representations}\label{Galrep}
\subsection{Swan conductor}\label{sseq:Swc}
Let $K$ be a non-archimedean local field 
with residue field $k$. 
Let $p$ be the characteristic of $k$. 
Let $f$ be the extension degree of $k$ over $\bF_p$. 
We put $q=p^f$. 

Let 
\[
\mathrm{Art}_K \colon K^{\times} 
 \stackrel{\sim}{\longrightarrow} W_K^{\mathrm{ab}}
\]
be the Artin reciprocity map, which sends 
a uniformizer to a lift of the geometric Frobenius element. 

Let $\tau$ be a finite dimensional irreducible 
continuous representation of $W_K$ over $\bC$. 
Let $\Psi \colon K \to \bC^{\times}$ 
be a non-trivial additive character. 
Let 
$\varepsilon( \tau, s, \Psi)$ denote the 
Deligne--Langlands local constant of 
$\tau$ with respect to $\Psi$. 
We simply write 
$\varepsilon( \tau, \Psi)$ for $\varepsilon( \tau, 1/2, \Psi)$. 

We define an unramified character 
$\omega_s \colon K^{\times} \to \bC^{\times}$ 
by 
$\omega_s (\varpi)=q^{-s}$ 
for $s \in \bR$, where $\varpi$ is a uniformizer of $K$. 
We recall that 
\begin{equation}\label{twun}
\varepsilon(\tau,s,\Psi)=\varepsilon(\tau \otimes \omega_s,0,\Psi) 
\end{equation}
(\cf \cite[(3.6.4)]{TatNtb}). 

We define $\psi_0 \in \mathbb{F}_p^{\vee}$ by 
$\psi_0 (1) =e^{2 \pi \sqrt{-1}/p}$. 
We take an additive character 
$\psi_K \colon K \to \bC^{\times}$ 
such that 
$\psi_K (x) =\psi_0 (\Tr_{k/\bF_p} (\bar{x} ))$ 
for $x \in \mathcal{O}_K$. 
By \cite[29.4 Proposition]{BHLLCGL2}, 
there exists an integer $\mathrm{sw}(\tau)$ such that 
\[
 \varepsilon(\tau,s,\psi_K) 
 =q^{-\mathrm{sw}(\tau) s} 
 \varepsilon(\tau, 0, \psi_K). 
\]
We put $\mathrm{Sw}(\tau) =\max \{ \mathrm{sw}(\tau),0 \}$, 
which we call the Swan conductor of $\tau$.


\subsection{Construction}\label{constGal}
In this subsection, we construct a group $Q$ which acts on a curve $C$ over an algebraic closure of $k$. 
By using this action of $Q$ and Frobenius action, 
we construct a representation of 
a semi-direct product $Q \rtimes \bZ$ 
in etale cohomology of $C$. 
Then we use the representation of $Q \rtimes \bZ$ 
to construct a representation of a Weil group. 

We fix an algebraic closure $K^{\mathrm{ac}}$ of $K$. 
Let $k^{\mathrm{ac}}$ 
be the residue field of $K^{\rmac}$. 
Let $n$ be a positive integer. 
We write $n=p^e n'$ with $(p,n')=1$. 
Throughout this paper, 
we assume that $e \geq 1$. 
Let 
\[
 Q= \Bigl\{ 
 (a,b,c)\ \Big|\ 
 a \in 
 \mu_{p^e+1}(k^{\mathrm{ac}}),\ 
 b, c \in k^{\mathrm{ac}},\ \ 
 b^{p^{2e}}+b=0,\ \  
 c^{p}- 
 c +b^{p^e+1} =0 
 \Bigr\} 
\]
be the group whose multiplication is given by 
\[
 (a_1 ,b_1 ,c_1 ) \cdot 
 (a_2 ,b_2 ,c_2 )= 
 \biggl( a_1 a_2 ,b_1 + a_1 b_2 , 
 c_1 + c_2 +\sum_{i=0}^{e -1} 
 \bigl( a_1 b_1^{p^e} b_2 
 \bigr)^{p^{i}} \biggr). 
\]

\begin{rem}
The construction of the group $Q$ has its origin in a study of the 
automorphism of a curve $C$ defined below. 
We can check that 
the above multiplication gives a group structure on $Q$ directly, 
but it's also possible to show this by 
checking that the inclusion from $Q$ to the automorphism group of $C$ 
defined below 
is compatible with the multiplications. 
\end{rem}

Note that 
$|Q|=p^{2e+1}(p^e+1)$. 
Let $Q \rtimes \mathbb{Z}$ be a semidirect product, where  
$m \in \mathbb{Z}$ 
acts  on $Q$ by 
$(a,b,c) \mapsto (a^{p^{-m}},b^{p^{-m}},c^{p^{-m}})$. 
We put 
\begin{equation}\label{Frmdef}
 \Fr (m) = ((1,0,0),m ) \in Q \rtimes \mathbb{Z} 
\end{equation}
for $m \in \mathbb{Z}$. 

Let $C$ be the smooth affine curve over $k^{\mathrm{ac}}$ 
defined by 
\begin{equation*}\label{Cdef}
 x^{p}-x=y^{p^e+1} 
 \ \ \textrm{in}\ \mathbb{A}_{k^{\mathrm{ac}}}^2 . 
\end{equation*}
We define a right action of 
$Q \rtimes \mathbb{Z}$ on $C$ by 
\begin{align*}
 (x,y)((a,b,c),0) &= 
 \biggl( x + \sum_{i=0}^{e -1} 
 ( b y )^{p^i} + c , 
 a (y + b^{p^e} ) \biggr), \\ 
 (x,y )\Fr (1) &= 
 (x^{p},y^{p} ). 
\end{align*}
We consider the morphisms 
\begin{align*}
 h \colon \mathbb{A}_{k^{\mathrm{ac}}}^1 &
 \to \mathbb{A}_{k^{\mathrm{ac}}}^1 ; \ x \mapsto x^p -x, \\ 
 \pi \colon \mathbb{A}_{k^{\mathrm{ac}}}^1 & 
 \to \mathbb{A}_{k^{\mathrm{ac}}}^1 ; \ y 
 \mapsto y^{p^e +1}. 
\end{align*}
Then we have the fiber product 
\[
 \xymatrix{
 C
 \ar@{->}[r]^{h'} \ar@{->}[d]_{\pi'} \ar@{}[rd]|\square & 
 \mathbb{A}_{k^{\mathrm{ac}}}^1 
 \ar@{->}[d]^{\pi} \\ 
 \mathbb{A}_{k^{\mathrm{ac}}}^1 
 \ar@{->}[r]_{h} & 
 \mathbb{A}_{k^{\mathrm{ac}}}^1 
 }
\]
where $\pi'$ and $h'$ are the natural projections to 
the first and second coordinates respectively. 
Let $g =((a,b,c),m) \in Q \rtimes \mathbb{Z}$. 
We consider the morphism 
\[
 g_0 \colon \mathbb{A}_{k^{\mathrm{ac}}}^1 
 \to \mathbb{A}_{k^{\mathrm{ac}}}^1 ;\ y \mapsto 
 \bigl( a ( y +b^{p^e} ) \bigr)^{p^m}. 
\] 
Let $\ell$ be a prime number different from $p$. 
Then we have a natural isomorphism 
\[
 c_g \colon g_0^* h'_* \overline{\mathbb{Q}}_{\ell} 
 \xrightarrow{\sim} 
 h'_* g^* \overline{\mathbb{Q}}_{\ell} \xrightarrow{\sim} 
 h'_* \overline{\mathbb{Q}}_{\ell}. 
\]
We take an isomorphism 
$\iota \colon \overline{\mathbb{Q}}_{\ell} \simeq \mathbb{C}$. 
We sometimes view a character over $\bC$ 
as a character over $\overline{\mathbb{Q}}_{\ell}$ 
by $\iota$. 
Let $\psi \in \bF_p^{\vee}$. 
We write $\mathcal{L}_{\psi}$ for the Artin--Schreier 
$\overline{\mathbb{Q}}_{\ell}$-sheaf on 
$\mathbb{A}_{k^{\mathrm{ac}}}^1$
associated to $\psi$, 
which is equal to 
$\mathfrak{F} (\psi)$ in the notation of 
\cite[Sommes trig. 1.8 (i)]{DelCoet}. 
Then we have a decomposition 
$h_* \overline{\mathbb{Q}}_{\ell} = 
 \bigoplus_{\psi \in \bF_p^{\vee}} \mathcal{L}_{\psi}$. 
This decomposition gives canonical isomorphisms 
\begin{equation}\label{candec}
 h'_* \overline{\mathbb{Q}}_{\ell} \cong 
 \pi^* h_* \overline{\mathbb{Q}}_{\ell} \cong 
 \bigoplus_{\psi \in \bF_p^{\vee}} \pi^* \mathcal{L}_{\psi}. 
\end{equation}
The isomorphisms $c_g$ and \eqref{candec} induce 
$c_{g,\psi} \colon g_0^* \pi^* \mathcal{L}_{\psi} 
 \to \pi^* \mathcal{L}_{\psi}$. 
We define a left action of $Q \rtimes \mathbb{Z}$ on 
$H_{\mathrm{c}}^1 (\mathbb{A}_{k^{\mathrm{ac}}}^1 ,
 \pi^*\mathcal{L}_{\psi})$ 
by 
\[
 H_{\mathrm{c}}^1 (\mathbb{A}_{k^{\mathrm{ac}}}^1 ,\pi^* \mathcal{L}_{\psi}) 
 \xrightarrow{g_0^*} 
 H_{\mathrm{c}}^1 (\mathbb{A}_{k^{\mathrm{ac}}}^1 ,
 g_0^* \pi^* \mathcal{L}_{\psi}) 
 \xrightarrow{c_{g ,\psi}} 
 H_{\mathrm{c}}^1 (\mathbb{A}_{k^{\mathrm{ac}}}^1 
 ,\pi^* \mathcal{L}_{\psi}) 
\]
for $g \in Q \rtimes \mathbb{Z}$. 
Let $\tau_{\psi}$ be the representation of $Q \rtimes \mathbb{Z}$ 
over $\bC$ defined by 
$H_{\mathrm{c}}^1 (\mathbb{A}_{k^{\mathrm{ac}}}^1 ,
 \pi^* \mathcal{L}_{\psi})$ and $\iota$. 
For $\theta \in \mu_{p^e +1}(k^{\mathrm{ac}})^{\vee}$, 
let $\mathcal{K}_{\theta}$ be the 
smooth Kummer $\overline{\mathbb{Q}}_{\ell}$-sheaf on 
$\mathbb{G}_{\mathrm{m},k^{\mathrm{ac}}}$ 
associated to $\theta$. 
We view 
$\mu_{p^e +1}(k^{\rmac}) \times \bF_p$ 
as a subgroup of $Q$ by 
$(a,c) \mapsto (a,0,c)$. 

\begin{lem}\label{HLdec}
We have a natural isomorphism 
\begin{equation*}
 H_{\mathrm{c}}^1 (\mathbb{A}_{k^{\mathrm{ac}}}^1 ,
 \pi^* \mathcal{L}_{\psi}) 
 \simeq 
 \bigoplus_{\theta \in \mu_{p^e +1}(k^{\mathrm{ac}})^{\vee} 
 \setminus \{ 1 \} } 
 H_{\mathrm{c}}^1 (\mathbb{G}_{\mathrm{m}, k^{\mathrm{ac}}} ,
 \mathcal{L}_{\psi} \otimes \mathcal{K}_{\theta} ), 
\end{equation*}
which is compatible with the actions of 
$\mu_{p^e +1}(k^{\rmac}) \times \bF_p$ 
where $(a,c) \in \mu_{p^e +1}(k^{\rmac}) \times \bF_p$ acts on 
$H_{\mathrm{c}}^1 (\mathbb{G}_{\mathrm{m}, k^{\mathrm{ac}}} ,
 \mathcal{L}_{\psi} \otimes \mathcal{K}_{\theta} )$ 
by $\theta (a) \psi(c)$. 
Further, we have 
\[
 \dim H_{\mathrm{c}}^1 (\mathbb{G}_{\mathrm{m}, k^{\mathrm{ac}}} ,
 \mathcal{L}_{\psi} \otimes \mathcal{K}_{\theta} )=1 
\] 
for any 
$\theta \in \mu_{p^e +1}(k^{\mathrm{ac}})^{\vee} 
 \setminus \{ 1 \}$. 
\end{lem}
\begin{proof}
By the projection formula, 
we have natural isomorphisms 
\[
 \pi_\ast \pi^\ast \mathcal{L}_{\psi} \simeq 
 \pi_\ast (\pi^\ast \mathcal{L}_{\psi} 
 \otimes \overline{\mathbb{Q}}_{\ell} ) 
 \simeq 
 \mathcal{L}_{\psi} \otimes \pi_{\ast} \overline{\mathbb{Q}}_{\ell}
\]
on $\mathbb{A}^1_{k^{\mathrm{ac}}}$. 
Further, we have 
\[
 \pi_{\ast} \overline{\mathbb{Q}}_{\ell} \simeq 
 \bigoplus_{\theta \in \mu_{p^e+1}(k^{\mathrm{ac}})^{\vee}} 
 \mathcal{K}_{\theta} 
\]
on $\mathbb{G}_{\mathrm{m}, k^{\mathrm{ac}}}$, 
since $\pi$ is a finite etale $\mu_{p^e+1}(k^{\mathrm{ac}})$-covering 
over $\mathbb{G}_{\mathrm{m}, k^{\mathrm{ac}}}$. 
Therefore, we have 
\begin{equation}\label{ppLdec}
 \pi_\ast \pi^\ast \mathcal{L}_{\psi} \simeq 
 \mathcal{L}_{\psi} \otimes \pi_{\ast} \overline{\mathbb{Q}}_{\ell}
 \simeq 
 \bigoplus_{\theta \in \mu_{p^e+1}(k^{\mathrm{ac}})^{\vee}} 
 \mathcal{L}_{\psi} \otimes \mathcal{K}_{\theta}
\end{equation}
on $\mathbb{G}_{\mathrm{m}, k^{\mathrm{ac}}}$. 
Let $\{0\}$ denote the origin of $\bA^1_{k^{\rmac}}$. 
Let $i \colon \{0\} \to \bA^1_{k^{\rmac}}$ and 
$j \colon \mathbb{G}_{\mathrm{m}, k^{\mathrm{ac}}}
 \to \bA^1_{k^{\rmac}}$ be the natural immersions. 
From the exact sequence 
\[
 0 \to j_! j^* \pi^\ast \mathcal{L}_{\psi} 
 \to \pi^\ast \mathcal{L}_{\psi} \to 
 i_* i^* \pi^\ast \mathcal{L}_{\psi} \to 0, 
\]
we have the exact sequence 
\begin{equation}\label{pLex}
 0 \to H^0(\{0\},i^* \pi^\ast \mathcal{L}_{\psi} )
 \to H_{\mathrm{c}}^1
 (\mathbb{G}_{\mathrm{m},k^{\mathrm{ac}}},\pi^\ast \mathcal{L}_{\psi})
 \to H_{\mathrm{c}}^1
 (\mathbb{A}^1_{k^{\mathrm{ac}}},\pi^\ast \mathcal{L}_{\psi}) 
 \to 0, 
\end{equation}
since 
$H_{\mathrm{c}}^0
 (\mathbb{A}^1_{k^{\mathrm{ac}}},\pi^\ast \mathcal{L}_{\psi})=0$ and 
$H^1(\{0\},i^* \pi^\ast \mathcal{L}_{\psi})=0$. 
Note that $H^0(\{0\},i^* \pi^\ast \mathcal{L}_{\psi} ) \simeq \psi$. 
By \eqref{ppLdec}, we have isomorphisms 
\begin{equation}\label{GmpiL}
 H_{\mathrm{c}}^1
 (\mathbb{G}_{\mathrm{m},k^{\mathrm{ac}}},\pi^\ast \mathcal{L}_{\psi}) 
 \simeq 
 H_{\mathrm{c}}^1
 (\mathbb{G}_{\mathrm{m},k^{\mathrm{ac}}},\pi_* \pi^\ast \mathcal{L}_{\psi}) 
 \simeq 
 \bigoplus_{\theta \in \mu_{p^e+1}(k^{\mathrm{ac}})^{\vee}} 
  H_{\mathrm{c}}^1
 (\mathbb{G}_{\mathrm{m},k^{\mathrm{ac}}}, 
 \mathcal{L}_{\psi} \otimes \mathcal{K}_{\theta} ) . 
\end{equation}
We know that 
\begin{equation}\label{dimLK1}
 \dim H_{\mathrm{c}}^1
 (\mathbb{G}_{\mathrm{m},k^{\mathrm{ac}}}, 
 \mathcal{L}_{\psi} \otimes \mathcal{K}_{\theta} )=1 
\end{equation}
for any 
$\theta \in \mu_{p^e +1}(k^{\mathrm{ac}})^{\vee}$ by 
the proof of \cite[Lemma 7.1]{ITStab3} (\cf \cite[(2.3)]{ITGeomHW}). 
Since the composition of 
\[
 H^0(\{0\},i^* \pi^\ast \mathcal{L}_{\psi})
 \to H_{\mathrm{c}}^1
 (\mathbb{G}_{\mathrm{m},k^{\mathrm{ac}}},\pi^\ast \mathcal{L}_{\psi}) 
\]
and \eqref{GmpiL} is compatible with the actions of 
$\mu_{p^e +1}(k^{\rmac}) \times \bF_p$, 
it factors through an isomorphism 
$H^0(\{0\},i^* \pi^\ast \mathcal{L}_{\psi}) \simeq 
 H_{\mathrm{c}}^1(\mathbb{G}_{\mathrm{m},k^{\mathrm{ac}}}, \mathcal{L}_{\psi})$ 
by \eqref{dimLK1}. 
Then the claim follows from 
\eqref{pLex}, \eqref{GmpiL} and \eqref{dimLK1}. 
\end{proof}

Let 
$\varrho \colon \mu_2(k) \hookrightarrow \bC^{\times}$ 
be the non-trivial group homomorphism, 
if $p \neq 2$. 
We define a character $\theta_0 \in \mu_{p^e+1}(k^{\mathrm{ac}})^{\vee}$ 
by 
\begin{equation}\label{theta_0}
 \theta_0 (a) = 
 \begin{cases}
  \varrho (a^{(p^e+1)/2} ) \quad & \textrm{if} \ p \neq 2, \\ 
  1 \quad & \textrm{if} \ p=2 
 \end{cases}
\end{equation}
for $a \in \mu_{p^e+1}(k^{\mathrm{ac}})$. 
For an integer $m$ and a positive odd integer $m'$, 
let $\bigl( \frac{m}{m'} \bigr)$ denote the 
Jacobi symbol. 
For an odd prime $p$, we set
\[
\epsilon(p)=\begin{cases}
1 & \textrm{if $p \equiv 1 \mod 4$}, \\
\sqrt[]{-1} & \textrm{if $p \equiv 3 \mod 4$}. 
\end{cases}
\]
We have $\epsilon(p)^2=\bigl(\frac{-1}{p}\bigr)$. 
We define a representation $\tau_n$ of $Q \rtimes \mathbb{Z}$ as 
the twist of $\tau_{\psi_0}$ by the character 
\begin{equation}\label{chartwi}
 Q \rtimes \mathbb{Z} \to \bC^{\times};\ 
 ((a,b,c),m) \mapsto 
 \begin{cases}
 \theta_0 (a)^{n} 
 \bigl( \bigl( 
 -\epsilon(p)
 \bigl( \frac{-2n'}{p} \bigr) 
 \bigr)^n 
 p^{-\frac{1}{2}} 
 \bigr)^m 
 \quad & \textrm{if} \ p \neq 2, \\ 
 \bigl( (-1)^{\frac{n(n-2)}{8}} p^{-\frac{1}{2}} \bigr)^m 
 \quad & \textrm{if} \ p=2. 
 \end{cases}
\end{equation}
The value of this character is related to a quadratic Gauss sum.
A geometric origin of this character is given in \eqref{geori}. 
Let 
$(\zeta,\chi,c) \in \mu_{q-1}(K) \times 
(k^{\times})^{\vee} \times \bC^{\times}$. 
We take a uniformizer $\varpi$ of $K$. 
We choose an element $\varphi_{\zeta}' \in K^{\mathrm{ac}}$
such that $\varphi_{\zeta}'^{n'}=\zeta \varpi$ and 
set $E_{\zeta}=K(\varphi_{\zeta}')$. 
We choose elements 
$\alpha_{\zeta}, \beta_{\zeta}, \gamma_{\zeta} \in K^{\mathrm{ac}}$ 
such that 
\begin{equation}\label{extgen}
 \alpha_{\zeta}^{p^e +1} = - \varphi_{\zeta}', \quad 
 \beta_{\zeta}^{p^{2e}} +\beta_{\zeta} 
 =- \alpha_{\zeta}^{-1} , \quad 
 \gamma_{\zeta}^p - \gamma_{\zeta} =\beta_{\zeta}^{p^e+1}. 
\end{equation}
For $\sigma \in W_{E_{\zeta}}$, 
we set 
\begin{gather}\label{abcdef}
\begin{aligned}
 a_{\sigma} =\sigma(\alpha_{\zeta})/(\alpha_{\zeta}), \quad
 b_{\sigma} =a_{\sigma} \sigma(\beta_{\zeta})-\beta_{\zeta}, \quad 
 c_{\sigma}  =\sigma(\gamma_{\zeta})-\gamma_{\zeta} + 
 \sum_{i=0}^{e-1}
 ( b_{\sigma}^{p^e} (\beta_{\zeta} + b_{\sigma} ) )^{p^i}. 
\end{aligned}
\end{gather}
Then we have 
$a_{\sigma}, b_{\sigma}, c_{\sigma} \in \cO_{K^{\rmac}}$. 
For $\sigma \in W_{E_{\zeta}}$, 
we put 
$n_{\sigma} =v_{E_{\zeta}} (\mathrm{Art}_{E_{\zeta}}^{-1}(\sigma))$. 
We have the homomorphism 
\begin{equation}\label{hom}
 \Theta_{\zeta} \colon 
 W_{E_{\zeta}} \longrightarrow Q \rtimes \mathbb{Z};\ 
 \sigma \mapsto \bigl(
 (\bar{a}_{\sigma},\bar{b}_{\sigma},\bar{c}_{\sigma}),
 fn_{\sigma} \bigr). 
\end{equation}

\begin{lem}
The image of the homomorphism $\Theta_{\zeta}$ 
is $Q \rtimes (f\mathbb{Z})$. 
\end{lem}
\begin{proof}
It suffices to show that 
the image of $I_{E_{\zeta}} \subset W_{E_{\zeta}}$ 
under $\Theta_{\zeta}$ is equal to $Q \subset Q \rtimes \bZ$, 
since the homomorphism 
$W_{E_{\zeta}} \to f\bZ;\ \sigma \mapsto fn_{\sigma}$ 
is surjective. 
We put $N_{\zeta}=E_{\zeta}(\alpha_{\zeta}, \beta_{\zeta}, \gamma_{\zeta})$. 
Then the kernel of $\Theta_{\zeta}$ is equal to 
$I_{N_{\zeta}}$ by the definition. 
Hence we have an injection 
$I_{E_{\zeta}}/I_{N_{\zeta}} \hookrightarrow Q$. 
This injection is actually a bijection, 
since $N_{\zeta}$ is a totally ramified extension over 
$E_{\zeta}$ of degree $p^{2e+1}(p^e+1)$, which equals to $|Q|$. 
Therefore, we obtain the claim. 
\end{proof}

We write 
$\tau_{n,\zeta}$ 
for the representation of $W_{E_{\zeta}}$ given by 
$\Theta_{\zeta}$ and $\tau_n$. 
Recall that $c$ is an element of $\bC^{\times}$. 
Let 
$\phi_c \colon W_{E_{\zeta}} \to \bC^{\times}$ 
be the character defined by 
$\phi_c(\sigma)=c^{n_{\sigma}}$.
We have the isomorphism 
$\varphi_{\zeta}'^{\bZ} \times \cO_{E_{\zeta}}^{\times} 
 \simeq {E_{\zeta}}^{\times}$ 
given by the multiplication. 
Let 
$\mathrm{Frob}_p \colon k^{\times} \to k^{\times}$ 
be the inverse of the $p$-th power map. 
We consider the following composition: 
\[
 \lambda_{\zeta} \colon 
 W_{E_{\zeta}}^{\mathrm{ab}} \simeq {E_{\zeta}}^{\times} \simeq 
 \varphi_{\zeta}'^{\bZ} \times \cO_{E_{\zeta}}^{\times} 
 \xrightarrow{\mathrm{pr_2}} 
 \cO_{E_{\zeta}}^{\times} 
 \xrightarrow{\mathrm{can.}} k^{\times} 
 \xrightarrow{\mathrm{Frob}_p^e} k^{\times}.
\]
We put 
\begin{equation}\label{Galconst}
 \tau_{n,\zeta,\chi,c} =\tau_{n,\zeta} \otimes 
 (\chi \circ \lambda_{\zeta} ) \otimes \phi_c 
 \quad \textrm{and} \quad 
 \tau_{\zeta,\chi,c}=
 \Ind_{E_{\zeta}/K} \tau_{n,\zeta,\chi,c}. 
\end{equation}
We will see that 
$\tau_{\zeta,\chi,c}$ is an irreducible representation of Swan conductor $1$  
in Proposition \ref{tause}. 
This Galois representation $\tau_{\zeta,\chi,c}$ is our main object in this paper. We will study several invariants associated to this, for example 
its determinant and epsilon factor.

\subsection{Characterization}
We put 
\[
 Q_0 =\bigl\{ (1,b ,c ) \in Q \bigl\}, \quad 
 F=\bigl\{ (1,0,c) \in Q \bigm| 
 c \in \bF_p \bigr\}. 
\]
We identify $\bF_p$ with $F$ by 
$c \mapsto (1,0,c)$. 
\begin{lem}\label{Q0Fsur}
For any $g=(1,b,c) \in Q_0$ 
with $b \neq 0$, 
the map $Q_0 \to F;\ g' \mapsto [g,g']$ is surjective. 
\end{lem}
\begin{proof}
For $(1,b_1,c_1), (1,b_2,c_2) \in Q_0$, 
we have 
\[
 [(1,b_1,c_1), (1,b_2,c_2)]= 
 \biggl(1,0, \sum_{i=0}^{e -1} 
 ( b_1^{p^e} b_2 -b_1 b_2^{p^e} )^{p^i} \biggr). 
\]
If $b_1 \neq 0$, then 
\[
 \bigl\{ b \in k^{\mathrm{ac}} \mid b^{p^{2e}} +b =0 \bigr\} 
 \to \mathbb{F}_{p^e};\ b_2 \to 
 b_1^{p^e} b_2 -b_1 b_2^{p^e} 
\]
is surjective. 
The claim follows from 
the surjectivity of 
$\Tr_{\mathbb{F}_{p^e}/\mathbb{F}_p}$. 
\end{proof}
By this lemma, 
we can apply the results from Section \ref{Hgrp} to our situation 
with $G=Q_0$, $Z=F$ and 
$A=\mu_{p^e+1}(k^{\mathrm{ac}})$, where 
the action of $\mu_{p^e+1}(k^{\mathrm{ac}})$ on $Q_0$ 
is given by the 
embedding 
\[
 \mu_{p^e+1}(k^{\mathrm{ac}}) \longrightarrow Q ;\ 
 a \mapsto (a,0,0) 
\]
and the conjugation. 
Let $\tau^0$ 
denote the unique representation of 
$Q$ characterized by 
\begin{equation}\label{cht} 
 \tau^0 |_{F} 
 \simeq \psi_0^{\oplus p^e},\ \ \ 
 \Tr \tau^0 ((a,0,0))= 
 - 1 
\end{equation}
for 
$a \in \mu_{p^e+1}(k^{\mathrm{ac}}) \setminus \{1\}$ 
(\cf Corollary \ref{cor:exirr}). 

We have a decomposition 
\begin{equation}\label{taudec}
 \tau^0 = 
 \bigoplus_{\theta \in \mu_{p^e+1}(k^{\mathrm{ac}})^{\vee} \setminus \{1\}} 
 L_{\theta} 
\end{equation}
such that 
$a \in \mu_{p^e+1}(k^{\mathrm{ac}})$ acts on $L_{\theta}$ 
by 
$\theta(a)$, 
since the both sides of \eqref{taudec} 
have the same character as representations of 
$\mu_{p^e+1}(k^{\mathrm{ac}})$. 
For a positive integer $m$ dividing $p^e+1$, 
we consider 
$\mu_{m}(k^{\mathrm{ac}})^{\vee}$ as a subset of 
$\mu_{p^e+1}(k^{\mathrm{ac}})^{\vee}$ by the dual of the surjection 
\[
 \mu_{p^e+1}(k^{\mathrm{ac}}) \to \mu_{m}(k^{\mathrm{ac}}) ; \ 
 x \to x^{\frac{p^e+1}{m}}. 
\]
We simply write $Q$ for the subgroup 
$Q \times \{0\} \subset Q \rtimes \mathbb{Z}$. 

\begin{lem}\label{restQ}
We have 
$\tau_{\psi_0} |_Q \simeq \tau^0$. 
\end{lem}
\begin{proof}
The representation $\tau_{\psi_0} |_Q$ satisfies 
the characterization \eqref{cht} 
by Lemma \ref{HLdec}. 
Hence $\tau_{\psi_0} |_Q$ is isomorphic to $\tau^0$. 
\end{proof}

\begin{cor}\label{cor:irrQ0}
The representation $\tau_{\psi_0} |_{Q_0}$ 
is irreducible. 
\end{cor}
\begin{proof}
This follows from 
Corollary \ref{cor:exirr}, 
\eqref{cht} and Lemma \ref{restQ}. 
\end{proof}
For any odd prime $p$, we have 
\begin{equation}\label{Gss}
\sum_{x \in \mathbb{F}_p^{\times}} \psi_0 (x^2)
=\sum_{x \in \mathbb{F}_p^{\times}}\left(\frac{x}{p}\right) \psi_0 (x)
=\epsilon(p)\sqrt{p}
\end{equation}
by Gauss. 
\begin{lem}\label{TrFr}
We have 
\begin{equation*}
 \Tr \tau_{\psi_0} \bigl( \Fr(1) \bigr) = 
 \begin{cases}
  -\epsilon(p)
  \sqrt{p} 
  \quad & \textrm{if}\ p \neq 2, \\ 
  0 
  \quad & \textrm{if}\ p=2. 
 \end{cases}
\end{equation*}
\end{lem}
\begin{proof}
By the Lefschetz trace formula, we have 
\[
 \sum_{x \in \bA^1 (\bF_p)} \Tr (\Fr_p, (\pi^* \cL_{\psi_0})_x ) 
 = \sum_{i=0}^2 (-1)^i \Tr \bigr( \Fr_p, 
  H_{\mathrm{c}}^i (\mathbb{A}_{k^{\mathrm{ac}}}^1 ,
 \pi^* \mathcal{L}_{\psi}) \bigl) 
\]
where $\Fr_p$ denotes the geometric $p$-th power Frobenius morphism. 
Since 
$H_{\mathrm{c}}^i (\mathbb{A}_{k^{\mathrm{ac}}}^1 ,
 \pi^* \mathcal{L}_{\psi})$ vanishes for $i=0,2$, 
we have
\begin{align*}
 \Tr \tau_{\psi_0} \bigl( \Fr(1) \bigr) &= 
 -\sum_{x \in \bA^1 (\bF_p)} \Tr (\Fr_p, (\pi^* \cL_{\psi_0})_x ) \\ 
 &=
 -\sum_{x \in \bF_p} \psi_0 (x^{p^e +1}) = 
 -\sum_{x \in \bF_p} \psi_0 (x^2) = 
 \begin{cases}
  - \epsilon(p) \sqrt{p} 
  \quad & \textrm{if}\ p \neq 2, \\ 
  0 
  \quad & \textrm{if}\ p=2,  
 \end{cases}
\end{align*}
where we use \eqref{Gss} in the last equality. 
\end{proof}

We assume $p=2$ in this paragraph. 
We take $b_0 \in \mathbb{F}_{2^{2e}}$ such that 
$\Tr_{\mathbb{F}_{2^{2e}}/\mathbb{F}_2} (b_0 )=1$. 
Further, we put 
\begin{equation}\label{defgamma_0}
 c_0 = 
 b_0^{2^e} + \sum_{0 \leq i < j \leq e-1} 
 b_0^{2^{e+i} + 2^j}. 
\end{equation}
Then we have 
\begin{align}
c_0^2-c_0&=b_0^{2^{e+1}}+b_0^{2^e}+\sum_{0 \leq i < j \leq e-1} 
 b_0^{2^{e+i+1} + 2^{j+1}}+\sum_{0 \leq i < j \leq e-1} 
 b_0^{2^{e+i} + 2^j} \notag \\
 &=b_0^{2^{e+1}}+b_0^{2^e}
 +\sum_{i=0}^{e-2} b_0^{2^{e+i+1}+2^e} 
 +\sum_{j=1}^{e-1} b_0^{2^e+2^j} \notag \\
 &=b_0^{2^{e+1}}+b_0^{2^e}+b_0^{2^e}(1+b_0+b_0^{2^e})=b_0^{2^e+1}, \label{c_0b_0}
\end{align}
where we use $\Tr_{\mathbb{F}_{2^{2e}}/\mathbb{F}_2} (b_0 )=1$
at the third equality.
We put 
\[
 \bom{g} = \bigl( (1,b_0,c_0),-1 \bigr) \in Q \rtimes \mathbb{Z}. 
\]
\begin{lem}\label{2trFr}
We assume that $p=2$. 
Then we have 
$\Tr \tau_{\psi_0} ( \bom{g}^{-1} )= -2$. 
\end{lem}
\begin{proof}
We note that 
\begin{equation}\label{ginv}
 \bom{g}^{-1} = 
 \Fr (1) 
 \biggl( \biggr( 1,b_0,c_0 +\sum_{i=0}^{e-1} 
 (b_0^{2^e +1})^{2^i} \biggr),0 \biggr). 
\end{equation}
For $y \in {k^{\mathrm{ac}}}$ 
satisfying $y^2 +b_0^{2^e} =y$, 
we take $x_y \in {k^{\mathrm{ac}}}$ 
such that 
$x_y^2-x_y=y^{2^e+1}$. 
We take 
$y_0 \in k^{\mathrm{ac}}$ such that 
$y_0^2 +b_0^{2^e} =y_0$. 
Then, 
by the Lefschetz trace formula and \eqref{ginv}, 
we have 
\begin{align*}\label{Trquad}
 \Tr \tau_{\psi_0} ( \bom{g}^{-1} ) 
 &
 = -\sum_{y^2 +b_0^{2^e} =y} \Tr (\bom{g}^{-1}, (\pi^* \cL_{\psi_0})_y ) 
 \\ 
 &=
 -\sum_{y^2 +b_0^{2^e} =y} \psi_0 
 \biggl( x_y^2 -x_y +\sum_{i=0}^{e-1} 
 (b_0 y^2)^{2^i} +c_0 +\sum_{i=0}^{e-1} 
 (b_0^{2^e +1})^{2^i} \biggr) \\ 
 &= -\sum_{z \in \mathbb{F}_2 } \psi_0 
 \biggl( (y_0 +z)^{2^e +1} +\sum_{i=0}^{e-1} 
 \bigl( b_0 (y_0 +z) \bigr)^{2^i} 
 +c_0 \biggr) 
 = -2, 
\end{align*}
where we change a variable by $y=y_0 + z$ at the second equality, 
and use 
\begin{align*}
& y_0^{2^e +1} +\sum_{i=0}^{e-1} 
 ( b_0 y_0 )^{2^i} = 
 y_0 \biggl( y_0 +\sum_{i=0}^{e-1} b_0^{2^{e+i}} \biggr) +
 \sum_{i=0}^{e-1} b_0^{2^i} 
 \biggl( y_0 +\sum_{j=0}^{i-1} b_0^{2^{e+j}} \biggl) 
 =c_0, \\
& y_0^{2^e}+y_0+\sum_{i=0}^{e-1} b_0^{2^i}
=\sum_{i=0}^{e-1}(y_0^2+y_0)^{2^i}+
\sum_{i=0}^{e-1} b_0^{2^i}=\Tr_{\mathbb{F}_{2^{2e}}/\mathbb{F}_2}(b_0)=1
\end{align*}
at the last equality. 
\end{proof}

\begin{prop}\label{charrep}
The representation $\tau_{\psi_0}$ 
is characterized by 
$\tau_{\psi_0} |_Q \simeq \tau^0$ and 
\[
 \begin{cases} 
 \Tr \tau_{\psi_0} \bigl( \Fr(1) \bigr) = 
  - \epsilon(p) \sqrt{p} & \quad 
  \textrm{if $p \neq 2$}, \\ 
 \Tr \tau_{\psi_0} ( \bom{g}^{-1} )= -2 & \quad 
 \textrm{if $p=2$}. 
 \end{cases}
\]
In particular, $\tau_{\psi_0}$ 
does not depend on the choice of 
$\ell$ and $\iota$. 
\end{prop}
\begin{proof}
This follows from 
Lemma \ref{restQ}, Lemma \ref{TrFr} 
and Lemma \ref{2trFr}. 
\end{proof}

\section{Representations of general linear algebraic groups}\label{GLrep}
\subsection{Simple supercuspidal representation}
Let $\pi$ be an irreducible supercuspidal representation 
of $\iGL_n (K)$ over $\bC$. 
Let 
$\varepsilon( \pi, s, \Psi)$ denote the Godement--Jacquet local constant of 
$\pi$ with respect to the non-trivial character 
$\Psi \colon K \to \bC^{\times}$. 
We simply write 
$\varepsilon( \pi, \Psi)$ for $\varepsilon( \pi, 1/2, \Psi)$. 
By \cite[Theorem 3.3 (4)]{GJzsim}, 
there exists an integer $\mathrm{sw}(\pi)$ such that 
\[
 \varepsilon(\pi,s,\psi_K) 
 =q^{-\mathrm{sw}(\pi) s} 
 \varepsilon(\pi, 0, \psi_K). 
\]
We put $\mathrm{Sw}(\pi) =\max \{ \mathrm{sw}(\pi),0 \}$, 
which we call the Swan conductor of $\pi$. 

\begin{defn}
An irreducible supercuspidal representation $\pi$ of $\iGL_n (K)$ 
over $\bC$ 
is called 
simple supercuspidal if $\mathrm{Sw}(\pi)=1$. 
\end{defn}
This definition is equivalent to \cite[Definition 1.1]{ITsimpJL}
by \cite[Proposition 1.3]{ITsimpJL}. 
\subsection{Construction}
In the following, 
we construct a smooth representation 
$\pi_{\zeta,\chi,c}$ of  $\textit{GL}_n(K)$ 
for each triple 
$(\zeta,\chi,c) \in \mu_{q-1}(K) \times 
(k^{\times})^{\vee} \times \bC^{\times}$. 

Let $B \subset M_n(k)$ be the subring 
consisting of upper triangular matrices. 
Let 
$\mathfrak{I} \subset M_n(\mathcal{O}_K)$ 
be the inverse image of 
$B$ under the reduction map
$M_n(\mathcal{O}_K) \to M_n(k)$.
Then $\mathfrak{I}$ is a hereditary 
$\mathcal{O}_K$-order (cf.~\cite[(1.1)]{BKadmd}). 
Let $\mathfrak{P}$ 
denote the Jacobson radical 
of the order $\mathfrak{I}$. 
We put 
$U_{\mathfrak{I}}^1 =
 1+ \mathfrak{P} \subset \mathit{GL}_n (\mathcal{O}_K )$. 
We set 
\[
 \varphi_{\zeta}=\begin{pmatrix}
 \bm{0} & I_{n-1} \\
 \zeta \varpi & \bm{0}
 \end{pmatrix} \in M_n(K) \quad \textrm{and} \quad  
 L_{\zeta}=K(\varphi_{\zeta}). 
\]
Then, $L_{\zeta}$ is a 
totally ramified extension of $K$ 
of degree $n$. 

We put $\varphi_{\zeta,n} =n' \varphi_{\zeta}$ and 
\[
 \epsilon_0=
 \begin{cases}
  (n'+1)/2 \quad & \textrm{if $p^e=2$},\\
  0 \quad & \textrm{if $p^e \neq 2$}. 
 \end{cases}
\]
We define a character 
$\Lambda_{\zeta,\chi,c} \colon L_{\zeta}^{\times} U_{\mathfrak{I}}^1 
 \to \bC^{\times}$ by
\begin{align*}\label{defLambda}
 \Lambda_{\zeta,\chi,c} (\varphi_{\zeta})& 
 =(-1)^{n-1+\epsilon_0 f} c, \quad 
 \Lambda_{\zeta,\chi,c} (x) =\chi (\bar{x}) \quad 
 \textrm{for $x \in \mathcal{O}_{K} ^{\times}$},\\ 
 \Lambda_{\zeta,\chi,c} (x) & =
 (\psi_K \circ 
 \tr )( \varphi_{\zeta,n}^{-1}(x-1)) \quad 
 \textrm{for $x \in U_{\mathfrak{I}}^1$}, 
\end{align*}
where $\tr$ means 
the trace as an element of $M_n (K)$. 
We put 
\[
 \pi_{\zeta,\chi,c} = 
 \mathrm{c\mathchar`-Ind}_{L_{\zeta}^{\times} 
 U_{\mathfrak{I}}^1}^{\mathit{GL}_n(K)} 
 \Lambda_{\zeta,\chi,c}. 
\] 
Then, $\pi_{\zeta,\chi,c}$ 
is a simple supercuspidal representation of $\mathit{GL}_n (K)$, and 
every simple supercuspidal representation 
is isomorphic to 
$\pi_{\zeta,\chi,c}$ for a uniquely determined 
$(\zeta,\chi,c) \in \mu_{q-1}(K) \times (k^{\times})^{\vee} 
 \times \bC^{\times}$ 
by \cite[Proposition 1.3]{ITsimpJL}. 
The representation $\pi_{\zeta,\chi,c}$
contains the m-simple stratum 
$\bigl[ \mathfrak{I},1,0,\varphi_{\zeta,n}^{-1} \bigr]$ 
in the sense of \cite[2.1]{BHLepi}. 

\begin{prop}\label{epspi}
$\varepsilon (\pi_{\zeta,\chi,c} ,\psi_K) 
 =(-1)^{n-1+\epsilon_0 f} \chi(n') c$. 
\end{prop}
\begin{proof}
This follows from \cite[6.1 Lemma 2 and 6.3 Proposition 1]{BHloctII}. 
\end{proof}

\section{Local Langlands correspondence}\label{LLCstate}
Our main theorem is the following: 

\begin{thm}
The representations 
$\pi_{\zeta,\chi,c}$ and 
$\tau_{\zeta,\chi,c}$ correspond via the local Langlands correspondence. 
\end{thm}

To prove this theorem, 
we recall a characterization of the local Langlands correspondence for 
epipelagic representations due to Bushnell--Henniart. 
Recall that $\Psi \colon K \to \bC^{\times}$ 
is a non-trivial character. 
The following lemma is a special case of 
\cite[Proposition 4.13]{DHvartor}. 

\begin{lem}(\cite[2.3 Lemma]{BHLepi})\label{rswW}
Let $\tau$ be an 
irreducible smooth representation of $W_K$
such that $\mathrm{sw}(\tau) \geq 1$. 
Then, there exists 
$\gamma_{\tau,\Psi} \in K^{\times}$
such that 
\[
\varepsilon(\chi \otimes \tau,s,\Psi)=
\chi (\gamma_{\tau,\Psi})^{-1}
\varepsilon(\tau,s,\Psi)
\]
for any tamely ramified character 
$\chi$ of $W_K$.
This property determines the coset 
$\gamma_{\tau,\Psi} U_K^1$
uniquely. 
\end{lem}

\begin{defn}\label{rswtaudef}
Let $\tau$ be an 
irreducible smooth representation of $W_K$ 
such that $\mathrm{sw}(\tau) \geq 1$. 
We take $\gamma_{\tau,\Psi}$ 
as in Lemma \ref{rswW}. 
We put 
\[
 \mathrm{rsw}(\tau,\Psi ) =\gamma_{\tau,\Psi}^{-1} 
 \in K^{\times}/U_K^1, 
\] 
which we call 
the refined Swan conductor of $\tau$ 
with respect to $\Psi$. 
\end{defn}

\begin{rem}
By \eqref{twun}, 
we have 
$v_K \bigr( \mathrm{rsw}(\tau,\psi_K ) \bigr)= 
 \mathrm{Sw} (\tau)$ 
in Definition \ref{rswtaudef}. 
\end{rem}

\begin{lem}\label{rswGL}
Let $\pi$ be an 
irreducible supercuspidal representation of $\iGL_n (K)$ 
such that $\mathrm{sw}(\pi) \geq 1$. \\ 
{\rm 1.}
There exists 
$\gamma_{\pi,\Psi} \in K^{\times}$ 
such that 
\[
 \varepsilon(\chi \otimes \pi,s,\Psi)=
 \chi (\gamma_{\pi,\Psi})^{-1}
 \varepsilon(\pi,s,\Psi)
\]
for any tamely ramified character 
$\chi$ of $K^{\times}$. 
This property determines the coset 
$\gamma_{\pi,\Psi} U_K^1$
uniquely. \\ 
{\rm 2.} 
Let $\bigl[ \mathfrak{A},m,0,\alpha \bigr]$ 
be a simple stratum contained in $\pi$. 
Then we have 
$\gamma_{\pi,\Psi} \equiv \det \alpha \mod U_K^1$. 
\end{lem}
\begin{proof}
The first statement is 
\cite[1.4 Theorem (i)]{BHloctII}. 
The second statement follows from 
\cite[1.4 Remark]{BHloctII}. 
\end{proof}

\begin{defn}\label{rswpidef}
Let $\pi$ be an 
irreducible supercuspidal representation of $\iGL_n (K)$ 
such that $\mathrm{sw}(\pi) \geq 1$. 
We take $\gamma_{\pi,\Psi}$ 
as in Lemma \ref{rswGL}. 
Then we put 
\[
 \mathrm{rsw}(\pi,\Psi) =\gamma_{\pi,\Psi}^{-1} 
 \in K^{\times}/U_K^1, 
\]
which we call 
the refined Swan conductor of $\pi$ 
with respect to $\Psi$. 
\end{defn}

\begin{rem}
We have 
$v_K \bigr( \mathrm{rsw}(\pi ,\psi_K ) \bigr)= 
 \mathrm{Sw} (\pi)$ 
in Definition \ref{rswpidef}. 
\end{rem}

For an irreducible supercuspidal representation $\pi$ 
of $\iGL_n (K)$, 
let $\omega_{\pi}$ denote the central character of $\pi$. 

\begin{prop}(\cite[2.3 Proposition]{BHLepi})\label{LLCchar}
Let $\pi$ be 
a simple supercuspidal representation of $\iGL_n (K)$. 
The Langlands parameter for $\pi$ 
is characterized 
as the $n$-dimensional irreducible smooth representation $\tau$ of $W_K$ 
satisfying 
\begin{enumerate}
\item 
$\det \tau =\omega_{\pi}$, 
\item 
$\mathrm{rsw}(\tau,\psi_K ) = 
 \mathrm{rsw}(\pi,\psi_K )$, 
\item 
$\varepsilon (\tau, \psi_K )
 =\varepsilon (\pi, \psi_K)$. 
\end{enumerate}
\end{prop} 

We will show that 
$\tau_{\zeta,\chi,c}$ and $\pi_{\zeta,\chi,c}$ 
satisfy the conditions of 
Proposition \ref{LLCchar} 
in Proposition \ref{centdet}, 
Proposition \ref{rswdet}, 
Lemma \ref{tauirr} 
and Proposition \ref{coineps}. 

\section{General facts on epsilon factors}\label{sec:gene}
In this section, we recall 
some general facts on epsilon factors. 

For a finite separable extension $L$ over $K$, 
we put $\Psi_L =\Psi \circ \Tr_{L/K}$ 
and let  
\[
\lambda (L/K, \Psi )
=\frac{\varepsilon(\mathrm{Ind}_{L/K}1,s,\Psi)}{\varepsilon(1 ,s,\Psi_L )}
\] denote the 
Langlands constant which is independent of $s$, where 
$1$ is the trivial representation of 
$W_L$ (cf.~\cite[30.4]{BHLLCGL2}).

\begin{prop}\label{reseps}
Let $\tau$ be a finite dimensional smooth representation of 
$W_K$ such that 
$\tau |_{P_K}$ is irreducible and non-trivial. 
Let $L$ be a tamely ramified finite extension of $K$. 
Then we have 
\[
 \varepsilon (\tau |_{W_L} ,\Psi_L ) = 
 \lambda (L/K ,\Psi )^{-\dim \tau} 
 \delta_{L/K} (\mathrm{rsw} (\tau,\Psi)) 
 \varepsilon (\tau ,\Psi )^{[L:K]} . 
\]
\end{prop}
\begin{proof}
This is proved by the same arguments as \cite[48.3 Proposition]{BHLLCGL2}. 
\end{proof}

\begin{prop}\label{refSwres}
Let $\tau$ be a finite dimensional smooth representation of 
$W_K$ such that 
$\tau |_{P_K}$ does not contain the trivial character. \\ 
{\rm 1.} 
If $\phi$ is a tamely ramified character of $W_K$, 
then we have 
$\mathrm{rsw} (\tau \otimes \phi,\Psi)=\mathrm{rsw} (\tau,\Psi)$. \\ 
{\rm 2.} 
Let $L$ be a tamely ramified finite extension of $K$. 
Then we have 
\[
 \mathrm{rsw} (\tau |_{W_L} ,\Psi_L) 
 = \mathrm{rsw} (\tau,\Psi) \mod 
 U_L^1 . 
\]
\end{prop}
\begin{proof}
This is \cite[48.1 Theorem (2) and (3)]{BHLLCGL2}. 
\end{proof}

For a non-trivial character 
$\xi$ of $K^{\times}$, 
the level of $\xi$ means the least integer $m \geq 0$ 
such that $\xi$ is trivial on $U_K^{m+1}$. 

\begin{prop}\label{chare}
Let $\xi$ be a character of $K^{\times}$ of level $m \geq 1$. 
Assume that $\gamma \in K^{\times}$ satisfies 
\[
 \xi (1+x) =\Psi (\gamma x)
\] 
for $x \in \fp_K^{[m/2]+1}$. \\ 
{\rm 1.} 
We have $\mathrm{rsw}(\xi,\Psi)=\gamma^{-1}$. \\ 
{\rm 2.} 
We have 
\[
 \varepsilon (\xi, \Psi) = 
 q^{[(m+1)/2] -(m+1)/2} 
 \sum_{y \in U_K^{[(m+1)/2]} / U_K^{[m/2]+1}} 
 \xi(\gamma y)^{-1} \Psi (\gamma y) . 
\]
\end{prop}
\begin{proof}
The claim 1 follows from \cite[23.8 Stability theorem]{BHLLCGL2}. 
The claim 2 follows from \cite[23.5 Lemma 1, (23.6.2) and 
23.6 Proposition]{BHLLCGL2}. 
\end{proof}

For a finite Galois extension $L$ of $K$, 
let $\psi_{L/K}$ denote the Herbrand function 
of $L/K$ and 
$\Gal (L/K)_i$ denote the 
lower numbering $i$-th ramification 
subgroup of $\Gal (L/K)$ for $i \geq 0$ (\cf \cite[IV]{SerCL}). 
We use the following lemmas to calculate 
the refined Swan conductor of a character of a Weil group. 

\begin{lem}\label{Arfilt}
Let $m$ be a positive integer dividing $f$. 
Let $h$ be a positive integer that is prime to $p$ and less than 
$p^m v_K(p)/(p^m -1)$. 
Let $L$ be a Galois extension of $K$ 
defined by $x^{p^m} -x=1/{\varpi^h}$. 
Then we have 
\[
 \Gal (L/K)_i 
 =
 \begin{cases}
 \Gal (L/K) \quad & \textrm{if $i \leq h$, } \\ 
 \{ 1 \} \quad & \textrm{if $i>h$} 
 \end{cases} 
\]
and 
\[
 \psi_{L/K}(v) 
 =
 \begin{cases}
 v \quad & \textrm{if $v \leq h$, } \\ 
 p^m (v-h) +h \quad & \textrm{if $v>h$}. 
 \end{cases} 
\]
\end{lem}
\begin{proof}
Take an integer $l$ such that 
$l h \equiv 1 \mod p^m$. 
Then we have 
\[
 v_L \biggl( \frac{1}{x^l \varpi^{(lh-1)/p^m}} \biggr) =1. 
\] 
Hence, for $\sigma \in \Gal (L/K)$ and $i \geq 0$, 
we have $\sigma \in \Gal (L/K)_i$ if and only if 
\begin{equation}\label{Gicond}
 i+1 \leq v_L \biggl( 
 \sigma \biggl( \frac{1}{x^l \varpi^{(lh-1)/p^m}} \biggr) 
 -\frac{1}{x^l \varpi^{(lh-1)/p^m}} \biggr) 
 = v_L (\sigma(x)^l -x^l ) +hl+1 . 
\end{equation}
The right hand side of \eqref{Gicond} 
is $h+1$ if $\sigma \neq 1$. 
Hence the first claim follows. 
The second claim follows from the first claim. 
\end{proof}

\begin{lem}\label{rswcal}
Let $L$ be a totally ramified finite abelian extension of $K$. 
Let $m \geq 1$. \\ 
{\rm 1.} 
We have 
\begin{align*}
 &\Nr_{L/K} (U_L^{\psi_{L/K}(m)}) \subset U_K^m, \quad 
 \Nr_{L/K} (U_L^{\psi_{L/K}(m)+1}) \subset U_K^{m+1}, \\ 
 &\Art_K (U_K^m) \subset \Gal (L/K)_{\psi_{L/K}(m)}. 
\end{align*}
\\
{\rm 2.} 
We take $\alpha \in K$ and $\beta \in L$ 
such that $v_K(\alpha)=m$ and $v_L (\beta)=\psi_{L/K}(m)$. 
We put $P(z)=z^p -z$ for $z \in k$. 
Assume that 
\[
 \xymatrix{ 
 U_L^{\psi_{L/K}(m)}
 \ar[rr]^-{\Nr_{L/K}} 
 \ar[d]^{p_{L,\beta}} & & 
 U_K^m
 \ar[d]^{p_{K,\alpha}} 
 \\
 k \ar[rr]^{P} & & k 
 }
\] 
is commutative, where 
\begin{align*}
 &p_{K,\alpha} \colon U_K^m 
 \longrightarrow k;\ 
 1+\alpha x \mapsto \bar{x}, \\ 
 &p_{L,\beta} \colon  U_L^{\psi_{L/K}(m)}
 \longrightarrow k;\ 
 1+\beta x \mapsto \bar{x}. 
\end{align*}
Let $\varpi_L$ be a uniformizer of $L$. 
Then we have 
\[
 p_{L,\beta} \biggl( 
 \frac{\Art_K (1+\alpha x)(\varpi_L)}{\varpi_L} \biggr) 
 = \Tr_{k/\bF_p}(\bar{x}) 
\]
for $x \in \cO_K$. 
\end{lem}
\begin{proof}
The first claim follows from 
\cite[V, \S 3, Proposition 4 and XV, \S 2, Corollaire 3 of Th\'eor\`eme 1]{SerCL}. 
We note that 
our normalization of the Artin reciprocity map 
is inverse to that in \cite[XIII, \S 4]{SerCL}. 
Let $x \in \cO_K$. 
By \cite[XV, \S 3, Proposition 4]{SerCL} 
and the construction of the isomorphism of  
\cite[XV, \S 2, Proposition 3]{SerCL}, 
we have 
\[
 p_{L,\beta} \biggl( 
 \frac{\Art_K (1+\alpha x)(\varpi_L)}{\varpi_L} \biggr) 
 = z_x^q -z_x , 
\]
where we take $z_x \in k^{\rmac}$ 
such that $z_x^p -z_x =\bar{x}$. 
Then we have the second claim, since 
\[
 z_x^q -z_x = \Tr_{k/\bF_p}(z_x^p -z_x) 
 =\Tr_{k/\bF_p}(\bar{x}) 
\] 
for such $z_x$. 
\end{proof}

\section{Stiefel--Whitney class and discriminant}\label{SWdisc}
\subsection{Stiefel--Whitney class}

Let $R (W_K,\bR)$ be the Grothendieck group of 
finite-dimensional representations of $W_K$ over 
$\bR$ with finite images. 
For $V \in R (W_K,\bR)$, we put 
$V_{\bC}=V \otimes_{\bR} \bC$ 
and define 
$\varepsilon (V_{\bC},\Psi)$ by the additivity 
using the epsilon factors in subsection \ref{sseq:Swc}. 
For $V \in R (W_K,\bR)$, we define 
the $i$-th Stiefel--Whitney class 
$w_i (V) \in H^i (G_K,\bZ/2\bZ)$ 
for $i \geq 0$ as in \cite[(1.3)]{DelclAr}. 
Let 
\[
 \mathrm{cl} \colon 
 H^2 (G_K,\bZ/2\bZ) \to H^2 (G_K, K^{\mathrm{ac},\times}) 
 \xrightarrow{\sim} \bQ /\bZ, 
\]
where the first map is induced by 
$\bZ/2\bZ \to K^{\mathrm{ac},\times}; m \mapsto (-1)^m$ 
and the second isomorphism is the invariant map.

\begin{thm}(\cite[(1.5) Th\'{e}or\`{e}me]{DelclAr})\label{SWeps} 
Assume that $V \in R (W_K,\bR)$ has dimension $0$ and 
determinant $1$. 
Then we have 
\[
 \varepsilon (V_{\bC},\Psi) =\mathrm{exp} 
 \bigl( 2 \pi \sqrt{-1} \mathrm{cl}(w_2(V)) \bigr). 
\] 
In particular, we have 
$\varepsilon (V_{\bC},\Psi) =1$ if 
$\ch K =2$. 
\end{thm}

\subsection{Discriminant}
Let $L$ be a finite separable extension of $K$. 
We put 
\[
 \delta_{L/K} =\det (\Ind_{L/K} 1). 
\]

\subsubsection{Multiplicative discriminant}
We assume that $\ch K \neq 2$ in this subsubsection. 
We define $d_{L/K} \in K^{\times}/(K^{\times})^2$ 
as the discriminant of the quadratic form 
$\Tr_{L/K} (x^2)$ on $L$. 
For $a \in K^{\times}/(K^{\times})^2$, 
let 
$\{ a \} \in H^1 (G_K, \bZ/2\bZ)$ and 
$\kappa_a \in \Hom (W_K ,\{ \pm 1 \})$ 
be the elements corresponding to $a$ 
under the natural isomorphisms 
\[
 K^{\times}/(K^{\times})^2 \simeq H^1 (G_K, \bZ/2\bZ) 
 \simeq \Hom (W_K ,\{ \pm 1 \}). 
\]
We have 
\begin{equation}\label{dka}
 \delta_{L/K} =\kappa_{d_{L/K}} 
\end{equation}
by \cite[V, \S 10, 2 Example 6)]{BourAlg47} (\cf \cite[1.4]{SerinvWit}). 
For $a, b \in K^{\times}/(K^{\times})^2$, 
we put  
\[
 \{ a, b \} =\{ a \} \cup \{ b \} \in  H^2 (G_K,\bZ/2\bZ ). 
\]
\begin{prop}(\cite[Proposition 6.5]{AbSaLF})\label{dw2f}
Let $m$ be the extension degree of $L$ over $K$. 
We take a generator $a$ of $L$ over $K$. 
Let $f(x) \in K[x]$ be the minimal polynomial of $a$. 
We put $D =f' (a) \in L$. 
Then we have 
\begin{align*}
 d_{L/K} &= (-1)^{\binom{m}{2}} \Nr_{L/K} (D) 
 \in K^{\times}/(K^{\times})^2, \\ 
 w_2 (\Ind_{L/K} \kappa_D ) &= 
 \binom{m}{4} \{ -1,-1 \} + \{ d_{L/K} ,2 \} 
 \in  H^2 (G_K,\bZ/2\bZ ).
\end{align*}
\end{prop}

\subsubsection{Additive discriminant}\label{adddis}
We put $P_m (x)=x^m -x$ for any positive integer $m$. 
We assume that $\ch K=2$ in this subsubsection. 

\begin{defn}(\cite[D\'efinition 2.7]{BMFquad2})\label{adiscdef}
Let $m$ be the extension degree of $L$ over $K$. 
Let $f(x) \in K[x]$ be the minimal polynomial 
of a generator of $L$ over $K$. 
We have a decomposition 
$f(x) =\prod_{1 \leq i \leq m} (x-a_i )$ 
over the Galois closure of $L$ over $K$. 
We put 
\[
 d_{L/K}^+ = \sum_{1 \leq i < j \leq m} \frac{a_i a_j}{(a_i +a_j)^2} 
 \in K/P_2 (K), 
\]
which we call the additive discriminant of $L$ over $K$. 
\end{defn}

\begin{thm}(\cite[Th\'eor\`eme 2.7]{BMFquad2})\label{adiscthm}
Let $L'$ be the subextension of $K^{\rmac}$ over $K$ 
corresponding to $\Ker \delta_{L/K}$. 
Then the extension $L'$ over $K$ corresponds to 
$d_{L/K}^+ \in K/P_2 (K)$ by 
the Artin--Schreier theory. 
\end{thm}

\section{Product formula of Deligne--Laumon}\label{DLprod}
We recall a statement of 
the product formula of Deligne--Laumon. 
In this paper, we need only the rank one case, which is proved in 
\cite[Proposition 10.12.1]{DelconstL}, but 
we follow the notation in \cite{LauTFcW}. 

\subsection{Local factor}\label{lf}
We consider a triple $(T,\cF,\omega)$ 
which consists of the following: 

\begin{itemize}
\item
The affine scheme $T =\Spec \cO_{K_T}$ where 
$\cO_{K_T}$ is the ring of integers in 
a local field $K_T$ 
of characteristic $p$ whose residue field 
contains $k$. 
\item
A constructible $\ol{\bQ}_{\ell}$-sheaf $\cF$ on $T$. 
\item
A non-zero meromorphic $1$-form $\omega$ on $T$. 
\end{itemize}
Then we can associate 
$\varepsilon_{\psi_0} (T, \cF , \omega ) \in \bC^{\times}$ 
to the triple $(T,\cF,\omega)$ as in 
\cite[Th\'{e}or\`{e}me (3.1.5.4)]{LauTFcW} 
using $\iota$. 

Assume that $K_T =k((t))$. 
Let $\eta =\Spec k((t))$ be the generic point of $T$ 
with the natural inclusion 
$j \colon \eta \to T$. 
We define a character 
$\Psi_{\omega} \colon k((t)) \to \bC^{\times}$ 
by 
\[
 \Psi_{\omega} (a) = 
 (\psi_0 \circ \Tr_{k/\bF_p})(\mathrm{Res} (a \omega) )
\]
for $a \in k((t))$. 
Let $l(\Psi_{\omega})$ be the level of 
$\Psi_{\omega}$ in the sense of \cite[1.7 Definition]{BHLLCGL2}. 
We fix an algebraic closure $k((t))^{\mathrm{ac}}$ of $k((t))$. 
For a rank $1$ smooth $\ol{\bQ}_{\ell}$-sheaf $V$ on 
$\eta$ corresponding to a character 
$\chi \colon G_{k((t))} \to \bC^{\times}$ via $\iota$, 
we have 
\begin{equation}\label{locep}
 \varepsilon_{\psi_0} (T, j_* V , \omega ) = 
 q^{-l(\Psi_{\omega})/2} \varepsilon (\chi \omega_{-\frac{1}{2}}, \Psi_{\omega}) 
\end{equation}
by \cite[Th\'{e}or\`{e}me (3.1.5.4)(v)]{LauTFcW}, 
\cite[(3.6.2)]{TatNtb} and 
\cite[23.1 Proposition (3)]{BHLLCGL2}. 

\subsection{Product formula}
Let $X$ be a geometrically connected proper smooth curve over 
$k$ of genus $g$. 
Let $\cF$ be a constructible $\ol{\bQ}_{\ell}$-sheaf 
on $X$. 
Let $\mathrm{Frob}_q \in G_k$ be the geometric Frobenius element. 
We put 
\[
 \varepsilon (X,\cF) = 
 \iota \Biggl( \prod_{i=0}^2 \det \bigl( 
 - \mathrm{Frob}_q ; H^i (X \otimes_k k^{\rmac},\cF) \bigr)^{(-1)^{i-1}} 
 \Biggr). 
\]
Let $\mathrm{rk}(\cF)$ be the generic rank of $\cF$. 

\begin{thm}(\cite[Th\'{e}or\`{e}me (3.2.1.1)]{LauTFcW})\label{prodf}
Let $\omega$ be a non-zero meromorphic $1$-form on $X$. 
Then we have 
\[
 \varepsilon (X,\cF) = q^{\mathrm{rk}(\cF)(1-g)} 
 \prod_{x \in \lvert X \rvert} 
 \varepsilon_{\psi_0} (X_{(x)}, \cF |_{X_{(x)}} , \omega|_{X_{(x)}} ), 
\]
where 
$\lvert X \rvert$ is the set of closed points of $X$, 
and $X_{(x)}$ is the completion of $X$ at $x$. 
\end{thm}

\section{Determinant}\label{detsect}

In this section, we study $\det \tau_{\psi_0}$ 
to show the equality 
$\omega_{\pi_{\zeta,\chi,c}} =\det \tau_{\zeta,\chi,c}$ 
of the central character and the determinant. 
We use the product formula of Deligne--Laumon 
to study 
$\det \tau_{\psi_0} \bigl( \Fr(1) \bigr)$, 
where $\Fr(1)$ is defined in \eqref{Frmdef}. 

\begin{lem}\label{Qab}
We have $Q^{\rmab} =Q/Q_0$. 
\end{lem}
\begin{proof}
By Lemma \ref{Q0Fsur}, we have $Q^{\rmab} =(Q/F)^{\rmab}$. 
For $(a,b,c) \in Q$, let $(a,b)$ be the image of $(a,b,c)$ 
in $Q/F$. 
Then we have 
\[
 (a,0) (a,b) (a,0)^{-1} (a,b)^{-1} 
 = (1,(a-1)b ). 
\]
Hence, we obtain the claim. 
\end{proof}
We view $\theta_0$ defined in \eqref{theta_0} as a character of 
$Q$ by $(a,b,c) \mapsto \theta_0(a)$. 
Recall that $\tau^0$ is the representation of $Q$ defined in \eqref{cht}.
\begin{lem}\label{dettau0}
We have $\det \tau^0 =\theta_0$. 
\end{lem}
\begin{proof}
By Lemma \ref{Qab}, 
it suffices to show 
$\det \tau^0 =\theta_0$ on $\mu_{p^e +1} (k^{\rmac})$. 
By Lemma \ref{HLdec} and Lemma \ref{restQ}, 
we have 
\[
\det \tau^0(a)=\prod_{\chi \in \mu_{p^e+1}(k^{\rmac})^{\vee} \setminus \{1\}} \chi(a) 
\] 
for $a \in \mu_{p^e+1}(k^{\rmac})$.
Hence, the claim follows.
\end{proof}

For $a \in k^{\times}$, let $\bigl( \frac{a}{k} \bigr)$ 
denote the quadratic residue symbol of $k$ 
defined by 
\[
 \Bigl( \frac{a}{k} \Bigr) = 
 \begin{cases}
 1 & \textrm{if $a$ is square in $k$,}\\ 
 -1 & \textrm{if $a$ is not square in $k$.} 
 \end{cases}
\]
\begin{lem}\label{deltatame}
Let $m$ be a positive integer that is prime to $p$. 
We take an $m$-th root $\varpi^{1/m}$ of $\varpi$, 
and put $L=K(\varpi^{1/m})$. \\ 
{\rm 1.} 
If $m$ is odd, then 
$\delta_{L/K}$ is the unramified character 
satisfying 
$\delta_{L/K} (\varpi ) =  \bigl( \frac{q}{m} \bigr)$. \\ 
{\rm 2.} 
If $m$ is even, we have 
$\delta_{L/K} (\varpi ) =
 \bigl( \frac{-1}{q} \bigr)^{\frac{m}{2}}$ 
and 
$\delta_{L/K} (x ) = \bigl( \frac{\bar{x}}{k} \bigr)$ 
for $x \in \cO_K^{\times}$. 
\end{lem}
\begin{proof}
These are proved in \cite[(10.1.6)]{BFGdiv} 
if $\ch K =0$. 
Actually, the same proof works 
also in the positive characteristic case. 
\end{proof}

\begin{lem}\label{lconsttame}
Let $m, m'$ be positive integers that are prime to $p$. 
We take an $m$-th root $\varpi^{1/m}$ of $\varpi$, 
and put $L=K(\varpi^{1/m})$. 
Let $\psi'_K \colon K \to \bC^{\times}$ 
be a character such that 
$\psi'_K (x)=\psi_0 (\Tr_{k/\bF_p} (m' \bar{x}))$ 
for $x \in \cO_K$. 
Then we have 
\begin{equation*}
 \lambda (L/K, \psi'_K) =
 \begin{cases}
 \bigl( \frac{q}{m} \bigr) & \textrm{if $m$ is odd,}\\
 - \bigl(-\epsilon(p)
 \bigl( \frac{2mm'}{p} \bigr) 
 \bigl( \frac{-1}{p} \bigr)^{\frac{m}{2}-1} 
 \bigr)^f & \textrm{if $m$ is even.}
 \end{cases}
\end{equation*}
\end{lem}
\begin{proof}
If $m$ is odd, we have 
\[
 \lambda (L/K, \psi'_K) = 
 \varepsilon (\delta_{L/K}, \psi'_K) 
 =
 \Bigl( \frac{q}{m} \Bigr)
\]
by \cite[Proposition 2]{Henepsmod} and 
Lemma \ref{deltatame}.1. 

Assume that $m$ is even. 
Note that $p \neq 2$ in this case. 
Then we have 
\begin{equation}\label{dkap2}
 d_{L/K}= (-1)^{m/2} \Nr_{L/K} (m (\varpi^{1/m})^{m-1}) 
 =-(-1)^{m/2} \varpi \in K^{\times}/(K^{\times})^2 
\end{equation}
by Proposition \ref{dw2f}. 
For 
$\chi \in (\mathbb{F}_q^{\times})^{\vee}$
and $\psi \in \mathbb{F}_q^{\vee} \setminus \{1\}$, 
we set 
\[
\tau(\chi,\psi)=-\sum_{x \in \mathbb{F}_q^{\times}} \chi^{-1}(x) \psi(x)
\]
and have the Hasse--Davenport formula  
\begin{equation}\label{HD}
\tau(\chi \circ \Nr_{\mathbb{F}_{q^n}/\mathbb{F}_q},\psi \circ \Tr_{\mathbb{F}_{q^n}/\mathbb{F}_q})=\tau(\chi,\psi)^n.  
\end{equation}
Let 
\[
 (\ ,\ )_K \colon K^{\times}/(K^{\times})^2 \times 
 K^{\times}/(K^{\times})^2 \to \{\pm 1\} 
\] 
denote the Hilbert symbol. 
By \eqref{dka} and \eqref{dkap2}, 
we have 
\[
 \delta_{L/K}(x)=\kappa_{d_{L/K}}(x)=(x,d_{L/K})_K
=(x,\varpi)_K=\left(\frac{\bar{x}}{k}\right) 
\] 
for 
$x \in \mathcal{O}_K^{\times}$. 
By \cite[23.5 Theorem]{BHLLCGL2}, 
we have
\[
\varepsilon (\delta_{L/K},\psi'_K) =q^{-\frac{1}{2}} 
\sum_{x \in \mathcal{O}_K^{\times}/U_K^1} \delta_{L/K}(x) 
\psi'_K(x)=
q^{-\frac{1}{2}} \sum_{x \in k^{\times}} 
 \Bigl( \frac{x}{k} \Bigr) \psi_0 (\Tr_{k/\bF_p} (m' x)). 
\]
By applying \eqref{HD} to the extension $k$ over $\mathbb{F}_p$ 
and using \eqref{Gss}, 
we have 
\[
q^{-\frac{1}{2}} \sum_{x \in k^{\times}} 
 \Bigl( \frac{x}{k} \Bigr) \psi_0 (\Tr_{k/\bF_p} (m' x)) 
 =- \Bigl( -\epsilon(p) \Bigl( \frac{m'}{p} \Bigr) \Bigr)^f . 
\]
Hence, we have 
\[
\lambda (L/K, \psi'_K) = 
 \varepsilon (\delta_{L/K},\psi'_K) 
 \Bigl( \frac{m}{q} \Bigl) 
 \Bigl( \frac{-1}{q} \Bigl)^{\frac{m}{2}-1} 
 (d_{L/K},2)_K 
 = - \Bigl(-
 \epsilon(p)
 \Bigl( \frac{2mm'}{p} \Bigr) 
 \Bigl( \frac{-1}{p} \Bigr)^{\frac{m}{2}-1} 
 \Bigr)^f 
\]
by \cite[Theorem II.2B]{SaiLcInd} 
 and \cite[(3.6.1)]{TatNtb}. 
\end{proof}

\begin{lem}\label{detFr}
We have 
\begin{equation*}
 \det \tau_{\psi_0} \bigl( \Fr(1) \bigr) = 
 \begin{cases}
  \bigl( -\epsilon(p)
  \bigl( {\frac{2}{p}} \bigr) \bigr)^f 
  q^{\frac{p^e}{2}} 
  \quad &\textrm{if $p \neq 2$,} \\ 
  q^{2^{e-1}} 
  \quad &\textrm{if $p = 2$.} 
 \end{cases}
\end{equation*}
\end{lem}
\begin{proof}
Let $x$ be the standard coordinate of $\bA_k^1$. 
Let $j$ be the open immersion 
$\bA_k^1 \hookrightarrow \bP_k^1$. 
We put $t=1/x$. 
As in Subsection \ref{lf}, 
we put $T=\Spec k[[t]]$ and $\eta=\Spec k((t))$ 
with the open immersion $j \colon \eta \to T$. 

We consider $k((s))$ as a subfield of 
$k((t))$ by $s=t^{p^e +1}$. 
Let 
$\tilde{\xi} \colon G_{k((s))} \to \bC^{\times}$ 
be the Artin--Schreier character 
associated to 
$y^p -y=1/s$ and $\psi_0$, 
which means the composite of 
\[
G_{k((s))} \longrightarrow \mathbb{F}_p;\ 
\sigma \mapsto \sigma(y)-y 
\]
and $\psi_0^{-1}$ where $y$ is an element of $k((t))^{\mathrm{ac}}$ such that 
$y^p -y=1/s$. 

We use the notation in Lemma \ref{rswcal}. 
Note that $\psi_{k((s))(y)/k((s))}(1)=1$ by 
Lemma \ref{Arfilt}. 
We can check that 
\[ 
 \Nr_{k((s))(y)/k((s))}(1+y^{-1}x)=1+s(x^p-x)
\] 
for $x \in k$. 
For $x \in \mathcal{O}_{k((s))}$, we have 
\begin{align*}
\tilde{\xi}(1+s x) &=\psi_0^{-1} 
 \Bigl( \mathrm{Art}_{k((s))}(1+sx)(y) - y \Bigr) \\
&=\psi_0^{-1}\left(-p_{k((s))(y),y^{-1}}\left(
\frac{\mathrm{Art}_{k((s))}(1+sx)(y^{-1})}{y^{-1}}
\right)\right)=\psi_0(\Tr_{k/\mathbb{F}_p}(\bar{x})), 
\end{align*}
where we use Lemma \ref{rswcal} 
with $\alpha=s$, $\beta=\varpi_{k((s))}=y^{-1}$. 
Hence, we have 
$\mathrm{rsw}(\tilde{\xi},\Psi_{s^{-1}ds})=s$ 
by Proposition \ref{chare}.1.

Let 
$\xi \colon G_{k((t))} \to \bC^{\times}$ 
be the restriction of 
$\tilde{\xi}$ to $G_{k((t))}$. 
Then $\xi$ is 
the Artin--Schreier character associated to
$y^p -y=1/t^{p^e +1}$ and $\psi_0$.

Let $V_{\xi}$ be the smooth 
$\overline{\mathbb{Q}}_{\ell}$-sheaf on 
$\eta$ corresponding to $\xi$ via $\iota$. 
Then we have $V_{\xi} \simeq 
\mathcal{L}_{\psi_0}|_{\eta}$ by \cite[D\'{e}finition 1.7 in Sommes trig.]{DelCoet}. 
Let the notation be as in Lemma \ref{HLdec}.
We write $\omega$ for the meromorphic $1$-form $dx$
on $\mathbb{P}_k^1$. 
By \cite[Th\'{e}or\`{e}me (3.1.5.4)(v)]{LauTFcW}, 
we have 
\[
 \varepsilon_{\psi_0} \bigl( X_{(x)}, 
 (j_!\pi^\ast  \mathcal{L}_{\psi_0})
 |_{X_{(x)}}, \omega|_{X_{(x)}} \bigr) =1 
\]
for any $x \in |\mathbb{A}_k^1|$ with 
$X=\mathbb{P}_k^1$ in the notation of 
Theorem \ref{prodf}. 
We simply write $\omega$ for $\omega|_T$. 
Then we have 
\begin{equation*}\label{detep}
 \det \tau_{\psi_0} \bigl( \Fr (1) \bigr) = 
 (-1)^{p^e} \varepsilon (\bP_k^1 ,j_! \pi^* \cL_{\psi_0} ) 
 = (-1)^p 
 q \varepsilon_{\psi_0} (T, j_! V_{\xi} , \omega) 
\end{equation*}
by Theorem \ref{prodf}. 
Since $\xi$ is a ramified character, 
we have $j_! V_{\xi} \simeq j_\ast V_{\xi}$. 
Hence, we obtain 
\begin{equation*}\label{DepLep}
 \varepsilon_{\psi_0} (T, j_! V_{\xi} , \omega) 
 = 
 \varepsilon_{\psi_0} (T, j_* V_{\xi} , \omega) 
 =
 q^{-1} \varepsilon (\xi \omega_{-\frac{1}{2}}, \Psi_{\omega}) 
\end{equation*}
by \eqref{locep}. 
Since $\omega=-t^{-2} dt$ on $T$, 
we have 
\begin{equation*}\label{epep}
 \varepsilon (\xi \omega_{-\frac{1}{2}}, \Psi_{\omega})
 = 
 (\xi \omega_{-\frac{1}{2}} ) (-t^{-1}) 
 \varepsilon (\xi \omega_{-\frac{1}{2}}, \Psi_{t^{-1}dt})
\end{equation*}
by \cite[23.5 Lemma 1]{BHLLCGL2}. 
We have 
\[
 \xi(-t^{-1})=\xi(-t^{p^e})=\xi(-t)^{p^e}=1, 
\] 
since $\Nr_{k((t))(y)/k((t))} (y)=1/t^{p^e+1}$. 
Hence we obtain 
\[ 
 (\xi \omega_{-\frac{1}{2}} ) (-t^{-1}) 
 \varepsilon (\xi \omega_{-\frac{1}{2}}, \Psi_{t^{-1}dt}) 
 = 
 q^{\frac{p^e}{2}} 
 \varepsilon (\xi, \Psi_{t^{-1}dt}) 
\]
by Lemma \ref{rswW}, 
since 
$\mathrm{rsw}(\xi,\Psi_{t^{-1}dt}) =s$ 
by $\mathrm{rsw}(\tilde{\xi},\Psi_{s^{-1}ds})=s$ 
and Proposition \ref{refSwres}.2. 
By Proposition \ref{chare}.2, 
we have $\varepsilon(\tilde{\xi},\Psi_{s^{-1}ds})=\tilde{\xi}(s)=1$, 
since the level of 
$\tilde{\xi}$ is $1$ and $\Nr_{k((s))(y)/k((s))}(y^{-1})=s$. 
Hence, we obtain 
\begin{equation*}\label{eqlamdel}
 \varepsilon (\xi, \Psi_{t^{-1}dt}) 
 = 
 \lambda (k((t))/k((s)) ,\Psi_{s^{-1}ds})^{-1} 
 \delta_{k((t))/k((s))} 
 \bigl( \mathrm{rsw}(\tilde{\xi},\Psi_{s^{-1}ds}) \bigr)
\end{equation*}
by Proposition \ref{reseps}. 
We have 
\begin{align*}
\lambda(k((t))/k((s)),\Psi_{s^{-1}ds})&=\begin{cases}
-\biggl(-\epsilon(p) \left(\frac{2}{p}\right)
\left(\frac{-1}{p}\right)^{\frac{p^e-1}{2}}\biggr)^f & \textrm{if $p \neq 2$}, \\
\left(\frac{q}{p^e+1}\right) & \textrm{if $p=2$}, 
\end{cases} \\
\delta_{k((t))/k((s))}\bigl( \mathrm{rsw}(\tilde{\xi},\Psi_{s^{-1}ds}) \bigr)
&=\begin{cases}
\left(\frac{-1}{q}\right)^{\frac{p^e+1}{2}} & \textrm{if $p \neq 2$},\\
\left(\frac{q}{p^e+1}\right) & \textrm{if $p=2$} 
\end{cases}
\end{align*}
by Lemma \ref{lconsttame} and Lemma \ref{deltatame} respectively. 
The claim follows from 
the above equalities.
\end{proof}

We simply write 
$\tau_{\zeta}$ for 
$\tau_{\zeta,1,1}$. 

\begin{prop}\label{centdet}
We have 
$\omega_{\pi_{\zeta,\chi,c}} =\det \tau_{\zeta,\chi,c}$. 
\end{prop}
\begin{proof}
By \eqref{Galconst} and 
\cite[(1)]{GallDet}, 
we have 
\begin{equation}\label{detres}
 \det \tau_{\zeta,\chi,c} =\delta_{E_{\zeta}/K}^{p^e} 
 (\det \tau_{n,\zeta,\chi,c} ) |_{K^{\times}}, 
\end{equation}
since $\delta_{E_{\zeta}/K} =\det (\Ind_{E_{\zeta}/K} 1)$ and 
the transfer homomorphism 
$W_K^{\mathrm{ab}} \to W_{E_{\zeta}}^{\rmab}$ 
is compatible with the natural inclusion 
$K^{\times} \to E_{\zeta}^{\times}$ 
under the Artin reciprocity maps. 
Hence, we may assume $\chi=1$ and $c=1$ by twist 
(\cf \eqref{Galconst}). 
Then it suffices to show $\det \tau_{\zeta} =1$. 
We see that 
$\det \tau_{\zeta}$ is unramified by \eqref{chartwi}, 
Lemma \ref{restQ}, Lemma \ref{dettau0}, 
Lemma \ref{deltatame} and \eqref{detres}. 

If $p$ and $n'$ are odd, 
we have 
\begin{align*}
 \det \tau_{\zeta} (\varpi ) &= 
 \biggl( \frac{q}{n'} \biggr)^{p^e} 
 \biggl(-\epsilon(p)
 \biggl( \frac{2}{p} \biggr) 
 p^{\frac{p^e}{2}} \biggr)^{fn'} 
 \biggl( \biggl(-\epsilon(p) 
 \biggl( \frac{-2n'}{p} \biggr) 
 \biggr)^n 
 p^{-\frac{1}{2}} 
 \biggr)^{fn'p^e} \\ 
 &=\biggl( \biggl( \frac{p}{n'} \biggr) 
 \biggl( \frac{n'}{p} \biggr) 
 \epsilon(p)^{p^e n-1}
 \biggr)^{fn'} = 
 \biggl( \biggl( \frac{p}{n'} \biggr) 
 \biggl( \frac{n'}{p} \biggr) 
 (-1)^{\frac{p-1}{2} \frac{n'-1}{2}} 
 \biggr)^{fn'}=1 
\end{align*}
by \eqref{detres}, 
Lemma \ref{deltatame}.1 and Lemma \ref{detFr}. 
We see that 
$\det \tau_{\zeta} (\varpi )=1$ 
similarly 
also in the other case 
using \eqref{detres}, 
Lemma \ref{deltatame} and Lemma \ref{detFr}. 
\end{proof}

\section{Imprimitive field}\label{Indchar}
In this section, 
we construct a field extension 
$T_{\zeta}^{\mathrm{u}}$ of 
$E_{\zeta}$ such that 
$\tau_{n,\zeta}|_{W_{T_{\zeta}^{\mathrm{u}}}}$ 
is an induction of a character. 
We call $T_{\zeta}^{\mathrm{u}}$ an imprimitive field of $\tau_{n,\zeta}$, 
since $\tau_{n,\zeta}|_{W_{T_{\zeta}^{\mathrm{u}}}}$ 
is not primitive. 

\subsection{Construction of character}
In this subsection, 
we construct subgroups 
$R \subset Q' \subset Q \rtimes \bZ$ 
and a character $\phi_n$ of $R$. 
In the next subsection, 
we will see that 
$\tau_n |_{Q'} \simeq \Ind_{R}^{Q'} \phi_n$. 
Our imprimitive field $T_{\zeta}^{\mathrm{u}}$ 
will correspond to the subgroup 
$Q' \subset Q \rtimes \bZ$. 

Let $e_0$ be the positive integer 
such that 
$e_0 \in 2^{\mathbb{N}}$ and $e/e_0$ is odd. 

\begin{lem}\label{TrdFr}
Assume $p \neq 2$. 
Then we have 
$\Tr \tau_{\psi_0} \bigl( 
\Fr (2e_0 ) \bigr) =p^{e_0}$. 
\end{lem}
\begin{proof}
For $a \in k^{\mathrm{ac}}$ and $b \in \bF_{p^{2e_0}}$ such that 
$a^p-a=b^{p^e+1}$, 
we have that 
\begin{equation}\label{Ar2m0}
 a^{p^{2e_0}} -a =\Tr_{\bF_{p^{2e_0}}/\bF_p} (b^{p^e+1}). 
\end{equation}
By \eqref{Ar2m0} and the Lefschetz trace formula, 
we see that 
\begin{align*}
 \Tr \tau_{\psi_0} \bigl( \Fr (2e_0) \bigr) &= 
 -\sum_{b \in \mathbb{F}_{p^{2e_0}}} 
 (\psi_0 \circ \Tr_{\mathbb{F}_{p^{e_0}}/\bF_p} ) 
 \bigl( 
 \Tr_{\mathbb{F}_{p^{2e_0}}/\mathbb{F}_{p^{e_0}}} 
 (b^{p^{e_0} +1}) \bigr) \\
 &= -1 - (p^{e_0}+1) \sum_{x \in \mathbb{F}_{p^{e_0}}^{\times}} 
 (\psi_0 \circ \Tr_{\mathbb{F}_{p^{e_0}}/\bF_p} ) (x) 
 =p^{e_0 } 
\end{align*} 
using $(p^e +1, p^{2e_0 } -1) =p^{e_0} +1$. 
\end{proof}

\begin{cor}\label{TrdFrn}
Assume $p \neq 2$. 
Then we have 
$\Tr \tau_n \bigl( 
\Fr (2e_0 ) \bigr) =(-1)^{ne_0(p-1)/2}$. 
\end{cor}
\begin{proof}
This follows from \eqref{chartwi} and 
Lemma \ref{TrdFr}. 
\end{proof}

Let $n_0$ be the biggest integer such that 
$2^{n_0}$ divides $p^{e_0} +1$. 
We take $r \in k^{\mathrm{ac}}$ such that 
$r^{2^{n_0}} =-1$. 
We define a subgroup $R_0$ of $Q_0$ by 
\[
 R_0 =\{ (1,b,c) \in Q_0 \mid b^{p^e} -r b =0 \}. 
\]
\begin{lem}\label{conjstab}
{\rm 1.}
If $p \neq 2$, the action of 
$2 e_0 \mathbb{Z} \subset \mathbb{Z}$ on $Q$ 
stabilizes $R_0$. \\ 
{\rm 2.}
If $p=2$, the action of $\bom{g}$ on 
$Q \rtimes \bZ$ by conjugation 
stabilizes $R_0$. 
\end{lem}
\begin{proof}
The first claim follows from $r^{p^{2e_0} -1}=1$. 
We can see the second claim easily using 
\eqref{ginv}. 
\end{proof}
We put 
\[
 Q'=
 \begin{cases}
 Q_0 \rtimes (2 e_0 \mathbb{Z}) & \quad \textrm{if $p \neq 2$,} \\ 
 Q_0 \rtimes \bZ & \quad \textrm{if $p = 2$,} 
 \end{cases}
 \quad 
 R= 
 \begin{cases}
 R_0 \rtimes (2 e_0 \mathbb{Z}) & \quad \textrm{if $p \neq 2$,} \\ 
 R_0 \cdot \langle \bom{g} \rangle 
 & \quad \textrm{if $p = 2$}
 \end{cases}
\]
as subgroups of $Q \rtimes \bZ$, 
which are well-defined by Lemma \ref{conjstab}. 
We are going to construct 
a character $\phi_n$ of $R$ in this subsection. 
Then, we will show that 
$\tau_n |_{Q'} \simeq \Ind_{R}^{Q'} \phi_n$ 
in the next subsection. 

First, we consider the case where $p$ is odd. 
We define a homomorphism 
$\phi_n \colon R \to \bC^{\times}$ 
by 
\begin{gather}\label{phiodd}
\begin{aligned}
 \phi_n \bigl( ((1,b,c),0) \bigr) &= 
 \psi_0 \biggr( c -\frac{1}{2} \sum_{i=0}^{e -1} 
 (r b^2 )^{p^i} \biggr) \quad 
 \textrm{for} \ \ (1,b,c) \in R_0, \\ 
 \phi_n \bigl( \Fr(2e_0) \bigr) 
 &= (-1)^{n e_0 \frac{p-1}{2}}.  
\end{aligned}
\end{gather}
Then $\phi_n$ extends the character $\psi_0$ of $F$. 

Next, we consider the case where $p=2$. 
We define an abelian group $R_0'$ as 
\[
 R_0' =\Bigl\{ (b, c) \Bigm| b \in \mathbb{F}_{2}, \ 
 c \in \mathbb{F}_{2^{2e}} , \ c^{2^e} -c =b \Bigr\}
\]
with the multiplication given by 
\[
 (b_1, c_1) \cdot (b_2, c_2) =
 (b_1 +b_2 , c_1 +c_2 +b_1 b_2 ). 
\]
We define $\phi \colon R_0 \to R_0'$ by 
\begin{equation*}
 \phi ((1,b,c) ) = 
 \biggl( \Tr_{\mathbb{F}_{2^e}/\mathbb{F}_2} (b), 
 c +\sum_{0 \leq i < j \leq e -1} 
 b^{2^i +2^j} \biggr) 
 \quad 
 \textrm{for} \ \ (1,b,c) \in R_0, 
\end{equation*}
which is a homomorphism by
\begin{equation}\label{trtr}
\Tr_{\mathbb{F}_{2^e}/\mathbb{F}_2}(b)
\Tr_{\mathbb{F}_{2^e}/\mathbb{F}_2}(b')=\Tr_{\mathbb{F}_{2^e}/\mathbb{F}_2}
(bb')+\sum_{0 \leq i<j \leq e-1} (b^{2^i}b'^{2^j}+b'^{2^i}b^{2^j})
\end{equation}
for $b,b' \in \mathbb{F}_{2^e}$. 
Let $b_0 \in \mathbb{F}_{2^{2e}}$ be as before Lemma \ref{2trFr}. 
Let $F'$ be the kernel of the homomorphism 
\[
 \bF_{2^e} \to \bF_{2} ;\ c \mapsto 
 \Tr_{\bF_{2^e} /\bF_2} \bigl( 
 (b_0 +b_0^{2^{e}} ) c \bigr). 
\] 
We put $R_0'' =R_0' /F'$, 
where we consider $F'$ as a subgroup of $R_0'$ by 
$c \mapsto (0,c)$. 
Then $R_0''$ is a cyclic group of 
order $4$. 
We write $\bar{g}(b,c)$ for the image of 
$(b,c) \in R_0'$ under the projection 
$R_0' \to R_0''$. 
Let $\phi' \colon R_0 \to R_0''$ 
be the composite of $\phi$ and the projection 
$R_0' \to R_0''$. 
We put 
\begin{equation}\label{defst}
 s= \sum_{i=0}^{e-1} b_0^{2^i}, \quad  
 t= \Tr_{\mathbb{F}_{2^{2e}}/\mathbb{F}_{2^e}}(b_0). 
\end{equation}
We have $s^2+s=t$ and $\Tr_{\mathbb{F}_{2^e}/\mathbb{F}_2}(t)=\Tr_{\mathbb{F}_{2^{2e}}/\mathbb{F}_2}(b_0)=1$. 
We have 
\[
\biggl( 
 1, s^2 + 
 \sum_{0 \leq i < j \leq e-1} t^{2^i +2^j} \biggr) \in R'_0, 
\]
which is of order $4$. 
The element 
$\bar{g}\bigl( 
 1, s^2 + 
 \sum_{0 \leq i < j \leq e-1} t^{2^i +2^j} \bigr)$ is a generator of 
 $R''_0$, because $2\bar{g}\bigl( 
 1, s^2 + \sum_{0 \leq i < j \leq e-1} t^{2^i +2^j} \bigr)=\bar{g}(0,1) \neq 0$. 
Let 
$\tilde{\psi}_0 \colon R_0'' \to \bC^{\times}$ be 
the faithful character satisfying 
\[
 \tilde{\psi}_0 \biggl( \bar{g} \biggl( 
 1, s^2 + 
 \sum_{0 \leq i < j \leq e-1} t^{2^i +2^j} \biggr) 
 \biggr) = -\sqrt{-1}. 
\]
We define a homomorphism 
$\phi_n \colon R \to \bC^{\times}$ 
by 
\begin{gather}\label{phieven}
\begin{aligned}
 \phi_n \bigl( ((1,b,c),0) \bigr) &= 
 (\tilde{\psi}_0 \circ \phi') 
 ((1,b,c)) 
 \quad 
 \textrm{for} \ \ (1,b,c) \in R_0, \\ 
 \phi_n (\bom{g})  
 &= (-1)^{\frac{n(n-2)}{8}} 
 \frac{-1+\sqrt{-1}}{\sqrt{2}},  
\end{aligned}
\end{gather}
which is a character of order $8$. 
Then $\phi_n$ extends the character $\psi_0$ of $F$. 

\subsection{Induction of character}
\begin{lem}\label{taunind}
We have 
$\tau_n |_{Q'} \simeq 
 \Ind_{R}^{Q'} \phi_n$. 
\end{lem}
\begin{proof}
We write $\tilde{\psi}_n$ for 
$\phi_n|_{R_0}$. 
We know that 
$\tau_n |_{Q_0} \cong \Ind_{R_0}^{Q_0} \tilde{\psi}_n$ 
by Proposition \ref{prop:rhopsi}, 
since $R_0$ is an abelian group such that 
$2 \dim_{\bF_p} (R_0/F)=\dim_{\bF_p} (Q_0/F)$. 

First, we consider the case where $p$ is odd. 
The claim for general $f$ follows from the 
claim for $f=1$ by the restriction. 
Hence, we may assume that $f=1$. 

If 
$\tilde{\psi} \in R_0^{\vee}$ satisfies 
$\tilde{\psi}|_F =\psi_0$, 
then we have 
$\tau_n |_{Q_0} \cong \Ind_{R_0}^{Q_0} \tilde{\psi}$ 
by Proposition \ref{prop:rhopsi}, 
and obtain an injective homomorphism 
$\tilde{\psi} \hookrightarrow 
 \tau_n |_{R_0}$ as 
representations of $R_0$ 
by Frobenius reciprocity. 
Hence we have a decomposition 
\begin{equation}\label{R0dec}
 \tau_n |_{R_0} = 
 \bigoplus_{\tilde{\psi} \in R_0^{\vee},\, \tilde{\psi}|_F =\psi_0} 
 \tilde{\psi}, 
\end{equation} 
since the number of 
$\tilde{\psi} \in R_0^{\vee}$ such that 
$\tilde{\psi}|_F =\psi_0$ is $p^e$. 

We put 
$\overline{R}_0 =\{ b \in k^{\mathrm{ac}} \mid b^{p^e} -rb=0 \}$. 
The $\tilde{\psi}_n$-component in \eqref{R0dec} 
is the unique component that is stable by the action of 
$((1,0,0),2e_0 )$, since the homomorphism 
\[
 \overline{R}_0 \to \overline{R}_0 ;\ 
 b \mapsto b^{p^{2e_0}} -b 
\]
is an isomorphism. 
Hence, we have a non-trivial homomorphism 
$\phi_n \to \tau_n |_{R}$ by Corollary \ref{TrdFrn}. 
Then  we have 
a non-trivial homomorphism 
$\Ind_{R}^{Q'} \phi_n \to \tau_n |_{Q'}$ 
by Frobenius reciprocity. 
The representation $\tau_n |_{Q'}$ 
is irreducible by Corollary \ref{cor:irrQ0}. 
Then we obtain the claim, since $[Q':R]=p^e$.

Next we consider the case where $p=2$. 
Then it suffices to show that 
\[
 \Tr (\Ind_{R}^{Q'} 
 \phi_n ) ( \bom{g}^{-1} ) 
 = -(-1)^{\frac{n(n-2)}{8}} \sqrt{2} 
\]
by \eqref{chartwi} and Proposition \ref{charrep}. 
We have a decomposition 
\begin{equation}\label{R0dec2}
 (\Ind_{R}^{Q'} \phi_n ) |_{R_0} = 
 \bigoplus_{\phi \in R_0^{\vee},\, \phi|_F =\psi_0} 
 \phi. 
\end{equation}
Let $\tilde{\psi}'_n$ be the twist of $\tilde{\psi}_n$ by the character 
\[
 R_0 \to \overline{\bQ}_{\ell}^{\times} ;\ 
 (1,b,c) \mapsto \psi_0 \bigl( 
 \Tr_{\bF_{2^e}/\bF_2} (b) \bigl). 
\]
Then only the $\tilde{\psi}_n$-component and 
the $\tilde{\psi}'_n$-component 
in \eqref{R0dec} 
are stable by the action of 
$((1,b_0,c_0),1)$, since the image of 
the homomorphism 
\[
 \bF_{2^e} \to \bF_{2^e} ;\ 
 b \mapsto b^2 -b 
\]
is equal to $\Ker \Tr_{\bF_{2^e}/\bF_2}$. 
The action of $\Fr(e)$ 
permutes the $\tilde{\psi}_n$-component and 
the $\tilde{\psi}'_n$-component. 
Hence, $\bom{g}$ 
acts on the $\tilde{\psi}'_n$-component by 
$\phi_n (\bom{g})$ times 
\begin{align*}
 \phi_n \Bigl( \Fr (e)^{-1} 
 \bom{g} \Fr (e) 
 \bom{g}^{-1} \Bigr) 
 =\phi_n \biggl( \Bigl( \bigl( 1,t ,
 c_0 + c_0^{2^e} +\sum_{i=0}^{e -1} 
 (b_0^{2^e +1} + b_0^{2^{e+1}})^{2^i} \bigr),0 \Bigr) \biggr) 
 = \sqrt{-1}.   
\end{align*}
Hence we have 
\[
 \Tr (\Ind_{R}^{Q'} 
 \phi_n ) ( \bom{g}^{-1} ) 
 = \bigl( 1 -\sqrt{-1} \bigr) 
 \phi_n ( \bom{g}^{-1} ) 
 =-(-1)^{\frac{n(n-2)}{8}} \sqrt{2}. 
\]
\end{proof}

We use the notations in \eqref{extgen}. 
We set 
$T_{\zeta}=E_{\zeta} (\alpha_{\zeta} )$, 
$M_{\zeta}=T_{\zeta}(\beta_{\zeta})$ and 
$N_{\zeta}=M_{\zeta}(\gamma_{\zeta})$. 
Let $f_0$ be the positive integer 
such that 
$f_0 \in 2^{\mathbb{N}}$ and $f/f_0$ is odd. 
We put 
\[
 N = 
 \begin{cases}
  2e_0/f_0 \quad &\textrm{if $p \neq 2$ and $f_0 \mid 2e_0$,} \\ 
  1 \quad &\textrm{otherwise.} 
 \end{cases}
\]
Let $K^{\mathrm{ur}}$ 
be the maximal unramified extension 
of $K$ in $K^{\mathrm{ac}}$. 
Let $K^{\mathrm{u}} \subset K^{\mathrm{ur}}$ 
be the unramified extension of degree $N$ over $K$. 
Let $k_N$ be the residue field of $K^{\mathrm{u}}$. 
For a finite 
field extension $L$ of $K$ in $K^{\mathrm{ac}}$, 
we write $L^{\mathrm{u}}$ 
for the composite field of $L$ and 
$K^{\mathrm{u}}$ in $K^{\mathrm{ac}}$. 
For $a \in k^{\mathrm{ac}}$, 
we write $\hat{a} \in \mathcal{O}_{K^{\mathrm{ur}}}$ 
for the Teichm\"{u}ller lift of $a$. 
We put 
\begin{equation}\label{ddt}
 \delta'_{\zeta}=
 \begin{cases}
 \beta_{\zeta}^{p^e} -\hat{r} \beta_{\zeta} \quad & \textrm{if $p \neq 2$},\\ 
 \beta_{\zeta}^{2^e} -\beta_{\zeta} 
 +\sum_{i=0}^{e-1} \hat{b}_0^{2^i} \quad & \textrm{if $p=2$}, 
 \end{cases} 
 \quad 
 \epsilon_1 = 
 \begin{cases}
  0 \quad & \textrm{if $p \neq 2$}, \\ 
  1 \quad & \textrm{if $p=2$}. 
 \end{cases}
\end{equation}
Then we have 
$\delta_{\zeta}'^{p^e} -\hat{r}^{-1} \delta'_{\zeta} \equiv 
 - \alpha_{\zeta}^{-1} 
 +\epsilon_1 \mod \fp_{T^{\mathrm{u}}_{\zeta}(\delta'_{\zeta})}$. 
We take 
$\delta_{\zeta} \in T^{\mathrm{u}}_{\zeta}(\delta'_{\zeta})$ such that 
\begin{equation}\label{delz}
 \delta_{\zeta}^{p^e} -\hat{r}^{-1} \delta_{\zeta} = 
 - \alpha_{\zeta}^{-1} 
 +\epsilon_1 , \quad 
 \delta_{\zeta} \equiv \delta'_{\zeta} 
 \mod \fp_{T^{\mathrm{u}}_{\zeta}(\delta'_{\zeta})}. 
\end{equation}
We put 
$M'^{\mathrm{u}}_{\zeta}=T^{\mathrm{u}}_{\zeta}(\delta_{\zeta})$. 
The image of $\Theta_{\zeta}|_{W_{M'^{\mathrm{u}}_{\zeta}}}$ is contained in 
$R$. 
Let 
$\xi_{n,\zeta} \colon W_{M'^{\mathrm{u}}_{\zeta}} 
 \to \bC^{\times}$ 
be the composite of the restrictions 
$\Theta_{\zeta}|_{W_{M'^{\mathrm{u}}_{\zeta}}}$ 
and 
$\phi_n|_R$. 
By the local class field theory, 
we regard 
$\xi_{n,\zeta}$ as a character of 
${M'^{\mathrm{u}}_{\zeta}}^{\times}$. 
\begin{prop}\label{taunzind}
We have 
$\tau_{n,\zeta}|_{W_{T_{\zeta}^{\mathrm{u}}}} \simeq 
 \mathrm{Ind}_{M'^{\mathrm{u}}_{\zeta}/T^{\mathrm{u}}_{\zeta}} 
 \xi_{n,\zeta}$. 
\end{prop}
\begin{proof}
This follows from Lemma \ref{taunind}. 
\end{proof}

\begin{rem}
Our imprimitive field is different from 
that in \cite[5.1]{BHLepi}. 
In our case, 
$T_{\zeta}^{\mathrm{u}}$ need not be normal over $K$. 
This choice is technically important 
in our proof of the main result. 
\end{rem}

\subsection{Study of character}\label{Study}
In this subsection, we study 
the character $\xi_{n,\zeta}$ in detail. 

Assume that $\ch K=p$ and 
$f=1$ 
in this subsection. 
We will use results in this subsection to compute 
the epsilon factor of $\xi_{n,\zeta}$ later 
after a reduction to the case where 
$\ch K=p$ and $f=1$. 
By \eqref{extgen}, \eqref{ddt}, \eqref{delz} and $\ch K=p$, we have that $\delta_{\zeta} = \delta'_{\zeta}$.

\subsubsection{Odd case}
Assume $p \neq 2$ in this subsubsection. 
We put 
\begin{equation}\label{tzeta}
 \theta_{\zeta} = 
 \gamma_{\zeta} +\frac{1}{2} \sum_{i=0}^{e -1} 
 (r \beta_{\zeta}^2 )^{p^i}. 
\end{equation}
Since $r^{p^{e_0}+1}=-1$ and 
$(p^e+1)/(p^{e_0}+1)$ is an odd integer, we have $r^{p^e+1}=-1$. 
Then we have 
\begin{equation}\label{thetad}
 \theta^p_{\zeta} - \theta_{\zeta} = 
 \beta_{\zeta}^{p^e +1} -\frac{1}{2 r} 
 (\beta_{\zeta}^{2p^e} +r^2 \beta_{\zeta}^2 ) 
 = -\frac{1}{2 r} 
 (\beta_{\zeta}^{2p^e} -2 r \beta_{\zeta}^{p^e +1} 
 +r^2 \beta_{\zeta}^2 ) 
 = -\frac{1}{2 r} \delta_{\zeta}^2 . 
\end{equation}
We put 
$N'^{\mathrm{u}}_{\zeta} =M'^{\mathrm{u}}_{\zeta}(\theta_{\zeta})$. 
Let $\xi'_{n,\zeta}$ 
be the twist of 
$\xi_{n,\zeta}$ 
by the unramified character 
\[
 W_{M'^{\mathrm{u}}_{\zeta}} \to \bC^{\times};\ 
 \sigma \mapsto  
 \sqrt{-1}^{n  n_{\sigma} \frac{p-1}{2}},  
\]
where $n_{\sigma}$ is as before \eqref{hom}. 
\begin{lem}\label{oddfac}
If $p \neq 2$, then 
$\xi'_{n,\zeta}$ factors through 
$\Gal (N'^{\mathrm{u}}_{\zeta}/M'^{\mathrm{u}}_{\zeta})$. 
\end{lem}
\begin{proof}
Let $\sigma \in \Ker \xi'_{n,\zeta}$. 
Recall that $a_{\sigma}, b_{\sigma}, c_{\sigma}$ 
are defined in \eqref{abcdef}. 
Then we have 
$(\bar{a}_{\sigma}, \bar{b}_{\sigma}, 
 \bar{c}_{\sigma}) \in R_0$ and 
\[
 \bar{c}_{\sigma} -\frac{1}{2} \sum_{i=0}^{e -1} 
 (r \bar{b}_{\sigma}^2 )^{p^i} =0
\] 
by \eqref{phiodd}. 
Hence, we see that 
\begin{align*}
 \sigma (\theta_{\zeta}) - \theta_{\zeta} &= 
 c_{\sigma} - 
 \sum_{i=0}^{e-1} 
 \bigl( r b_{\sigma} (\beta_{\zeta} +b_{\sigma} ) \bigr)^{p^i} 
 +\frac{1}{2} \sum_{i=0}^{e-1} 
 \Bigl( r \bigl( 
 (\beta_{\zeta} +b_{\sigma})^2 
 -\beta_{\zeta}^2 \bigr) \Bigr)^{p^i} \\ 
 &= c_{\sigma} -\frac{1}{2} \sum_{i=0}^{e -1} 
 (r b_{\sigma}^2 )^{p^i} 
 \equiv 0 \mod \mathfrak{p}_{N'^{\mathrm{u}}_{\zeta}}  
\end{align*}
by \eqref{abcdef}.
Therefore, we obtain the claim by $\sigma(\delta_{\zeta})=\delta_{\zeta}$ and \eqref{thetad}. 
\end{proof}

\subsubsection{Even case}
Assume $p=2$ in this subsubsection. 
Let 
$\xi'_{n,\zeta}$ 
be the twist of 
$\xi_{n,\zeta}$ 
by the character 
\begin{equation}\label{dya}
 W_{M'^{\mathrm{u}}_{\zeta}} \to \bC^{\times};\ 
 \sigma \mapsto 
 \biggl( 
 (-1)^{\frac{n(n-2)}{8}} \frac{-1+\sqrt{-1}}{\sqrt{2}}
 \biggl)^{n_{\sigma}}. 
\end{equation}
We take $b_1 ,b_2 \in k^{\mathrm{ac}}$ such that 
\begin{equation}\label{bb}
 b_1^2 -b_1 =s, 
 \quad 
 b_2^2 -b_2 =t 
 \Bigl( b_1^2 +\sum_{i=0}^{e-1} ( b_1 s )^{2^i} \Bigr). 
\end{equation}
We put 
\begin{align}
 \eta_{\zeta} &= 
 \sum_{i=0}^{e -1} \beta_{\zeta}^{2^i} 
 +b_1 , \quad 
 \gamma_{\zeta}' = \gamma_{\zeta} 
 + \sum_{0 \leq i < j \leq e-1} \beta_{\zeta}^{2^i+2^j}, \label{defetagam'}\\ 
 \theta'_{\zeta} &= 
 \sum_{i=0}^{e -1} ( t \gamma_{\zeta}' )^{2^i} + 
 \sum_{0 \leq i \leq j \leq e-2}  
 t^{2^i} 
 (\delta_{\zeta} \eta_{\zeta})^{2^j} 
 +\sum_{0 \leq j < i \leq e-1} 
 t^{2^i} 
 (b_1 \delta_{\zeta} +s \eta_{\zeta})^{2^j} 
 +b_1^2 \eta_{\zeta} 
 + b_2 . \label{deftheta'}
\end{align}

\begin{lem}\label{tileq}
We have 
$\eta_{\zeta}^2 -\eta_{\zeta} 
 = \delta_{\zeta}$ 
and 
$\theta'^2_{\zeta} -\theta'_{\zeta} 
 = (\delta_{\zeta} \eta_{\zeta})^{2^{e-1}}$. 
\end{lem}
\begin{proof}
We can check the first claim easily. 
We show the second claim. 
We use $P_m$ in Subsubsection \ref{adddis}. 
We have 
\begin{equation}\label{P2theta}
 P_2 ( \gamma_{\zeta}' ) 
 = (\beta_{\zeta}^{2^e} -\beta_{\zeta}) 
 \sum_{i=0}^{e-1} \beta_{\zeta}^{2^i} + 
 \beta_{\zeta}^2 
 = ( \delta_{\zeta} -s ) 
 (\eta_{\zeta} -b_1 ) + \beta_{\zeta}^2 . 
\end{equation}
Hence, we have 
\begin{equation}\label{P2e}
 P_{2^e} ( \gamma_{\zeta}' ) 
 = \sum_{i=0}^{e-1} \bigl( 
 ( \delta_{\zeta} -s ) 
 (\eta_{\zeta} -b_1 ) \bigr)^{2^i} + 
 (\eta_{\zeta} -b_1 )^2 . 
\end{equation}
By $b_1^4+b_1=s^2+s=t$ and 
$\eta_{\zeta}^2 -\eta_{\zeta} =\delta_{\zeta}$,
we have 
\[
(b_1^2 \eta_{\zeta})^2+b_1^2 \eta_{\zeta}=t \eta_{\zeta}^2+b_1 \eta_{\zeta}^2
+b_1^2 \eta_{\zeta}=
t \eta_{\zeta}^2+b_1(\eta_{\zeta}^2+\eta_{\zeta})+s \eta_{\zeta}
=t \eta_{\zeta}^2+b_1 \delta_{\zeta}+s \eta_{\zeta}. 
\]
Hence, by using $\sum_{i=1}^{e-1} t^{2^i}=1-t$ and $t \in \mathbb{F}_{2^e}$, 
we have 
\begin{align*}
\theta'^2_{\zeta} -\theta'_{\zeta} 
 &=t P_{2^e}(\gamma'_{\zeta})+t  \sum_{i=0}^{e-1} (\delta_{\zeta} \eta_{\zeta} 
 +b_1 \delta_{\zeta} +s\eta_{\zeta})^{2^i}+(\delta_{\zeta} \eta_{\zeta})^{2^{e-1}}
 +t \eta_{\zeta}^2 +b_2^2 -b_2 \\
 &=t\left(\sum_{i=0}^{e-1}(b_1 s)^{2^i}+\eta_{\zeta}^2+b_1^2\right)
 +(\delta_{\zeta} \eta_{\zeta})^{2^{e-1}}
 +t \eta_{\zeta}^2 +b_2^2 -b_2=(\delta_{\zeta} \eta_{\zeta})^{2^{e-1}}, 
\end{align*}
where we use \eqref{P2e} at the second equality and \eqref{bb} at the third one. 
\end{proof} 

We take $\theta_{\zeta} \in K^{\mathrm{ac}}$ 
such that 
$\theta'_{\zeta} = \theta_{\zeta}^{2^{e-1}}$. 
Then we have 
$\theta^2_{\zeta} -\theta_{\zeta} 
 = \delta_{\zeta} \eta_{\zeta}$. 
We put 
$N'^{\mathrm{u}}_{\zeta} =M'^{\mathrm{u}}_{\zeta} 
 (\eta_{\zeta}, \theta_{\zeta})$, 
which is a cyclic extension of 
$M'^{\mathrm{u}}_{\zeta}$ of order $4$ 
by Lemma \ref{tileq}. 

\begin{lem}\label{lem:factor}
The character 
$\xi'_{n,\zeta}$ factors through 
$\Gal (N'^{\mathrm{u}}_{\zeta}/M'^{\mathrm{u}}_{\zeta})$. 
\end{lem}
\begin{proof}
Let $\sigma \in \Ker \xi'_{n,\zeta}$. 
We take 
$\sigma_1, \sigma_2 \in \Ker \xi'_{n,\zeta}$ 
such that 
$\sigma =\sigma_1 \sigma_2^{-n_{\sigma}}$, 
$\sigma_1 \in I_{M'^{\mathrm{u}}_{\zeta}}$ and 
$\Theta_{\zeta} (\sigma_2 )=((1,b_0,c_0),-1)$. 
Then we have 
$(\bar{a}_{\sigma_1}, \bar{b}_{\sigma_1}, 
 \bar{c}_{\sigma_1}) \in R_0$, 
$\Tr_{\bF_{2^e}/\bF_2} (\bar{b}_{\sigma_1})=0$ and 
\begin{equation}\label{kercond2}
 \Tr_{\bF_{2^e}/\bF_2} \biggl( t 
 \Bigl( \bar{c}_{\sigma_1} 
 +\sum_{0 \leq i < j \leq e-1} 
 \bar{b}_{\sigma_1}^{2^i +2^j} \Bigr) \biggr)=0
\end{equation} 
by \eqref{phieven}. 
It suffices to show that 
$\sigma_i (\eta_{\zeta}) = \eta_{\zeta}$ 
and 
$\sigma_i (\theta'_{\zeta}) = \theta'_{\zeta}$ 
for $i=1,2$. 

We have 
\begin{align*}
\sigma_1(\eta_{\zeta})-\eta_{\zeta} &\equiv \sum_{i=0}^{e-1} b_{\sigma_1}^{2^i} \equiv 0 \mod \mathfrak{p}_{N'^{\mathrm{u}}_{\zeta}}, \\
\sigma_2(\eta_{\zeta})-\eta_{\zeta} & \equiv 
\sum_{i=0}^{e-1} b_0^{2^i}+b_1^2-b_1 \equiv 0 \mod 
\mathfrak{p}_{N'^{\mathrm{u}}_{\zeta}}
\end{align*}
by $\Tr_{\bF_{2^e}/\bF_2} (\bar{b}_{\sigma_1})=0$
and $b_1^2-b_1=s$. 
By Lemma \ref{tileq}, we have 
$\sigma_i(\eta_{\zeta})-\eta_{\zeta} \in \mathbb{F}_2$ for $i=1,2$. 
Hence, we have $\sigma_i(\eta_{\zeta})=\eta_{\zeta}$ for 
$i=1,2$. 
We have 
\[
 \sigma_1(\theta'_{\zeta})-\theta'_{\zeta}=\sum_{i=0}^{e-1}
\left( t (\sigma_1(\gamma'_{\zeta})-\gamma'_{\zeta})\right)^{2^i}. 
\]
Further, we have 
\begin{align*}
\sigma_1(\gamma'_{\zeta})-\gamma'_{\zeta}
& \equiv c_{\sigma_1}+\sum_{i=0}^{e-1}(b_{\sigma_1})^{2^{i+1}}+\sum_{i=0}^{e-1} b_{\sigma_1}^{2^i}\sum_{i=0}^{e-1} \beta_{\zeta}^{2^i}
+\sum_{0 \leq i<j \leq e-1} b_{\sigma_1}^{2^i+2^j} \\
& \equiv c_{\sigma_1}+\sum_{0 \leq i<j \leq e-1} b_{\sigma_1}^{2^i+2^j} \mod 
\mathfrak{p}_{N'^{\mathrm{u}}_{\zeta}}, 
\end{align*}
where we use \eqref{abcdef} 
and $\bar{b}_{\sigma_1} \in \bF_{2^e}$ at the first equality, 
and use $\Tr_{\bF_{2^e}/\bF_2} (\bar{b}_{\sigma_1})=0$ 
at the second one. 
This implies 
$\sigma_1(\theta'_{\zeta}) \equiv \theta'_{\zeta} \mod \mathfrak{p}_{N'^{\mathrm{u}}_{\zeta}}$ by \eqref{kercond2}.
By a similar argument as above using Lemma \ref{tileq}, 
we obtain $\sigma_1(\theta'_{\zeta})=\theta'_{\zeta}$. 

It remains to show 
$\sigma_2 (\theta'_{\zeta}) = \theta'_{\zeta}$. 
Using \eqref{P2theta} 
and $\Tr_{\mathbb{F}_{2^e}/\mathbb{F}_2}(t)=1$, 
we see that 
\begin{align}
 \sum_{i=0}^{e -1} ( t \gamma_{\zeta}' )^{2^i} = 
 \gamma_{\zeta}' 
 +
 \sum_{1 \leq i \leq j \leq e -1} 
 t^{2^j} \beta_{\zeta}^{2^i}  + 
 \sum_{0 \leq i <j \leq e -1} t^{2^j} 
 \bigl( ( \delta_{\zeta} -s )
 (\eta_{\zeta} -b_1) \bigl)^{2^i}
 . \label{b0td}
\end{align}
We put 
\[
\gamma_{\zeta}'' =\gamma_{\zeta}' 
 +\sum_{1 \leq i \leq j \leq e -1} t^{2^j} 
 \beta_{\zeta}^{2^i}. 
\] 
By $c_0^2+c_0=b_0^{2^e+1}$ and $t=b_0+b_0^{2^e}$
(\cf \eqref{c_0b_0}, \eqref{defst}),
we have 
\begin{align*}
\sigma_2(\gamma_{\zeta})-\gamma_{\zeta} 
\equiv c_0+\sum_{i=0}^{e-1}(b_0^{2^e}(\beta_{\zeta}+b_0))^{2^i} 
\equiv 
c_0^{2^e}+\sum_{i=0}^{e-1}((b_0+t) \beta_{\zeta})^{2^i}
\mod \mathfrak{p}_{N'^{\mathrm{u}}_{\zeta}}. 
\end{align*}
Then we have 
\[
\sigma_2(\gamma'_{\zeta})-\gamma'_{\zeta}
\equiv c_0^{2^e}+s(\eta_{\zeta}-b_1)+\sum_{i=0}^{e-1}(t \beta_{\zeta})^{2^i}
+\sum_{0 \leq i<j \leq e-1} b_0^{2^i+2^j} 
\mod \mathfrak{p}_{N'^{\mathrm{u}}_{\zeta}} 
\]
by \eqref{defst} and \eqref{defetagam'}. 
Hence, we have 
\begin{align*}
\sigma_2(\gamma''_{\zeta})-\gamma''_{\zeta}
&
 \equiv \sigma_2(\gamma'_{\zeta})-\gamma'_{\zeta} 
 +\sum_{1 \leq i \leq j \leq e-1} 
 t^{2^{j+1}} (\beta_{\zeta} +b_0)^{2^i} 
 - \sum_{1 \leq i \leq j \leq e-1} 
 t^{2^j} \beta_{\zeta}^{2^i} 
 \\ 
&\equiv 
\sigma_2(\gamma'_{\zeta})-\gamma'_{\zeta}+t(\eta_{\zeta}-b_1)
+\sum_{i=0}^{e-1}(t \beta_{\zeta})^{2^i}+\sum_{1 \leq i<j \leq e}
b_0^{2^i} t^{2^j}  \\
& \equiv c_0^{2^e}+s^2(\eta_{\zeta}-b_1)+\sum_{0 \leq i<j \leq e-1} b_0^{2^i+2^j}+\sum_{1 \leq i<j \leq e}b_0^{2^i} t^{2^j} 
\mod \mathfrak{p}_{N'^{\mathrm{u}}_{\zeta}} 
\end{align*}
where we use 
\eqref{defetagam'} and $t \in \bF_{2^{e}}$ at the second equality 
and $s^2 +s=t$ at the last equality. 
We can check that 
\[
c_0^{2^e}+\sum_{0 \leq i<j \leq e-1} b_0^{2^i+2^j}+
\sum_{1 \leq i<j \leq e} b_0^{2^i} t^{2^j} =st 
\]
by \eqref{defgamma_0}, \eqref{defst} and 
$\Tr_{\mathbb{F}_{2^{2e}}/\mathbb{F}_2} (b_0 )=1$. 
As a result, we obtain 
\[
\sigma_2(\gamma''_{\zeta})-\gamma''_{\zeta} \equiv 
s^2 \eta_{\zeta} +b_1 s^2 +st \mod \mathfrak{p}_{N'^{\mathrm{u}}_{\zeta}}. 
\]
Hence, by \eqref{deftheta'} and \eqref{b0td}, 
we have 
\[
 \sigma_2 (\theta'_{\zeta}) - \theta'_{\zeta} 
 \equiv 
 \sum_{i=0}^{2^{e-1} + 2^{e-2}} d_i \eta_{\zeta}^i 
 \mod \mathfrak{p}_{N'^{\mathrm{u}}_{\zeta}} 
\] 
for some $d_i \in k^{\mathrm{ac}}$. 
We have 
\[
 d_0 = 
 b_1 s^2 + st 
 + 
 t \sum_{j=1}^{e-1} ( b_1 s )^{2^j} 
 + b_1 s \sum_{l=1}^{e-1} t^{2^l} 
 +b_2^2 -b_2 =0. 
\]
This implies $\sigma_2 (\theta'_{\zeta}) = \theta'_{\zeta}$, 
since we know that 
$\sigma_2 (\theta'_{\zeta}) - \theta'_{\zeta} \in \bF_2$ 
by Lemma \ref{tileq}. 
\end{proof}

\section{Refined Swan conductor}\label{RefSwan}
Let $\widetilde{K} \subset K^{\mathrm{ur}}$ 
be the unramified extension of $K^{\mathrm{u}}$ 
generated by $\mu_{p^{4pe}-1} (K^{\mathrm{ur}})$.
For a finite 
field extension $L$ of $K$ in $K^{\mathrm{ac}}$, 
we write $\widetilde{L}$ 
for the composite field of $L$ and 
$\widetilde{K}$ in $K^{\mathrm{ac}}$. 
We write 
$\widetilde{M}'_{\zeta}$ for 
${\widetilde{M}'^{\mathrm{u}}_{\zeta}}$. 
Then 
$\widetilde{N}_{\zeta}$ 
is a Galois extension of 
$\widetilde{M}'_{\zeta}$. 
By \eqref{ddt} and \eqref{delz}, 
we can take 
$\beta_{\zeta}' \in \widetilde{M}_{\zeta}$ such that 
\begin{equation}\label{beta'}
 \beta_{\zeta}'^{p^e} -\hat{r} \beta_{\zeta}' = 
 \delta_{\zeta}, \quad 
 \beta_{\zeta}' \equiv \beta_{\zeta} 
 \mod \fp_{\widetilde{M}_{\zeta}}, 
\end{equation}
since there is $x \in \bF_{2^{4e}}$ such that 
$x^{2^e} -x=\sum_{i=0}^{e-1} b_0^{2^i}$ 
if $p=2$. 
Then we have 
$\widetilde{M}_{\zeta} = \widetilde{M}'_{\zeta} (\beta_{\zeta}')$ 
by Krasner's lemma.

\begin{lem}\label{galfil}
{\rm 1.}
We have 
\begin{equation}
\psi_{\widetilde{N}_{\zeta}/\widetilde{M}'_{\zeta}}(v)=
\begin{cases}
 v \quad & \textrm{if $v \leq 1$, }\\
 p^e (v-1) +1 \quad & \textrm{if $1 < v \leq 2$,} \\ 
 p^{e+1} (v-2)+p^e+1 \quad & \textrm{if $2 < v$.} 
\end{cases}
\end{equation}
{\rm 2.}
We have 
\begin{equation*}
 \Gal (\widetilde{N}_{\zeta}/\widetilde{M}'_{\zeta} )_i = 
\begin{cases}
 \Gal (\widetilde{N}_{\zeta}/\widetilde{M}'_{\zeta} ) \quad 
 & \textrm{if $i \leq 1$, }\\
 \Gal (\widetilde{N}_{\zeta}/\widetilde{M}_{\zeta} ) \quad 
 & \textrm{if $2 \leq i \leq p^e +1$,} \\ 
 \{ 1 \} \quad & \textrm{if $p^e +2 \leq i$.} 
\end{cases}
\end{equation*}
\end{lem}
\begin{proof}
We have 
\begin{align*} 
 \psi_{\widetilde{M}_{\zeta}/\widetilde{M}'_{\zeta}}(v) 
 &=
 \begin{cases}
 v \quad & \textrm{if $v \leq 1$, } \\ 
 p^e (v-1) +1 \quad & \textrm{if $v>1$}, 
 \end{cases} \\ 
 \psi_{\widetilde{N}_{\zeta}/\widetilde{M}_{\zeta}}(v)
 &=
 \begin{cases}
 v \quad & \textrm{if $v \leq p^e+1$,} \\ 
 p (v-p^e-1) +p^e +1 \quad & \textrm{if $v>p^e+1$} 
 \end{cases}
\end{align*}
by \eqref{extgen}, \eqref{beta'} and Lemma \ref{Arfilt} 
noting that $\hat{r}$ has a $(p^e -1)$-st root in 
$\widetilde{M}'_{\zeta}$. 
Hence, 
the claim 1 follows from 
$\psi_{\widetilde{N}_{\zeta}/\widetilde{M}'_{\zeta}}=
 \psi_{\widetilde{N}_{\zeta}/\widetilde{M}_{\zeta}} 
 \circ \psi_{\widetilde{M}_{\zeta}/\widetilde{M}'_{\zeta}}$. 
The claim 2 follows from 
the claim 1 and 
$\Gal (\widetilde{N}_{\zeta}/\widetilde{M}'_{\zeta} )_{p^e+1} 
 \supset \Gal (\widetilde{N}_{\zeta}/\widetilde{M}_{\zeta} )_{p^e+1} 
 =\Gal (\widetilde{N}_{\zeta}/\widetilde{M}_{\zeta} )$. 
\end{proof}
We set 
$\varpi_{\widetilde{M}'_{\zeta}}=\delta_{\zeta}^{-1}$, 
$\varpi_{\widetilde{M}_{\zeta}}=\beta_{\zeta}^{-1}$ 
and 
$\varpi_{\widetilde{N}_{\zeta}}
 =(\gamma_{\zeta} \varpi_{\widetilde{M}_{\zeta}}^{p^{e-1}} )^{-1}$. 
Then the elements 
$\varpi_{\widetilde{M}'_{\zeta}}$, $\varpi_{\widetilde{M}_{\zeta}}$ 
and $\varpi_{\widetilde{N}_{\zeta}}$ are uniformizers of 
$\widetilde{M}'_{\zeta}$, $\widetilde{M}_{\zeta}$ and $\widetilde{N}_{\zeta}$ respectively. 
Let $\widetilde{k}$ be the residue field of $\widetilde{K}$. 
\begin{lem}\label{galex}
We have a commutative diagram
\[
 \xymatrix@C=40pt{
 U_{\widetilde{N}_{\zeta}}^{p^e+1}
 \ar[r]^-{\Nr_{\widetilde{N}_{\zeta}/\widetilde{M}'_{\zeta}}}
 \ar[d] &   
 U_{\widetilde{M}'_{\zeta}}^2
 \ar[d]
 & 
 \\
 \widetilde{k} \ar[r]^{P} & \widetilde{k} 
 } 
\]
where the map $P$ is given by 
$x \mapsto x^p -x$ 
and 
the vertical maps 
are given by 
\begin{align*}
 & p_{\widetilde{N}_{\zeta},-\gamma_{\zeta}^{-1}} 
 \colon 
 U_{\widetilde{N}_{\zeta}}^{p^e+1}
 \longrightarrow \widetilde{k} ;\ 
1-x \gamma_{\zeta}^{-1} \mapsto \bar{x}, 
\\ 
 & 
 p_{\widetilde{M}'_{\zeta},\hat{r} \varpi_{\widetilde{M}'_{\zeta}}^2} \colon 
U_{\widetilde{M}'_{\zeta}}^{2}
 \longrightarrow \widetilde{k} ;\ 
 1+x \hat{r} \varpi_{\widetilde{M}'_{\zeta}}^2 \mapsto \bar{x} .
\end{align*}
\end{lem}
\begin{proof}
The norm maps 
$\mathrm{Nr}_{\widetilde{N}_{\zeta}/\widetilde{M}_{\zeta}}$ 
and 
$\mathrm{Nr}_{\widetilde{M}_{\zeta}/\widetilde{M}'_{\zeta}}$ 
induce 
\begin{gather}
\begin{align} \notag
 & U_{\widetilde{N}_{\zeta}}^{p^e+1}/
 U_{\widetilde{N}_{\zeta}}^{p^e+2} \to 
 U_{\widetilde{M}_{\zeta}}^{p^e+1}/
 U_{\widetilde{M}_{\zeta}}^{p^e+2};\ 
 1-u \gamma_{\zeta}^{-1} \mapsto 
 1-(u^p -u) \varpi_{\widetilde{M}_{\zeta}}^{p^e+1}, \\ \notag
 & U_{\widetilde{M}_{\zeta}}^{p^e+1}/
 U_{\widetilde{M}_{\zeta}}^{p^e+2} 
 \to U_{\widetilde{M}'_{\zeta}}^2
 /U_{\widetilde{M}'_{\zeta}}^3;\ 
 1-u \varpi_{\widetilde{M}_{\zeta}}^{p^e+1} 
 =1-u \beta_{\zeta}'^{-1} \varpi_{\widetilde{M}'_{\zeta}} 
 \mapsto 1+ u \hat{r} 
 \varpi_{\widetilde{M}'_{\zeta}}^2 
\end{align}
\end{gather}
respectively 
by Lemma \ref{rswcal}.1 and calculations of the norms. 
Hence, the claim follows. 
\end{proof}

For any finite extension 
$M$ of $K$, we write $\psi_M$
for the composite $\psi_K \circ \mathrm{Tr}_{M/K}$.

\begin{lem}\label{rswxi}
We have 
$\mathrm{rsw}(\xi_{n,\zeta}|_{W_{M'^{\mathrm{u}}_{\zeta}}},
 \psi_{M'^{\mathrm{u}}_{\zeta}})=
 -n' \delta_{\zeta}^{-(p^e+1)} \mod U_{M'^{\mathrm{u}}_{\zeta}}^1$. 
\end{lem}
\begin{proof}
We put 
$\widetilde{\xi}_{n,\zeta} = 
 \xi_{n,\zeta}|_{W_{\widetilde{M}'_{\zeta}}}$, 
and regard it 
as a character of 
$\widetilde{M}_{\zeta}'^{\times}$. 
By \eqref{hom}, 
Lemma \ref{rswcal}.1 and Lemma \ref{galfil}, 
the restriction of 
$\widetilde{\xi}_{n,\zeta}$ 
to 
$U_{\widetilde{M}'_{\zeta}}^2$ 
is given by the composition 
\[
 U_{\widetilde{M}'_{\zeta}}^2 
 \xrightarrow{\Art_{\widetilde{M}'_{\zeta}}}
 \mathrm{Gal}(\widetilde{N}_{\zeta}/\widetilde{M}_{\zeta}) 
 \simeq  \mathbb{F}_p 
 \xrightarrow{\psi_0}
 \overline{\mathbb{Q}}_{\ell}^{\times}, 
\]
where the isomorphism 
$\mathrm{Gal}(\widetilde{N}_{\zeta}/\widetilde{M}_{\zeta}) 
 \simeq  \mathbb{F}_p$ 
is given by 
$\sigma \mapsto \overline{\sigma(\gamma_{\zeta}) -\gamma_{\zeta}}$. 
We define 
$p_{\widetilde{N}_{\zeta},-\gamma_{\zeta}^{-1}}$ 
as in Lemma \ref{galex}. 
For $u \in \cO_{\widetilde{M}'_{\zeta}}$, 
we put $\sigma_u=\Art_{\widetilde{M}'_{\zeta}}(1+u \hat{r}\varpi_{\widetilde{M}'_{\zeta}}^2)$ 
and then have 
\begin{align}
 \widetilde{\xi}_{n,\zeta} (1+u \hat{r} \varpi_{\widetilde{M}'_{\zeta}}^2 )
 &=\psi_0 \Bigl( \overline{\sigma_u (\gamma_{\zeta}) -\gamma_{\zeta}} \Bigr)  
 =\psi_0 \biggl( p_{\widetilde{N}_{\zeta},-\gamma_{\zeta}^{-1}} \biggl(
  \frac{\gamma_{\zeta}}{\sigma_u (\gamma_{\zeta})} \biggr)  
 \biggr) \notag \\ 
 &=\psi_0 \biggl( p_{\widetilde{N}_{\zeta},-\gamma_{\zeta}^{-1}} \biggl(
  \frac{\sigma_u (\varpi_{\widetilde{N}_{\zeta}})}{\varpi_{\widetilde{N}_{\zeta}}} \biggr)  
 \biggr) 
 =\psi_0 \circ \mathrm{Tr}_{\widetilde{k}/\mathbb{F}_p} 
 (\bar{u} ), \label{txiTr}
\end{align}
where we use Lemma \ref{rswcal}.2 and 
Lemma \ref{galex} at the last equality. 
Since we have 
\[ 
 \ol{\Tr_{\widetilde{M}'_{\zeta}/\widetilde{T}_{\zeta}}
 (\delta_{\zeta}^{p^e -1}u )}
 = -r^{-1} \bar{u} 
\] 
for $u \in \mathcal{O}_{\widetilde{M}'_{\zeta}}$, 
we obtain 
\[
 \widetilde{\xi}_{n,\zeta} (1+x )=
 \psi_{\widetilde{M}'_{\zeta}}(-n'^{-1} \delta_{\zeta}^{p^e+1} x) 
\]
for 
$x \in \mathfrak{p}_{\widetilde{M}'_{\zeta}}^2$ 
by \eqref{txiTr}. 
This implies 
\begin{equation}\label{xipsi}
 \xi_{n,\zeta} (1+x )=
 \psi_{M'^{\mathrm{u}}_{\zeta}}(-n'^{-1} \delta_{\zeta}^{p^e+1} x) 
\end{equation}
for 
$x \in \mathfrak{p}_{M'^{\mathrm{u}}_{\zeta}}^2$, 
because 
$\Tr_{\widetilde{k}/k_N} \colon \widetilde{k} \to k_N$ is surjective. 
The claim follows from \eqref{xipsi} and 
Proposition \ref{chare}.1. 
\end{proof}

\begin{lem}\label{rswtau}
We have 
$\mathrm{rsw}(\tau_{n,\zeta,\chi,c},\psi_{E_{\zeta}})
 =n' \varphi_{\zeta}' \mod U_{E_{\zeta}}^1$. 
\end{lem}
\begin{proof}
By Proposition \ref{refSwres}.1, 
we may assume that $\chi=1$ and $c=1$. 
By Proposition \ref{taunzind} and Lemma \ref{rswxi}, 
we have
\begin{equation}\label{rswNr}
 \mathrm{rsw}(
 \tau_{n,\zeta} 
 |_{W_{T^{\mathrm{u}}_{\zeta}}},\psi_{T^{\mathrm{u}}_{\zeta}}) 
 =\mathrm{Nr}_{M'^{\mathrm{u}}_{\zeta}/T^{\mathrm{u}}_{\zeta}} 
 \bigl( \mathrm{rsw}(\xi_{n,\zeta},
 \psi_{M'^{\mathrm{u}}_{\zeta}}) \bigr) 
 =
 n' \varphi_{\zeta}' 
 \mod U_{T^{\mathrm{u}}_{\zeta}}^1 . 
\end{equation}
Since $T^{\mathrm{u}}_{\zeta}$ is a tamely ramified 
extension of $E_{\zeta}$, 
we have 
\begin{equation}\label{rswres}
 \mathrm{rsw}(\tau_{n,\zeta} ,\psi_{E_{\zeta}})
 =\mathrm{rsw}(\tau_{n,\zeta} 
 |_{W_{T^{\mathrm{u}}_{\zeta}}},\psi_{T^{\mathrm{u}}_{\zeta}}) 
 \mod U_{T^{\mathrm{u}}_{\zeta}}^1 
\end{equation}
by Proposition \ref{refSwres}.2. 
The claim follows from \eqref{rswNr} and \eqref{rswres}. 
\end{proof}
\begin{prop}\label{rswdet}
We have 
$\mathrm{rsw}(\tau_{\zeta,\chi,c},\psi_K) 
 = \mathrm{rsw}(\pi_{\zeta,\chi,c},\psi_K)$. 
\end{prop}
\begin{proof}
By 
$\tau_{\zeta,\chi,c}=
 \mathrm{Ind}_{E_{\zeta}/K} \tau_{n,\zeta,\chi,c}$, 
we have 
\begin{equation}\label{rswNrtau}
 \mathrm{rsw}(\tau_{\zeta,\chi,c},\psi_K)= 
 \mathrm{Nr}_{E_{\zeta}/K} \bigl( 
 \mathrm{rsw}(\tau_{n,\zeta,\chi,c}, 
 \psi_{E_{\zeta}} ) \bigr). 
\end{equation}
Hence, the claim follows from 
Lemma \ref{rswGL} and 
Lemma \ref{rswtau}. 
\end{proof}

\begin{lem}\label{Swtau}
We have 
$\mathrm{Sw} (\tau_{\zeta,\chi,c})=1$. 
\end{lem}
\begin{proof}
This follows from 
Lemma \ref{rswtau} and \eqref{rswNrtau}. 
\end{proof}

\begin{lem}\label{tauirr}
The representation 
$\tau_{\zeta,\chi,c}$ is irreducible. 
\end{lem}
\begin{proof}
We know that the restriction of 
$\tau_{n,\zeta,\chi,c}$ 
to the wild inertia subgroup of $W_{E_{\zeta}}$ 
is irreducible by Corollary \ref{cor:irrQ0}. 
Assume that $\tau_{\zeta,\chi,c}$ 
is not irreducible. 
Then we have an irreducible factor $\tau'$ of 
$\tau_{\zeta,\chi,c}$ such that 
$\mathrm{Sw} (\tau')=0$, 
by Lemma \ref{Swtau} and 
the additivity of $\mathrm{Sw}$. 
Then, the restriction of 
$\tau'$ to the wild inertia subgroup of $W_K$ 
is trivial by $\mathrm{Sw} (\tau')=0$. 
On the other hand, 
we have an injective homomorphism 
$\tau_{n,\zeta,\chi,c} \to \tau'|_{W_{E_{\zeta}}}$ 
by Frobenius reciprocity. 
This is a contradiction. 
\end{proof}

\begin{prop}\label{tause}
The representation 
$\tau_{\zeta,\chi,c}$ is irreducible of Swan conductor $1$. 
\end{prop}
\begin{proof}
This follows from Lemma \ref{Swtau} and 
Lemma \ref{tauirr}. 
\end{proof}

\section{Epsilon factor}\label{Epsfac}
\subsection{Reduction to special cases}
In this subsection, we show the equality 
$\varepsilon (\tau_{\zeta,\chi,c}, \psi_K )
 =\varepsilon (\pi_{\zeta,\chi,c}, \psi_K)$ 
of epsilon factors assuming some results in 
the special case where $n=p^e$, $\ch K =p$ and 
$f=1$. The results in the special case will be proved 
in the next subsection. 

\begin{lem}\label{lconstLT}
We have 
\begin{align*}
 \lambda (E_{\zeta}/K, \psi_K) &=
 \begin{cases}
 \bigl( \frac{q}{n'} \bigr) & \textrm{if $n'$ is odd,}\\
 -\bigl(-\epsilon(p)
 \bigl( \frac{2n'}{p} \bigr) 
 \bigl( \frac{-1}{p} \bigr)^{\frac{n'}{2}-1} 
 \bigr)^f & \textrm{if $n'$ is even,}
 \end{cases}\\
 \lambda (T^{\mathrm{u}}_{\zeta}/E_{\zeta}, \psi_{E_{\zeta}}) &=
 \begin{cases}
 - (-1)^{\frac{(p-1)fN}{4}} & \textrm{if $p \neq 2$,}\\
 \bigl( \frac{q}{p^e +1} \bigr) 
 & \textrm{if $p=2$.} 
 \end{cases}
\end{align*}
\end{lem}
\begin{proof}
We have 
\[
 \lambda (T^{\mathrm{u}}_{\zeta}/E_{\zeta}, \psi_{E_{\zeta}}) 
 =
 \lambda (T^{\mathrm{u}}_{\zeta}/E^{\mathrm{u}}_{\zeta}, 
 \psi_{E^{\mathrm{u}}_{\zeta}})
 \lambda (E^{\mathrm{u}}_{\zeta}/E_{\zeta}, \psi_{E_{\zeta}})^{p^e +1} 
 =\lambda (T^{\mathrm{u}}_{\zeta}/E^{\mathrm{u}}_{\zeta}, 
 \psi_{E^{\mathrm{u}}_{\zeta}}). 
\]
If $p \neq 2$, we have
\[
\lambda (T^{\mathrm{u}}_{\zeta}/E^{\mathrm{u}}_{\zeta}, 
 \psi_{E^{\mathrm{u}}_{\zeta}})=
 -\biggl(-\epsilon(p)\left(\frac{2n'}{p}\right)\left(\frac{-1}{p}\right)^{\frac{p^e-1}{2}}\biggr)^{f N}= - (-1)^{\frac{(p-1)fN}{4}}
\]
by Lemma \ref{lconsttame}, 
since $f N$ is even. The other assertions immediately follow from 
Lemma \ref{lconsttame}. 
\end{proof}

\begin{lem}\label{lconstM}
We have 
\begin{equation*}
 \lambda (M'^{\mathrm{u}}_{\zeta}/T^{\mathrm{u}}_{\zeta}, 
 \psi_{T^{\mathrm{u}}_{\zeta}}) =
 \begin{cases}
 (-1)^f & \textrm{if $p =2$ and $e \leq 2$,}\\ 
 \bigl( \frac{r}{k_N} \bigr) & \textrm{otherwise.}
 \end{cases}\\
\end{equation*}
\end{lem}
\begin{proof}
Let $K_{(0)}$ and $K_{(p)}$ be 
non-archimedean local fields of 
characteristic $0$ and $p$ respectively. 
Assume that 
the residue fields of $K_{(0)}$ and $K_{(p)}$ 
are isomorphic to $k$. 
We take uniformizers $\varpi_{(0)}$ and 
$\varpi_{(p)}$ of $K_{(0)}$ and $K_{(p)}$ respectively. 
We define 
$T^{\mathrm{u}}_{\zeta,(0)}$ similarly as 
$T^{\mathrm{u}}_{\zeta}$ 
starting from $K_{(0)}$. 
We use similar notations also for other objects 
in the characteristic zero side and 
the positive characteristic side. 
We have the isomorphism 
\begin{align*}
 \cO_{T^{\mathrm{u}}_{\zeta,(p)}} 
 / \mathfrak{p}_{T^{\mathrm{u}}_{\zeta,(p)}}^2 
 \stackrel{\sim}{\longrightarrow} 
 \cO_{T^{\mathrm{u}}_{\zeta,(0)}} 
 / \mathfrak{p}_{T^{\mathrm{u}}_{\zeta,(0)}}^2 ;\ 
 \xi_0 + \xi_1 \varpi_{T^{\mathrm{u}}_{\zeta,(p)}} 
 \mapsto 
 \hat{\xi}_0 + \hat{\xi}_1 \varpi_{T^{\mathrm{u}}_{\zeta,(0)}} 
\end{align*}
of algebras, where 
$\xi_0, \xi_1 \in k$. 
Hence, 
it suffices to show the claim in one of the 
characteristic zero side and 
the positive characteristic side 
by \cite[Proposition 3.7.1]{Delp0}, 
since 
$\Gal (M'^{\mathrm{u}}_{\zeta,(p)}/T^{\mathrm{u}}_{\zeta,(p)})^2=1$ 
and 
$\Gal (M'^{\mathrm{u}}_{\zeta,(0)}/
 T^{\mathrm{u}}_{\zeta,(0)})^2=1$, 
where we use upper numbering filtration of Galois groups.

First, we consider the case where 
$p \neq 2$ 
and $\ch K=p$. Then, we have 
$d_{M'^{\mathrm{u}}_{\zeta}/T^{\mathrm{u}}_{\zeta}}=\hat{r}$ 
by Proposition \ref{dw2f} and the fact that $fN$ is even. 
Hence, $\delta_{M'^{\mathrm{u}}_{\zeta}/T^{\mathrm{u}}_{\zeta}}$
is unramified by \eqref{dka}. 
Hence, we have 
\[
\lambda (M'^{\mathrm{u}}_{\zeta}/T^{\mathrm{u}}_{\zeta}, 
 \psi_{T^{\mathrm{u}}_{\zeta}}) 
 =\varepsilon(\delta_{M'^{\mathrm{u}}_{\zeta}/T^{\mathrm{u}}_{\zeta}},\psi_{T^{\mathrm{u}}_{\zeta}})^{p^e} =
 \biggl( \frac{r}{k_N} \biggr)
 \]
by \cite[Proposition 2]{Henepsmod}, \cite[23.5 Proposition]{BHLLCGL2} and \eqref{dka}.

We consider the case where $p=2$. 
Assume that $e \geq 3$ and $\ch K =0$. 
We have $D=2^e \delta_{\zeta}^{2^e-1}+1$ in the notation of Proposition \ref{dw2f}
with $(L,K,a)=(M'^{\mathrm{u}}_{\zeta},T^{\mathrm{u}}_{\zeta},\delta_{\zeta})$. 
Then, we have $D \in ({M'^{\mathrm{u}}_{\zeta}}^{\times})^2$. 
Hence, we have $\kappa_D=1$, $d_{M'^{\mathrm{u}}_{\zeta}/T^{\mathrm{u}}_{\zeta}}=1$ and 
\[ 
 w_2 (\Ind_{M'^{\mathrm{u}}_{\zeta}/T^{\mathrm{u}}_{\zeta}} 1)=1 
\]
by Proposition \ref{dw2f} and 
$\binom{p^e}{4} \equiv 0 \mod 2$. 
Therefore we have 
\[
 \lambda (M'^{\mathrm{u}}_{\zeta}/T^{\mathrm{u}}_{\zeta}, 
 \psi_{T^{\mathrm{u}}_{\zeta}}) = 
 \varepsilon (\Ind_{M'^{\mathrm{u}}_{\zeta}/T^{\mathrm{u}}_{\zeta}} 1, 
 \psi_{T^{\mathrm{u}}_{\zeta}}) 
 =
 \varepsilon (1^{\oplus p^e}, 
 \psi_{T^{\mathrm{u}}_{\zeta}})
 =1 
\]
by Theorem \ref{SWeps}.

Assume that $e=2$ and $\ch K =2$. 
Then we see that 
$d^{+}_{M'^{\mathrm{u}}_{\zeta}/T^{\mathrm{u}}_{\zeta}} = 1$ 
by Definition \ref{adiscdef}. 
Hence, 
$\delta_{M'^{\mathrm{u}}_{\zeta}/T^{\mathrm{u}}_{\zeta}}$ 
is the unramified character satisfying 
$\delta_{M'^{\mathrm{u}}_{\zeta}/T^{\mathrm{u}}_{\zeta}} 
 (\varpi_{T^{\mathrm{u}}_{\zeta}}) =(-1)^f$ 
by Theorem \ref{adiscthm}. 
Then we see that 
\[
 \lambda (M'^{\mathrm{u}}_{\zeta}/T^{\mathrm{u}}_{\zeta}, 
 \psi_{T^{\mathrm{u}}_{\zeta}}) = 
 \varepsilon (\Ind_{M'^{\mathrm{u}}_{\zeta}/T^{\mathrm{u}}_{\zeta}} 1, 
 \psi_{T^{\mathrm{u}}_{\zeta}}) = 
 \varepsilon (\delta_{M'^{\mathrm{u}}_{\zeta}/T^{\mathrm{u}}_{\zeta}} 
 \oplus 1^{\oplus 3}, 
 \psi_{T^{\mathrm{u}}_{\zeta}}) =(-1)^f , 
\]  
where we use Theorem \ref{SWeps} at the second equality.

Assume that $e=1$ and $\ch K=2$. 
Let $\kappa_{M'^{\mathrm{u}}_{\zeta}/T^{\mathrm{u}}_{\zeta}}$ 
be the quadratic character associated to 
the extension $M'^{\mathrm{u}}_{\zeta}$ over 
$T^{\mathrm{u}}_{\zeta}$. 
Then we have 
\[
 \lambda (M'^{\mathrm{u}}_{\zeta}/T^{\mathrm{u}}_{\zeta}, 
 \psi_{T^{\mathrm{u}}_{\zeta}}) = 
 \varepsilon ( \kappa_{M'^{\mathrm{u}}_{\zeta}/T^{\mathrm{u}}_{\zeta}}, 
 \psi_{T^{\mathrm{u}}_{\zeta}}) 
\]
by Theorem \ref{SWeps} similarly as above.
We can check that 
the norm map 
$\mathrm{Nr}_{M'^{\mathrm{u}}_{\zeta}/T^{\mathrm{u}}_{\zeta}}$ 
induces 
\[ 
 U_{M'^{\mathrm{u}}_{\zeta}}^1/
 U_{M'^{\mathrm{u}}_{\zeta}}^2 \to 
 U_{T^{\mathrm{u}}_{\zeta}}^1/
 U_{T^{\mathrm{u}}_{\zeta}}^2;\ 
 1+u \delta_{\zeta}^{-1} \mapsto 
 1+(u^2 -u) \alpha_{\zeta} . 
\]
Then, by Lemma \ref{rswcal}, we have 
\begin{align}
\kappa_{M'^{\mathrm{u}}_{\zeta}/T^{\mathrm{u}}_{\zeta}}(1+\alpha_{\zeta} x) &=\psi_0 
 \Bigl( \mathrm{Art}_{T^{\mathrm{u}}_{\zeta}}(1+\alpha_{\zeta} x)(\delta_{\zeta}) - \delta_{\zeta} \Bigr) \notag \\
&=\psi_0 \left( p_{M'^{\mathrm{u}}_{\zeta},\delta_{\zeta}^{-1}}\left(
\frac{\mathrm{Art}_{T^{\mathrm{u}}_{\zeta}}(1+\alpha_{\zeta} x)(\delta_{\zeta}^{-1})}{\delta_{\zeta}^{-1}}
\right)\right)=\psi_0(\Tr_{k/\mathbb{F}_p}(\bar{x})) \label{kaprsw}
\end{align}
for $x \in \mathcal{O}_{T^{\mathrm{u}}_{\zeta}}$
noting that $k_N=k$. 
Hence, we have $\mathrm{rsw}(\kappa_{M'^{\mathrm{u}}_{\zeta}/T^{\mathrm{u}}_{\zeta}},
\psi_{T^{\mathrm{u}}_{\zeta}})=\alpha_{\zeta}$ 
by Proposition \ref{chare}.1. 
By Proposition \ref{chare}.2, 
we have 
\[
 \varepsilon ( \kappa_{M'^{\mathrm{u}}_{\zeta}/T^{\mathrm{u}}_{\zeta}}, 
 \psi_{T^{\mathrm{u}}_{\zeta}})=
 \kappa_{M'^{\mathrm{u}}_{\zeta}/T^{\mathrm{u}}_{\zeta}}
 (\alpha_{\zeta})=
\kappa_{M'^{\mathrm{u}}_{\zeta}/T^{\mathrm{u}}_{\zeta}}
 (1+\alpha_{\zeta})
 =(-1)^f,  
\]
where we use 
$\Nr_{M'^{\mathrm{u}}_{\zeta}/T^{\mathrm{u}}_{\zeta}}(\delta_{\zeta})=\alpha_{\zeta}^{-1} +1$ 
and 
\eqref{kaprsw} at the last equality.

\end{proof}

\begin{lem}\label{Treta}
We have 
\[
 \Tr_{M'^{\mathrm{u}}_{\zeta}/T^{\mathrm{u}}_{\zeta}} 
 (\delta_{\zeta}^{i}) = 
 \begin{cases}
 0 & \textrm{if $1 \leq i \leq p^e -2$,} \\ 
 \hat{r}^{-1} (p^e -1) & \textrm{if $i = p^e -1$.} 
 \end{cases} 
\] 
\end{lem}
\begin{proof}
Vanishing for $1 \leq i \leq p^e -2$ follows from 
\eqref{delz}. 
We have also 
\[
 \Tr_{M'^{\mathrm{u}}_{\zeta}/T^{\mathrm{u}}_{\zeta}} 
 (\delta_{\zeta}^{p^e -1}) = 
 \Tr_{M'^{\mathrm{u}}_{\zeta}/T^{\mathrm{u}}_{\zeta}} 
 \Bigl( \hat{r}^{-1} + \delta_{\zeta}^{-1} 
 (-\alpha_{\zeta}^{-1} +\epsilon_1 ) 
 \Bigr) =\hat{r}^{-1} (p^e -1) 
\]
by \eqref{delz}. 
\end{proof}

\begin{lem}\label{krsw}
We have 
\[
 \delta_{T^{\mathrm{u}}_{\zeta}/E_{\zeta}} 
 \bigl( 
 \mathrm{rsw} (\tau_{n,\zeta},\psi_{E_{\zeta}}) \bigr) = 
 \begin{cases}
 1 & \textrm{if $p \neq 2$,}\\ 
 \bigl( \frac{q}{p^e +1} \bigr) & \textrm{if $p=2$.} 
 \end{cases} 
\]
\end{lem}
\begin{proof}
If $p=2$, the claim follows from 
Lemma \ref{deltatame}.1 and 
Lemma \ref{rswtau}, 
since $T^{\mathrm{u}}_{\zeta}$ is 
totally ramified over $E_{\zeta}$. 

Assume that $p \neq 2$. 
Then we have 
$d_{T^{\mathrm{u}}_{\zeta}/E^{\mathrm{u}}_{\zeta}} 
 =(-1)^{(p^e+1)/2} \varphi_{\zeta}'$ 
by Proposition \ref{dw2f}. 
Hence, we have 
$\delta_{T^{\mathrm{u}}_{\zeta}/E^{\mathrm{u}}_{\zeta}} 
 \bigl( (-1)^{(p^e-1)/2} \varphi_{\zeta}' \bigr)=1$ by Lemma \ref{deltatame}.2. 
Therefore, we have 
\[
 \delta_{T^{\mathrm{u}}_{\zeta}/E_{\zeta}} 
 \bigl( 
 \mathrm{rsw} (\tau_{n,\zeta},\psi_{E_{\zeta}}) \bigr) 
 =\delta_{T^{\mathrm{u}}_{\zeta}/E^{\mathrm{u}}_{\zeta}} 
 (n' \varphi_{\zeta}' ) 
 =\delta_{T^{\mathrm{u}}_{\zeta}/E^{\mathrm{u}}_{\zeta}} 
 (n' (-1)^{\frac{p^e-1}{2}} )=
 \biggl( \frac{n' (-1)^{\frac{p^e-1}{2}}}{q^N} \biggr)=1 
\]
by \cite[(1)]{GallDet},
Lemma \ref{deltatame}.2, 
Lemma \ref{rswtau} and the fact that $fN$ is even. 
\end{proof}

\begin{lem}\label{epsmod}
Assume that $n=p^e$. 
Then we have 
$\varepsilon (\tau_{\zeta,\chi,c}, \psi_K )
 \equiv \varepsilon (\pi_{\zeta,\chi,c}, \psi_K) 
 \mod{\mu_{p^e} (\bC)}$. 
\end{lem}
\begin{proof}
Let $\pi$ be the representation of $\iGL_n (K)$ corresponding to 
$\tau_{\zeta,\chi,c}$ by the local Langlands correspondence. 
By the proof of \cite[2.2 Proposition]{BHLepi}, 
Proposition \ref{centdet} and Proposition \ref{rswdet}, 
we have 
$\pi \simeq 
 \mathrm{c\mathchar`-Ind}_{L_{\zeta}^{\times} 
 U_{\mathfrak{I}}^1}^{\mathit{GL}_n(K)} \Lambda$ 
for a character 
$\Lambda \colon L_{\zeta}^{\times} U_{\mathfrak{I}}^1 
 \to \bC^{\times}$ 
which coincides with $\Lambda_{\zeta,\chi,c}$ 
on $K^{\times} U_{\mathfrak{I}}^1$. 
Then, the claim follows from 
\cite[2.2 Lemma (1)]{BHLepi}, 
because 
$L_{\zeta}^{\times} U_{\mathfrak{I}}^1/(K^{\times} U_{\mathfrak{I}}^1)$ 
is the cyclic group of order $p^e$. 
\end{proof}

\begin{prop}\label{coineps}
We have 
$\varepsilon (\tau_{\zeta,\chi,c}, \psi_K )
 =\varepsilon (\pi_{\zeta,\chi,c}, \psi_K)$. 
\end{prop}
\begin{proof}
By Proposition \ref{epspi} and 
$\tau_{\zeta,\chi,c} \simeq \Ind_{E_{\zeta}/K} \tau_{n,\zeta,\chi,c}$, 
it suffices to show that 
\begin{equation*}\label{red}
 \lambda (E_{\zeta}/K, \psi_K)^{p^e} 
 \varepsilon (\tau_{n,\zeta,\chi,c}, \psi_{E_{\zeta}} ) = 
 (-1)^{n-1+\epsilon_0 f} \chi(n') c. 
\end{equation*}
By Lemma \ref{rswtau}, 
we may assume $\chi=1$ and $c=1$.
Hence, it suffices to show 
\begin{equation}\label{red}
 \lambda (E_{\zeta}/K, \psi_K)^{p^e} 
 \varepsilon (\tau_{n,\zeta}, \psi_{E_{\zeta}} ) = 
 (-1)^{n-1+\epsilon_0 f}. 
\end{equation}

Assuming that \eqref{red} is proved for $n=p^e$, we show 
\eqref{red} for general $n$. 
Let $\tau'_{n,\zeta}$ denote the representation of 
$W_{E_{\zeta}}$ given by $\Theta_{\zeta}$ in \eqref{hom} 
and $\tau_{p^e}$. 
We put $\psi'_{E_{\zeta}}=n'^{-1}\psi_{E_{\zeta}}$. 
Applying the result for $n=p^e$ 
to $E_{\zeta}, \varphi'_{\zeta}$ 
in place of $K, \varpi$, 
we have 
$\varepsilon (\tau'_{n,\zeta}, \psi'_{E_{\zeta}} )=(-1)^{p^e-1+\epsilon'_0 f}$, 
where $\epsilon_0'$ denotes $\epsilon_0$ for $n=p^e$. 
Since $\det \tau'_{n,\zeta}$ is unramified as in the proof of 
Proposition \ref{centdet}, 
we have 
\begin{equation}\label{eq:etau'}
\varepsilon (\tau'_{n,\zeta}, \psi_{E_{\zeta}} )
 =\det \tau'_{n,\zeta} (n') \varepsilon (\tau'_{n,\zeta}, \psi'_{E_{\zeta}} ) 
 =(-1)^{p^e-1+\epsilon_0' f} . 
\end{equation}
We note that the inflation of the character in \eqref{chartwi} by 
$\Theta_{\zeta}$ factors through 
\[
W_{E_{\zeta}} \to \{\pm 1\} \times \mathbb{Z};\ 
\sigma \mapsto \left(a_{\sigma}^{\frac{p^e+1}{2}}, 
f n_{\sigma}\right).   
\]
If $p \neq 2$, 
we have 
$(n'\varphi'_{\zeta},-\varphi'_{\zeta})_{E_{\zeta}}=\bigl(\frac{n'}{q}\bigr)$, 
where 
\[
(\ ,\ )_{E_{\zeta}} \colon E_{\zeta}^{\times}/(E_{\zeta}^{\times})^2 \times 
E_{\zeta}^{\times}/(E_{\zeta}^{\times})^2 \to \{\pm 1\} 
\]
denotes the Hilbert symbol. 
Hence, we have 
\begin{equation}\label{etautau'}
\frac{\varepsilon(\tau_{n,\zeta},\psi_{E_{\zeta}})}{\varepsilon(\tau'_{n,\zeta}, \psi_{E_{\zeta}})}=
\begin{cases}
\left(\frac{n'}{q}\right)^{n-p^e} \left( \left(\frac{n'}{p}\right)^n 
\left(-\epsilon(p) \left(\frac{-2}{p}\right)\right)^{n-p^e}\right)^{f} & 
\textrm{if $p \neq 2$,} \\ 
(-1)^{(\frac{n(n-2)}{8} - \frac{2^e(2^e-2)}{8})f} & 
\textrm{if $p = 2$} 
\end{cases}
\end{equation}
by \eqref{chartwi}, Lemma \ref{rswW} and Lemma \ref{rswtau}. 
Then we have \eqref{red} 
by Lemma \ref{lconstLT}, \eqref{eq:etau'} 
and \eqref{etautau'}. 

Therefore, we may assume that $n=p^e$. 
By Lemma \ref{lconstLT} and Lemma \ref{epsmod}, 
it suffices to show that 
\[
 \varepsilon (\tau_{n,\zeta}, \psi_{E_{\zeta}} )^{N(p^e +1)} = 
 \begin{cases}
 1 & \textrm{if $p \neq 2$,}\\ 
 (-1)^{1+\epsilon_0 f} & \textrm{if $p = 2$.} 
 \end{cases} 
\]
By Proposition \ref{reseps}, we have 
\begin{equation*}\label{epskl}
 \varepsilon (\tau_{n,\zeta}, \psi_{E_{\zeta}} )^{N(p^e +1)} 
 = 
 \delta_{T^{\mathrm{u}}_{\zeta}/E_{\zeta}} 
 \bigl( 
 \mathrm{rsw} (\tau_{n,\zeta},\psi_{E_{\zeta}}) \bigr)^{-1} 
 \lambda (T^{\mathrm{u}}_{\zeta}/E_{\zeta}, \psi_{E_{\zeta}})^{p^e} 
 \varepsilon (\tau_{n,\zeta}|_{W_{T^{\mathrm{u}}_{\zeta}}}, 
 \psi_{T^{\mathrm{u}}_{\zeta}} ). 
\end{equation*}
By this, Lemma \ref{lconstLT} and Lemma \ref{krsw}, 
it suffices to show that 
\[
 \varepsilon (\tau_{n,\zeta}|_{W_{T^{\mathrm{u}}_{\zeta}}}, 
 \psi_{T^{\mathrm{u}}_{\zeta}} ) =
 \begin{cases}
 -(-1)^{\frac{(p-1)fN}{4}} & \textrm{if $p \neq 2$,}\\ 
 (-1)^{1+\epsilon_0 f} 
 \bigl( \frac{q}{p^e +1} \bigr) & \textrm{if $p = 2$.} 
 \end{cases} 
\]
This follows from Lemma \ref{lconstM} 
and Proposition \ref{epsxi2}. 
\end{proof}

We set $\varpi_{M'^{\mathrm{u}}_{\zeta}}=\delta_{\zeta}^{-1}$. 
\begin{prop}\label{epsxi2}
Assume that $n=p^e$. 
Then we have 
\[
 \varepsilon (\xi_{n,\zeta}, \psi_{M'^{\mathrm{u}}_{\zeta}} )=
 \begin{cases}
 -(-1)^{\frac{(p-1)fN}{4}} \bigl( \frac{r}{k_N} \bigr) 
 & \textrm{if $p \neq 2$,} \\ 
 (-1)^{1+\epsilon_0 f} & \textrm{if $p = 2$.} 
\end{cases} 
\]
\end{prop}
\begin{proof}
First, we reduce the problem to the positive characteristic case. 
Assume that $\ch K =0$. 
Take a positive characteristic 
local field $K_{(p)}$ 
whose residue field is isomorphic to $k$. 
We define 
$M'^{\mathrm{u}}_{\zeta,(p)}$ 
similarly as 
$M'^{\mathrm{u}}_{\zeta}$ 
starting from $K_{(p)}$. 
We use similar notations also 
for other objects in the positive characteristic side. 
Then we have the isomorphism 
\[
 \cO_{M'^{\mathrm{u}}_{\zeta,(p)}} 
 / \mathfrak{p}_{M'^{\mathrm{u}}_{\zeta,(p)}}^3 
 \stackrel{\sim}{\longrightarrow} 
 \cO_{M'^{\mathrm{u}}_{\zeta}} 
 / \mathfrak{p}_{M'^{\mathrm{u}}_{\zeta}}^3 ;\ 
 \xi_0 + \xi_1 \varpi_{M'^{\mathrm{u}}_{\zeta,(p)}} 
 + \xi_2 \varpi_{M'^{\mathrm{u}}_{\zeta,(p)}}^2 
 \mapsto 
 \hat{\xi}_0 + \hat{\xi}_1 \varpi_{M'^{\mathrm{u}}_{\zeta}} 
 + \hat{\xi}_2 \varpi_{M'^{\mathrm{u}}_{\zeta}}^2 
\] 
of algebras, 
where $\xi_1, \xi_2, \xi_3 \in k$. 
Hence, the problem is reduced to 
the positive characteristic case 
by \cite[Proposition 3.7.1]{Delp0}. 

We may assume $K=\bF_q ((t))$. 
We put $K_{\langle 1 \rangle} =\bF_p ((t))$. 
We define 
$M'^{\mathrm{u}}_{\zeta,\langle 1 \rangle}$ 
similarly as 
$M'^{\mathrm{u}}_{\zeta}$ 
starting from $K_{\langle 1 \rangle}$. 
We use similar notations also 
for other objects in the $K_{\langle 1 \rangle}$-case. 
We put 
$f' =[M'^{\mathrm{u}}_{\zeta} : 
 M'^{\mathrm{u}}_{\zeta,\langle 1 \rangle}]$. 
We have 
\[
 \delta_{M'^{\mathrm{u}}_{\zeta}/
 M'^{\mathrm{u}}_{\zeta,\langle 1 \rangle}} 
 (\mathrm{rsw} (\xi_{n,\zeta,\langle 1 \rangle},
 \psi_{M'^{\mathrm{u}}_{\zeta,\langle 1 \rangle}} )) 
 =(-1)^{f'-1} 
\] 
by Lemma \ref{rswxi}. 
We have 
$\lambda (M'^{\mathrm{u}}_{\zeta}/M'^{\mathrm{u}}_{\zeta,\langle 1 \rangle},
 \psi_{M'^{\mathrm{u}}_{\zeta,\langle 1 \rangle}} ) =1$, 
since the level of 
$\psi_{M'^{\mathrm{u}}_{\zeta,\langle 1 \rangle}}$ is 
$2-p^e$ by Lemma \ref{Treta}. 
Then, we obtain 
\begin{equation}\label{epsf1}
 \varepsilon (\xi_{n,\zeta},\psi_{M'^{\mathrm{u}}_{\zeta}} ) 
 =(-1)^{f'-1} 
 \varepsilon (\xi_{n,\zeta,\langle 1 \rangle},
 \psi_{M'^{\mathrm{u}}_{\zeta,(1)}} )^{f'} 
\end{equation}
by Proposition \ref{reseps}. 
By \eqref{epsf1}, 
the problem is reduced to the case where $f=1$. 
In this case, the claim follows from 
Lemma \ref{epsxiodd} and Lemma \ref{epsxipos}. 
\end{proof}

\subsection{Special cases}
We assume that $n=p^e$, $\ch K =p$ and 
$f=1$ in this subsection. 

\subsubsection{Odd case}
Assume that $p \neq 2$ in this subsubsection. 
\begin{lem}\label{psi=1}
We have 
$\psi_{M'^{\mathrm{u}}_{\zeta}} \bigl( 
 - \delta_{\zeta}^{p^e +1} 
 (1+x \varpi_{M'^{\mathrm{u}}_{\zeta}}) \bigr)=1$ 
for 
$x \in k_N$. 
\end{lem}
\begin{proof}
For $x \in k_N$, 
we have 
\[
 \psi_{M'^{\mathrm{u}}_{\zeta}} \bigl( 
 - \delta_{\zeta}^{p^e +1} 
 (1+x \varpi_{M'^{\mathrm{u}}_{\zeta}}) \bigr) = 
 \psi_{M'^{\mathrm{u}}_{\zeta}} \bigl( 
 -(r^{-1} \delta_{\zeta} 
 -\alpha_{\zeta}^{-1} ) 
 (\delta_{\zeta} +x ) \bigr) = 
 \psi_{M'^{\mathrm{u}}_{\zeta}} ( 
 -r^{-1} \delta_{\zeta}^2 ), 
\]
because 
$\Tr_{M'^{\mathrm{u}}_{\zeta}/T^{\mathrm{u}}_{\zeta}} 
 (\delta_{\zeta})=0$ and $[M'^{\mathrm{u}}_{\zeta}:T^{\mathrm{u}}_{\zeta}]=p^e$. 
If $p^e \neq 3$, then we have 
the claim, because 
$\Tr_{M'^{\mathrm{u}}_{\zeta}/T^{\mathrm{u}}_{\zeta}} 
 (\delta_{\zeta}^2)=0$. 
 
We assume that $p^e = 3$. 
Then we have 
\[
 \psi_{M'^{\mathrm{u}}_{\zeta}} \bigl( 
 -r^{-1} \delta_{\zeta}^2 \bigr)= 
 \psi_{T^{\mathrm{u}}_{\zeta}} \bigl( 
 -2 r^{-2} \bigr) = 
 \psi_0 \bigl( \Tr_{k_N /\bF_p} 
 (-2 r^{-2} ) \bigr) = 
 \psi_0 \bigl( - N \Tr_{\bF_{p^2}/\bF_p} 
 ( r^{-2} ) \bigr) =1 
\] 
by 
$\Tr_{M'^{\mathrm{u}}_{\zeta}/T^{\mathrm{u}}_{\zeta}} 
 (\delta_{\zeta}^2) =2r^{-1}$ 
and 
$r^4=-1$. 
\end{proof}

Let $\theta_{\zeta}$ be as in \eqref{tzeta}. 

\begin{lem}\label{Nrmod}
We have 
\[
 \Nr_{N'^{\mathrm{u}}_{\zeta}/M'^{\mathrm{u}}_{\zeta}} 
 (1 +x \theta_{\zeta}^{\frac{p-1}{2}} 
 \varpi_{M'^{\mathrm{u}}_{\zeta}}) 
 \equiv 
 1 + ( -2r )^{\frac{1-p}{2}} x^p 
 \varpi_{M'^{\mathrm{u}}_{\zeta}} 
 + \frac{x^2}{2} \varpi_{M'^{\mathrm{u}}_{\zeta}}^2 
 \mod{\mathfrak{p}_{M'^{\mathrm{u}}_{\zeta}}^3} 
\]
for 
$x \in k_N$. 
\end{lem}
\begin{proof}
We put 
$T=1 +x \theta_{\zeta}^{\frac{p-1}{2}} 
 \varpi_{M'^{\mathrm{u}}_{\zeta}}$. 
By 
$\theta_{\zeta}^p -\theta_{\zeta} = 
 (-2r)^{-1} \delta_{\zeta}^2$ in \eqref{thetad}, 
we have 
\[
 \theta_{\zeta} = 
 -\frac{1}{2r} \delta_{\zeta}^2 
 \Bigl( \bigl( x^{-1} (T-1) \delta_{\zeta} \bigr)^2 
 -1 \Bigl)^{-1}. 
\]
Substituting this to 
$x^{-1} (T-1) \delta_{\zeta} =
 \theta_{\zeta}^{\frac{p-1}{2}}$, 
we have 
\[
 \bigl( T^2 -2T +1 -x^2 \varpi_{M'^{\mathrm{u}}_{\zeta}}^2 
 \bigr)^{\frac{p-1}{2}} (T-1) -( -2r )^{\frac{1-p}{2}} x^p  
 \varpi_{M'^{\mathrm{u}}_{\zeta}} 
 =0. 
\]
The claim follows from this. 
\end{proof}

\begin{lem}\label{sumxi}
We have 
\[
 \sum_{x \in k_N} 
 \xi_{n,\zeta} 
 (1+x \varpi_{M'^{\mathrm{u}}_{\zeta}})^{-1} =
 -((-1)^{\frac{p-1}{2}}p)^{e_0} 
 \Bigl( \frac{r}{k_N} \Bigr). 
\]
\end{lem}
\begin{proof}
Let $\xi'_{n,\zeta}$ be as in subsection \ref{Study}. 
We note that 
the left hand side of the claim does not change even if we replace 
$\xi_{n,\zeta}$ by $\xi'_{n,\zeta}$. 
We have 
\begin{align}
 \sum_{x \in k_N} 
 \xi'_{n,\zeta} 
 (1+x \varpi_{M'^{\mathrm{u}}_{\zeta}})^{-1} &=
 \sum_{x \in k_N} 
 \xi'_{n,\zeta} 
 \bigl( 1+ ( -2r )^{\frac{1-p}{2}} x^p 
 \varpi_{M'^{\mathrm{u}}_{\zeta}} \bigr)^{-1} \notag \\ 
 &=
 \sum_{x \in k_N} 
 \xi'_{n,\zeta} 
 \Bigl( 1-\frac{x^2}{2} \varpi_{M'^{\mathrm{u}}_{\zeta}}^2 \Bigr)^{-1} 
 =\sum_{x \in k_N} 
 \psi_{M'^{\mathrm{u}}_{\zeta}} 
 \Bigl( -\frac{x^2}{2} \delta_{\zeta}^{p^e -1} \Bigr) \label{xitopsi}, 
\end{align} 
where we use 
Lemma \ref{oddfac} and Lemma \ref{Nrmod} 
at the second equality and 
\eqref{xipsi} at the last equality. 
The last expression in \eqref{xitopsi} is equal to 
\[
 \sum_{x \in k_N} 
 \psi_{T^{\mathrm{u}}_{\zeta}} 
 \bigl( -(2r)^{-1} (p^e -1) x^2 \bigr) = 
 \sum_{x \in k_N } \psi_0 (\Tr_{k_N /\bF_p} (r x^2)) 
 =-\bigl( (-1)^{\frac{p-1}{2}} p \bigr)^{e_0}
 \Bigl( \frac{r}{k_N} \Bigr) 
\]
by \eqref{Gss}, \eqref{HD}, Lemma \ref{Treta} and $N=2e_0$. 
\end{proof}

\begin{lem}\label{epsxiodd}
We have 
$\varepsilon (\xi_{n,\zeta},\psi_{M'^{\mathrm{u}}_{\zeta}} )
 =-(-1)^{\frac{(p-1)e_0}{2}} \bigl( \frac{r}{k_N} \bigr)$. 
\end{lem}
\begin{proof}
We have 
\begin{align*}
 \varepsilon (\xi_{n,\zeta},\psi_{M'^{\mathrm{u}}_{\zeta}} ) &= 
 p^{-e_0} 
 \sum_{x \in k_N} 
 \xi_{n,\zeta} \bigl( - \delta_{\zeta}^{p^e +1} 
 (1+x \varpi_{M'^{\mathrm{u}}_{\zeta}}) \bigr)^{-1} 
 \psi_{M'^{\mathrm{u}}_{\zeta}} \bigl( 
 - \delta_{\zeta}^{p^e +1} 
 (1+x \varpi_{M'^{\mathrm{u}}_{\zeta}}) \bigr) \\ 
 &= 
 -(-1)^{\frac{(p-1)e_0}{2}} \Bigl( \frac{r}{k_N} \Bigr) 
 \xi_{n,\zeta} ( - \delta_{\zeta}^{p^e +1} )^{-1} 
\end{align*}
by Proposition \ref{chare}.2, 
Lemma \ref{psi=1} and Lemma \ref{sumxi}. 
We have 
\[
 \xi_{n,\zeta} ( - \delta_{\zeta}^{p^e +1} ) 
 =\xi'_{n,\zeta}(-\delta_{\zeta}^{p^e+1})
(-1)^{\frac{p-1}{2} \frac{p^e+1}{2} N} =
\xi'_{n,\zeta} ( - \delta_{\zeta}^{p^e +1} ) = 
 \xi'_{n,\zeta} \bigl( - ( -2r)^{(p^e +1) \frac{1-p}{2}} \bigr) 
 =1, 
\]
where we use 
\[
 \Nr_{N'^{\mathrm{u}}_{\zeta}/M'^{\mathrm{u}}_{\zeta}} 
 (\theta_{\zeta}^{\frac{p-1}{2}} \varpi_{M'^{\mathrm{u}}_{\zeta}}) 
 = ( -2r )^{\frac{1-p}{2}} 
 \varpi_{M'^{\mathrm{u}}_{\zeta}} 
\]
at the third equality and 
$k_N^{\times} \subset \Nr_{N'^{\mathrm{u}}_{\zeta}/M'^{\mathrm{u}}_{\zeta}}((N'^{\mathrm{u}}_{\zeta})^{\times})$  
at the last equality. 
Thus, we have the claim. 
\end{proof}
\subsubsection{Even case}
Assume that $p = 2$ in this subsubsection. 
\begin{lem}\label{Trpe}
We have 
$\Tr_{M'^{\mathrm{u}}_{\zeta}/K}(\delta_{\zeta}^{2^e +1}) =0$ and 
\[
 \Tr_{M'^{\mathrm{u}}_{\zeta}/K}(\delta_{\zeta}^{2^e}) =
 \begin{cases}
 1 & \textrm{if $e=1$,}\\ 
 0 & \textrm{if $e \geq 2$.} 
 \end{cases}
\]
\end{lem}
\begin{proof}
These follow from 
$\delta_{\zeta}^{2^e} -\delta_{\zeta} = 
 \alpha_{\zeta}^{-1} +1$.  
\end{proof}

\begin{lem}\label{Nrte}
We have 
$\Nr_{N'^{\mathrm{u}}_{\zeta}/M'^{\mathrm{u}}_{\zeta}} 
 (\theta_{\zeta} \delta_{\zeta}^{-1} ) 
 =  \delta_{\zeta}^{-1}$. 
\end{lem}
\begin{proof}
We have 
$\Nr_{N'^{\mathrm{u}}_{\zeta}/M'^{\mathrm{u}}_{\zeta}} 
 (\theta_{\zeta})=\delta_{\zeta}^3$
 by $\theta_{\zeta}^2 -\theta_{\zeta}=
 \delta_{\zeta} \eta_{\zeta}$ and 
$\eta_{\zeta}^2 -\eta_{\zeta} =\delta_{\zeta}$. The claim follows from this. 
\end{proof}

Let $\sigma_0 \in \Gal (N'^{\mathrm{u}}_{\zeta}/M'^{\mathrm{u}}_{\zeta})$ 
be a generator of 
$\Gal (N'^{\mathrm{u}}_{\zeta}/M'^{\mathrm{u}}_{\zeta})$ 
determined by 
$\sigma_0 (\eta_{\zeta}) -\eta_{\zeta}=1$ 
and 
$\sigma_0 (\theta_{\zeta}) -\theta_{\zeta} =
 \eta_{\zeta}$. 
\begin{lem}\label{indeph}
Let 
$\iota_{n,\zeta} \colon \Gal (N'^{\mathrm{u}}_{\zeta}/M'^{\mathrm{u}}_{\zeta}) \to \bC^{\times}$ be the homomorphism induced by 
$\xi'_{n,\zeta}$ (\cf Lemma \ref{lem:factor}). 
Then we have 
$\iota_{n,\zeta} (\sigma_0) =-\sqrt{-1}$. 
\end{lem}
\begin{proof}
Let $s,t$ be as in \eqref{defst}.
We take $\sigma \in I_{M'^{\mathrm{u}}_{\zeta}}$ such that 
$\Theta_{\zeta} (\sigma)=\bigl( 
 (1,t,s^2) ,0 \bigr)$. 
Recall that 
\[
\phi'\left((1,t,s^2)\right)=\bar{g}\left(1,s^2+\sum_{0 \leq i<j \leq e-1}t^{2^i+2^j}\right) \in R''_0
\] 
is a generator. 
Then it suffices to show that 
$\sigma (\eta_{\zeta}) -\eta_{\zeta}=1$ 
and 
$\sigma (\theta_{\zeta}) -\theta_{\zeta} =
 \eta_{\zeta}$. 
We can check the first equality easily. 
To show the second equality, 
it suffices to show that 
$\sigma (\theta'_{\zeta}) -\theta'_{\zeta} =
 \eta_{\zeta}^{2^{e-1}}$. 
By \eqref{abcdef}, we have
\[
\sigma(\gamma_{\zeta})-\gamma_{\zeta}
 \equiv s^2+\sum_{i=0}^{e-1}(t \beta_{\zeta}+t^2)^{2^i}
\mod \mathfrak{p}_{N'^{\mathrm{u}}_{\zeta}}. 
\]
By $t=\sigma(\beta_{\zeta})-\beta_{\zeta}$, $\Tr_{\mathbb{F}_{2^e}/\mathbb{F}_2}(t)=1$ 
and \eqref{defetagam'}, 
we have 
\[
 \sigma (\gamma_{\zeta}') -\gamma_{\zeta}'=
 \eta_{\zeta} -b_1 
 +s^2 
 +\sum_{0 \leq i \leq j \leq e-1} t^{2^i +2^j} \mod \mathfrak{p}_{N'^{\mathrm{u}}_{\zeta}}. 
\]
Hence, by \eqref{deftheta'} and \eqref{b0td}, 
we have 
\[
\sigma (\theta'_{\zeta}) -\theta'_{\zeta}
\equiv 
\sum_{i=0}^{2^{e-1}} d_i \eta_{\zeta}^i 
\mod \mathfrak{p}_{N'^{\mathrm{u}}_{\zeta}} 
\]
with some $d_i \in k^{\mathrm{ac}}$. 
By \eqref{b0td}, we have 
\[
\sum_{i=0}^{e-1}(t(\sigma(\gamma'_{\zeta})-\gamma'_{\zeta}))^{2^i}
=\sigma(\gamma'_{\zeta})-\gamma'_{\zeta}
+\sum_{1 \leq i \leq j \leq e-1} t^{2^i+2^j}
+\sum_{0 \leq i<j \leq e-1} t^{2^j} (\delta_{\zeta}-s)^{2^i}. 
\]
Therefore, again 
by \eqref{deftheta'} and \eqref{b0td}, we have 
\[
 d_0 = b_1 +s^2 
 +\sum_{0 \leq i \leq j \leq e-1} t^{2^i +2^j} 
 +\sum_{1 \leq i \leq j \leq e-1} t^{2^i +2^j}
 +b_1^2 =s+s^2 +t=0. 
\]
This implies 
$\sigma (\theta'_{\zeta}) -\theta'_{\zeta} =
 \eta_{\zeta}^{2^{e-1}}$, 
since we know that 
$\sigma (\theta'_{\zeta}) - \theta'_{\zeta} 
 -\eta_{\zeta}^{2^{e-1}} \in \bF_2$ 
by Lemma \ref{tileq} and 
$\sigma (\eta_{\zeta}) -\eta_{\zeta}=1$. 
\end{proof}

\begin{lem}\label{epsxipos'}
We have
\[
 \varepsilon 
 (\xi'_{n,\zeta},\psi_{M'^{\mathrm{u}}_{\zeta}} )
 = 
 \begin{cases}
 \frac{1+\sqrt{-1}}{\sqrt{2}} & 
 \textrm{if $e = 1$,}\\ 
 \frac{1-\sqrt{-1}}{\sqrt{2}} 
 & \textrm{if $e \geq 2$.} 
 \end{cases} 
\] 
\end{lem}
\begin{proof}
By Proposition \ref{chare}, 
\eqref{xipsi}, 
Lemma \ref{Treta}, 
Lemma \ref{Trpe} and Lemma \ref{Nrte}, 
we have 
\begin{align}
 \varepsilon (\xi'_{n,\zeta},\psi_{M'^{\mathrm{u}}_{\zeta}} ) &= 
 2^{-\frac{1}{2}} 
 \sum_{x \in \bF_2} 
 \xi'_{n,\zeta} \bigl( 
 \delta_{\zeta}^{2^e +1} (1+x\delta_{\zeta}^{-1} ) \bigr)^{-1} 
 \psi_{M'^{\mathrm{u}}_{\zeta}} \bigl( 
 \delta_{\zeta}^{2^e +1} (1+x\delta_{\zeta}^{-1} ) \bigr) \notag \\
 &=
 \begin{cases}
 2^{-\frac{1}{2}} 
 \bigl( 1 - \xi'_{n,\zeta} ( 1+\delta_{\zeta}^{-1} )^{-1} \bigr) 
 &\textrm{if $e=1$,}\\ 
 2^{-\frac{1}{2}} 
 \bigl( 1 + \xi'_{n,\zeta} ( 1+\delta_{\zeta}^{-1} )^{-1} \bigr) 
 &\textrm{if $e \geq 2$.} 
 \end{cases}\label{epsexp}
\end{align}
First assume that $e=1$. 
Then we know the 
equality in the claim 
modulo $\mu_{2} (\bC)$ 
by Lemma \ref{epsmod}. 
Hence it suffices to show the 
equality of the real parts. 
This follows from \eqref{epsexp}. 
In particular, we have 
$\xi'_{2,\zeta} ( 1+\delta_{\zeta}^{-1} )=\sqrt{-1}$. 

Next, we consider the general case. 
We put 
$\alpha_1'=1/(\delta_{\zeta}^2-\delta_{\zeta}+1)$ 
and $\varpi'=\alpha_1'^3$. 
Let $\xi'_{2,1,\zeta}$ denote 
$\xi'_{2,1}$ 
in the case where 
$K$ and $\varpi$ are replaced by 
$\bF_2((\varpi'))$ and $\varpi'$. 
By applying Lemma \ref{indeph} to 
$\xi'_{n,\zeta}$ and 
$\xi'_{2,1,\zeta}$, 
we have $\xi'_{n,\zeta}=\xi'_{2,1,\zeta}$. 
We know that 
$\xi'_{2,1,\zeta} ( 1+\delta_{\zeta}^{-1} )=\sqrt{-1}$ 
by the result in the case $e=1$. 
Hence, we have 
$\xi'_{n,\zeta}( 1+\delta_{\zeta}^{-1} )=\sqrt{-1}$, 
which shows the claim. 
\end{proof}

\begin{lem}\label{epsxipos}
We have
\[
 \varepsilon 
 (\xi_{n,\zeta},\psi_{M'^{\mathrm{u}}_{\zeta}} )
 = (-1)^{1+\epsilon_0} . 
\] 
\end{lem}
\begin{proof}
The epsilon factor $\varepsilon 
 (\xi_{n,\zeta},\psi_{M'^{\mathrm{u}}_{\zeta}} )$ 
equals 
$\varepsilon 
 (\xi'_{n,\zeta},\psi_{M'^{\mathrm{u}}_{\zeta}} )$ 
times 
\[
\begin{cases}
\left(\frac{1+\sqrt{-1}}{\sqrt{2}}\right)^{-3(2^e+1)} & \textrm{if $e \neq 2$}, \\
-\left(\frac{1+\sqrt{-1}}{\sqrt{2}}\right)^{-3(2^e+1)} & \textrm{if $e=2$}
\end{cases}
\]
by Lemma \ref{rswW}, \eqref{dya}, Lemma \ref{rswxi}. 
Hence, the claim follows from Lemma \ref{epsxipos'}. 
\end{proof}

\appendix
\section{Realization in cohomology of Artin--Schreier variety}\label{RealAS}
We realize $\tau_n$ in the cohomology of an Artin--Schreier variety. 
Let $\nu_{n-2}$ be the 
quadratic form on 
$\mathbb{A}_{k^{\rmac}}^{n-2}$ defined by 
\[
 \nu_{n-2} 
 \bigl( (y_i)_{1 \leq i \leq n-2} \bigr) 
 = -\frac{1}{n'} 
 \sum_{1 \leq i \leq j \leq n-2} y_i y_j . 
\]
Let $X$ be the smooth affine variety over $k^{\mathrm{ac}}$ 
defined by 
\begin{equation*}
 x^p-x=y^{p^e+1} +\nu_{n-2} 
 \bigl( (y_i)_{1 \leq i \leq n-2} \bigr) 
 \ \ \textrm{in}\ \mathbb{A}_{k^{\mathrm{ac}}}^n . 
\end{equation*}
We define a right action of 
$Q \rtimes \mathbb{Z}$ on $X$ by 
\begin{equation*}
\begin{split}
 (x,y,(y_i)_{1 \leq i \leq n-2} )((a,b,c),0) &= 
 \biggl( x + \sum_{i=0}^{e -1} 
 ( b y )^{p^i} + c , 
 a (y + b^{p^e} ) , (a^{\frac{p^e +1}{2}} y_i)_{1 \leq i \leq n-2} \biggr), \\ 
 (x,y,(y_i)_{1 \leq i \leq n-2} )\Fr (1) &= 
 (x^{p},y^{p},(y_i^{p} )_{1 \leq i \leq n-2} ). 
\end{split}
\end{equation*}
We consider the morphism 
\[
 \pi_{n-2} \colon \mathbb{A}_{k^{\mathrm{ac}}}^{n-1} 
 \to \mathbb{A}_{k^{\mathrm{ac}}}^1 ;\ 
 (y,(y_i)_{1 \leq i \leq n-2} ) 
 \mapsto y^{p^e +1} 
 + \nu_{n-2} \bigl( (y_i)_{1 \leq i \leq n-2} \bigr) . 
\]
Then we have a decomposition 
\begin{equation}
 H_{\mathrm{c}}^{n-1} ( X , 
 \ol{\bQ}_{\ell} ) \cong  
 \bigoplus_{\psi \in \bF_p^{\vee} \setminus \{ 1 \} } 
 H_{\mathrm{c}}^{n-1} (\mathbb{A}_{k^{\mathrm{ac}}}^{n-1} ,
 \pi_{n-2}^* \mathcal{L}_{\psi}) 
\end{equation}
as $Q \rtimes \mathbb{Z}$ representations. 
Let $\rho_n$ be the representation over $\bC$ of 
$Q \rtimes \mathbb{Z}$ defined by 
\[
 H_{\mathrm{c}}^{n-1} (\mathbb{A}_{k^{\mathrm{ac}}}^{n-1} ,
 \pi_{n-2}^* \mathcal{L}_{\psi_0}) 
 \Bigl( \frac{n-1}{2} \Bigr) 
\] 
and $\iota$, 
where $\bigl( \frac{n-1}{2} \bigr)$ means the twist by 
the character 
$((a,b,c),m) \mapsto p^{m(n-1)/2}$. 

\begin{lem}\label{detnu}
If $p \neq 2$, then we have 
$\det \nu_{n-2} =
 -(-2n')^n \in \bF_p^{\times}/(\bF_p^{\times})^2$. 
\end{lem}
\begin{proof}
This is an easy calculation. 
\end{proof}

\begin{prop}
We have $\tau_n \simeq \rho_n$.  
\end{prop}
\begin{proof}
Let $Y$ be the smooth affine variety over $k^{\mathrm{ac}}$ 
defined by 
\begin{equation*}
 x^{p}-x=  \nu_{n-2} \bigl( (y_i)_{1 \leq i \leq n-2} \bigr) 
 \ \ \textrm{in}\ \mathbb{A}_{k^{\mathrm{ac}}}^{n-1}. 
\end{equation*}
We define a right action of 
$Q \rtimes \mathbb{Z}$ on $Y$ by 
\begin{equation*}
\begin{split}
 \bigl( x,(y_i)_{1 \leq i \leq n-2} \bigr)
 \bigl( (a,b,c),0 \bigr) &= 
 \bigl( x , 
 (a^{\frac{p^e +1}{2}} y_i)_{1 \leq i \leq n-2} \bigr), \\ 
 \bigl( x,(y_i)_{1 \leq i \leq n-2} \bigr) \Fr (1) &= 
 \bigl( x^{p},(y_i^{p} )_{1 \leq i \leq n-2} \bigr) . 
\end{split}
\end{equation*}
Using the action of $Q \rtimes \mathbb{Z}$ on $Y$, 
we can define an action of 
$Q \rtimes \mathbb{Z}$ on 
$H_{\mathrm{c}}^{n-2} (\bA_{k^{\rmac}}^{n-2} ,
 \nu_{n-2}^* \mathcal{L}_{\psi_0})$. 
Then we have 
\begin{equation}\label{Kuiso}
 H_{\mathrm{c}}^{n-1} (\bA_{k^{\rmac}}^{n-1} ,
 \pi_{n-2}^* \mathcal{L}_{\psi_0})
 \cong 
 H_{\mathrm{c}}^{1} (\bA_{k^{\rmac}}^{1} ,
 \pi^* \mathcal{L}_{\psi_0})
 \otimes 
 H_{\mathrm{c}}^{n-2} (\bA_{k^{\rmac}}^{n-2} ,
 \nu_{n-2}^* \mathcal{L}_{\psi_0})
\end{equation}
by the K\"{u}nneth formula, 
where the isomorphism is compatible with 
the actions of $Q \rtimes \mathbb{Z}$. 
By \eqref{Kuiso}, 
it suffices to show the action of 
$Q \rtimes \mathbb{Z}$ on 
\begin{equation}\label{geori}
 H_{\mathrm{c}}^{n-2} (\bA_{k^{\rmac}}^{n-2} ,
 \nu_{n-2}^* \mathcal{L}_{\psi_0}) 
 \Bigl( \frac{n-1}{2} \Bigr) 
\end{equation}
is equal to 
the character \eqref{chartwi} via $\iota$. 

First, consider the case where $p \neq 2$. 
The equality of the actions of $Q$ follows from 
\cite[Lemma 2.2.3]{DenLoCh}. 
We have 
\begin{gather}\label{gsumy}
\begin{aligned}
 (-1)^{n-2} \sum_{\bfy \in \bF_p^{n-2}} 
 \psi_0 \bigl( 
 \nu_{n-2} ( \bfy ) \bigr) 
 &= \Bigl( \frac{-1}{p} \Bigr) 
 \Bigl( -\Bigl( \frac{-2n'}{p} \Bigr) \Bigr)^n 
 (\epsilon(p)\sqrt{p})^{n-2} \\
 &=\left(-\epsilon(p)\left(\frac{-2n'}{p}\right)\right)^n\sqrt{p}^{n-2}
\end{aligned}
\end{gather}
by Lemma \ref{detnu}. 
The equality of the actions of 
$\Fr (1) \in Q \rtimes \bZ$ follows from 
\cite[Sommes trig. Scholie 1.9]{DelCoet} and 
\eqref{gsumy}. 

If $p=2$, the equality follows from 
\cite[Proposition 4.5]{ITsimptame} and 
$\bigl( \frac{2}{n-1} \bigr) =(-1)^{n(n-2)/8}$. 
\end{proof}

\noindent
Naoki Imai\\ 
Graduate School of Mathematical Sciences, 
The University of Tokyo, 3-8-1 Komaba, Meguro-ku, 
Tokyo, 153-8914, Japan\\ 
naoki@ms.u-tokyo.ac.jp \\

\noindent
Takahiro Tsushima\\ 
Department of Mathematics and Informatics, 
Faculty of Science, Chiba University, 
1-33 Yayoi-cho, Inage, Chiba, 263-8522, Japan\\
tsushima@math.s.chiba-u.ac.jp

\end{document}